\documentclass[11pt]{article}
\usepackage[utf8]{inputenc}
\usepackage{amsthm,amsmath,amssymb,bbm,geometry,epsfig,hyperref,enumerate,comment,nicefrac,listings,graphicx,color,titling,makecell}
\usepackage[capitalize]{cleveref}
\usepackage{datetime}

\newtheorem{lemma}{Lemma}
\newtheorem{proposition}[lemma]{Proposition}
\newtheorem{theorem}[lemma]{Theorem}

 \newcommand{\be}{\begin{equation}}
 \newcommand{\ee}{\end{equation}}
 \newcommand{\bea}{\begin{eqnarray}}
 \newcommand{\eea}{\end{eqnarray}}
 \newcommand{\beas}{\begin{eqnarray*}}
\newcommand{\eeas}{\end{eqnarray*}}

\newcommand{\crl}[1]{\ensuremath{ \left\{ #1 \right\} }}
\newcommand{\edg}[1]{\ensuremath{\! \left[ #1 \right] }}
\newcommand{\brak}[1]{\ensuremath{\left( #1 \right)}}

\newcommand{\abs}[1]{\ensuremath{ \left| #1 \right| }}

\providecommand{\N}{{\ensuremath{\mathbb{N}}}}

\providecommand{\R}{{\ensuremath{\mathbb{R}}}}
\providecommand{\B}{\mathcal{B}}

\renewcommand{\P}{\mathbbm{P}}

\providecommand{\cR}{\mathcal{R}}

\providecommand{\fu}{\mathfrak{u}}

\providecommand{\fb}{\mathfrak{b}}

\providecommand{\cY}{R}

\providecommand{\E}{{\ensuremath{\mathbbm{E}}}}

\providecommand{\N}{{\ensuremath{\mathbbm{N}}}}

\providecommand{\R}{{\ensuremath{\mathbbm{R}}}}

\providecommand{\E}{{\ensuremath{\mathbb{E}}}}

\newcommand{\Borel}{\mathcal{B}}
\newcommand{\F}{{\ensuremath{\mathcal{F}}}}

\newcommand{\Rd}{\mathbb{R}^d}

\newcommand{\cL}{{\ensuremath{\mathcal{L}}}}

\begin{document}
\setlength{\droptitle}{-4em} 
\title{
An efficient Monte Carlo scheme for Zakai equations}
\author{Christian Beck$^{1,3}$, Sebastian Becker$^{1}$, 
	Patrick Cheridito$^1$, \\
	Arnulf Jentzen$^{2,3}$ and Ariel Neufeld$^4$
	\bigskip
	\\
	\small{$^1$ Department of Mathematics, ETH Zurich, Switzerland}
	\smallskip
	\\
	\small{$^2$ School of Data Science and Shenzhen Research Institute of Big Data,}\\
	\small{The Chinese University of Hong Kong, Shenzhen, China}
	\smallskip
	\\
	\small{$^3$ Applied Mathematics: Institute for Analysis and Numerics,}\\
	\small{Faculty of Mathematics and Computer Science, University of M\"unster, Germany}
	\smallskip
	\\
	\small{$^4$ Division of Mathematical Sciences, School of Physical and Mathematical Sciences,}\\
	\small{Nanyang Technological University, Singapore}
		}
\date{}
\maketitle

\begin{abstract}
In this paper we develop a numerical method for efficiently approximating solutions 
of certain Zakai equations in high dimensions. 
The key idea is to transform a given Zakai SPDE
into a PDE with random coefficients. We show that under suitable regularity 
assumptions on the coefficients of the Zakai equation, the corresponding random PDE
admits a solution random field which, for almost all realizations of the
random coefficients, can be written as a classical solution of a linear parabolic PDE. 
This makes it possible to apply the Feynman--Kac formula to obtain 
an efficient Monte Carlo scheme for 
computing approximate solutions of Zakai equations. The approach achieves good 
results in up to 25 dimensions with fast run times.
\end{abstract}
\vspace{0.2cm}

\section{Introduction}
The goal of stochastic filtering is to estimate the conditional distribution of a not directly 
observable stochastic process blurred by measurement noise. The process of interest is usually called 
\textit{signal process}, while the observed process is referred to as 
\textit{observation process}. Whereas the signal process 
follows a hidden dynamic, probing the system only reveals the observation 
process, which, in general, might depend nonlinearly on the signal process and, in addition, is
blurred by measurement noise. Stochastic filtering problems were first studied in connection with tracking and signal processing (see the seminal works by Kalman~\cite{Kalman} and Kalman \& Bucy~\cite{KalmanBucy}) but soon turned out to also be relevant in a variety of other applications in
finance, the natural sciences and engineering. Among others, nonlinear filtering problems 
naturally arise in e.g.,
financial engineering (\cite{BrigoHanzon98,CeciColaneri17,CoculescuGemanJeanblanc08,DuffieLando01,FreyRunggaldier10,FreySchmidt12}), weather forecasting 
(\cite{buehner2017ensemble,cassola2012wind,che2016wind,duc2015ensemble,falissard2013genuinely,pelosi2017adaptive}) or chemical engineering
(\cite{BudmanHolcombMorari91,ChenSun91,Rutzler87,SeinfeldGavalasHwang71,SolimanRay79b,WindesCinarRay89}). For further applications of nonlinear filtering, we refer to the survey paper \cite{SurveyNonlinFiltering}.
Stochastic filtering problems are naturally related to stochastic partial differential 
equations (SPDEs) since in continuous time, the (unnormalized) density of the 
unobserved signal process given the observations is described by a 
suitable SPDE, such as the Zakai equation \cite{zakai1969optimal} or 
Kushner equation \cite{kushner1964differential}. The SPDEs arising in this 
context can typically not be solved explicitly but instead, have to be computed numerically.
Moreover, they often are high-dimensional as the number of dimensions corresponds to the state space 
dimension of the filtering problem. 

In this paper, we focus on Zakai equations with coefficients that satisfy
certain regularity conditions. 
Let us assume the signal follows the $d$-dimensional dynamics 
\[
Y_t = Y_0 + \int_0^t \mu(Y_s)\,ds + \sigma W_t
\] 
for a $d$-dimensional random vector $Y_0$ with density
$\varphi \colon \R^d \to [0,\infty)$, a sufficiently regular function 
$\mu \colon \R^d \to \R^d$, a constant $d \times d$-matrix
$\sigma$ and a $d$-dimensional Brownian motion $(W_t)_{t \in [0,T]}$
independent of $Y_0$, while we observe a $k$-dimensional process of the form
\[
Z_t = \int_0^t h(Y_s) ds + V_t
\]
for a sufficiently regular function 
$h \colon \R^d \to \R^k$ and a $k$-dimensional Brownian motion 
$(V_t)_{t \in [0,T]}$ independent of $Y_0$ and $(W_t)_{t \in [0,T]}$. 
Then the solution of the corresponding Zakai equation
\begin{align}\label{eq:Zakai}
&X_t(x) 
= \varphi(x) +
	\int_0^t \bigl[\tfrac12 \operatorname{Trace}_{\R^d}\bigl( \sigma \sigma^{T} 
	\operatorname{Hess}(X_s)(x) \bigr) - \operatorname{div} (\mu X_s)(x) \bigr] \,ds 
	+ \int_0^t X_s( x ) \langle h(x), dZ_s\rangle_{\R^d}
\end{align}
describes the evolution of an unnormalized density of the conditional distribution of $Y_t$
given observations of $Z_s$, $s \le t$; that is,
\[
\P \edg{Y_t \in A \mid Z_s, s \in [0,t]}
= \frac{\int_A X_t(x)\,dx}{\int_{\R^d} X_t(x)\,dx} 
\quad \mbox{for every Borel subset } A \subseteq \R^d.
\]
Our numerical method is based on a transformation which transforms a 
Zakai SPDE of the form \eqref{eq:Zakai} into a 
PDE with random coefficients. We show that under suitable conditions on 
the coefficients of the Zakai SPDE, the solution 
of the resulting random PDE is $\omega$-wise 
a classical solution of a linear parabolic PDE. This makes it possible 
to apply the Feynman--Kac formula to obtain an efficient Monte Carlo scheme 
for the numerical approximation of solutions of high-dimensional Zakai equations.
The following is this paper's main theoretical result.

\begin{theorem} \label{thm:main}
	Let $T \in (0, \infty)$, $d,k  \in \N$, $\sigma \in \R^{d\times d}$, and consider 
	functions $\varphi \in C^2(\R^d, [0, \infty))$, $\mu \in C^3(\R^d, \R^d)$ and
	$h \in C^4(\R^d, \R^k)$ such that $\varphi$ has at most polynomially 
	growing derivatives up to the second order, $\mu$ has bounded derivatives 
	up to the third order and $h$ has bounded derivatives up to the fourth order.
	Let $(\Omega, \F, (\F_t)_{t \in [0, T]}, \P)$ be a filtered probability space satisfying 
	the usual conditions\footnote{A filtered probability space $(\Omega, \F, (\F_t)_{t \in [0, T]}, \P)$ 
	is said to satisfy the \emph{usual conditions} if for all $t \in [0, T)$, one has
	$\bigcup_{A \in \F, \P(A) = 0} \{B \subseteq \Omega \colon B \subseteq A\} \subseteq \F_t =  \bigcap_{s \in (t, T]} \F_s$.} which supports standard $(\F_t)_{t\in [0, T]}$-Brownian motions
	$W,U \colon [0, T] \times \Omega \to \R^d$ and 
	$V \colon [0, T] \times \Omega \to \R^k$ with continuous sample paths
	such that $W$ and $V$ are independent. Let
	$Y \colon [0, T] \times \Omega \to \R^d$ and $Z \colon [0, T] \times \Omega \to \R^k$
	be $(\F_t)_{t \in [0, T]}$-adapted stochastic processes such that $\P(Y_0 \in A) = \int_A \varphi(x)\,dx$ for every Borel subset 
	$A \subseteq \R^d$ and
	\begin{equation} \label{eq:YZ}
	Y_t = Y_0 + \int_0^t \mu(Y_s)\,ds + \sigma W_t,
	\quad
	Z_t = \int_0^t h(Y_s)\,ds + V_t \quad \mbox{for } t \in [0,T].
	\end{equation} 
For all $z \in C([0, T], \R^k)$, $t \in [0, T]$ and $x \in \R^d$, let 
$R^{z,t,x} \colon [0, t] \times \Omega \to \R^d$ be an $(\F_s)_{s \in [0, t]}$-adapted 
stochastic processes satisfying\footnote{By $D h(x)$ we denote the Jacobian matrix
$\brak{\tfrac{\partial}{\partial x_j} h_i(x)}_{1 \le i \le k, 1 \le j \le d}$, $x \in \R^d$.}
\begin{equation} \label{eq:Zakai:sde_for_mc_R}
R^{z, t, x}_s = 
	x + \int_0^s \edg{ \sigma \sigma^{T}[D h(R^{v, t, x}_r)]^{T} z(t-r)
	- \mu(R^{v,t,x}_r)} dr
	+ \sigma U_s \quad \mbox{for all } s \in [0,t].
	\end{equation} 
Moreover, let for all $z \in C([0, T], \R^k)$, the functions 
$B_z, u_z  \colon [0, T] \times\R^d \to \R$
be given by\footnote{For $d \in \N$, we denote by $\langle\cdot, \cdot \rangle_{\R^d}\colon \R^d \times \R^d \to \R$ the standard scalar product given by 
	$\langle x, y \rangle_{\R^d} = \sum_{i=1}^d x_i y_i$ and by 
	$\|.\|_{\R^d} \colon \R^d \to [0,\infty)$ the corresponding norm 
	$\|x \|_{\R^d} = \sqrt{\langle x, x \rangle}$ .}
	\be \label{eq:mathfrak-B}
	\begin{split}
	B_z(t, x) =
	&
	\tfrac12 \left\langle \sigma^{T} [D h(x)]^{T} z(t), 
	\sigma^T [D h(x)]^{T} z(t)\right\rangle_{\R^d} 
	- \tfrac12 \left\langle h(x), h(x) \right\rangle_{\R^k} \\ & +
	\tfrac12 \operatorname{Trace}_{\R^d}\bigl( \sigma \sigma^{T} 
	\operatorname{Hess}_x(\langle h(x), z(t) \rangle_{\R^k}) \bigr) 
	 - \langle \mu(x), [D h(x)]^{T} z(t) \rangle_{\R^d}
	- \operatorname{div}(\mu)(x) 
	\end{split}
	\ee
	and 
	\be \label{eq:u}
	u_z(t, x) 
	= 
	\E\Bigl[ \varphi(R^{z, t, x}_t) \exp\Bigl( {\textstyle \int\limits_{0}^{t}} 
	B_z(t-s, R^{z, t, x}_s)\,ds \Bigr) \Bigr],
	\quad t \in [0, T], \; x \in \R^d.
	\ee
 Then
	\be \label{eq:Zakai:definition_X}
	X_t(x, \omega) 
	= u_{Z(\omega)}(t, x) \exp\bigl( \langle h(x), Z_t(\omega) \rangle_{\R^k} \bigr),
	\quad \mbox{$t \in [0, T]$, $x \in \R^d$, $\omega \in \Omega$,}
	\ee
	is, up to indistinguishability, the unique random field
	$X \colon [0, T]\times\R^d\times\Omega \to \R$ satisfying the following properties:
	\begin{enumerate}[{\rm (i)}] 
		\item \label{it:Zakai-1}
		for all $t \in [0, T]$ and $x \in \R^d$, the mapping 
		$X_t(x) \colon \Omega \to \R$ is $\F_t$/$\Borel(\R)$-measurable, 
		
		\item \label{it:Zakai-2} for all $\omega \in \Omega$, the mapping 
		$(t, x) \mapsto X_t(x, \omega)$ is in $C^{0, 2}([0, T]\times\R^d, \R)$
		and there exist constants $a(\omega), c(\omega) \ge 0$ such that
		\[
		\sup_{t \in [0,T]} |X_t(x, \omega)| \le a(\omega) e^{c(\omega) \|x\|_{\R^d}}
		\quad \mbox{for all $x \in \R^d$,}
		\]
		
		\item \label{it:Zakai-3} $\mbox{}$\\[-10.8mm]
		\be \label{Xt}
		X_t(x) = \varphi(x) + \int_0^t \Bigl[ \tfrac12 \operatorname{Trace}_{\R^d}
		\bigl( \sigma \sigma^{T} \operatorname{Hess} (X_s)(x) \bigr) 
		- \operatorname{div}\bigl(\mu X_s \bigr)(x)\Bigr]\,ds + \int_0^t X_s(x)\,\langle h(x), dZ_s\rangle_{\R^k}
		\ee
	$\P$-a.s. for all $t \in [0,T]$ and $x \in \R^d$.
	\end{enumerate} 
\end{theorem}

Representation \eqref{eq:Zakai:definition_X}
makes it possible to approximate the solution $X_t(x, \omega)$ of the Zakai equation
\eqref{Xt} along a realization of the observation process $(Z_s(\omega))_{s \in [0,t]}$
by averaging over different Monte Carlo simulations of the process 
$R^{Z(\omega), t,x}$ given in \eqref{eq:Zakai:sde_for_mc_R}. 
We provide numerical results for a Zakai equation 
of the form \eqref{Xt} for dimensions $d \in \{1, 2, 5, 10, 20, 25\}$
in \cref{sec:numerics} below. The proof of \cref{thm:main} is given in the 
Appendix. 

The idea of transforming a stochastic differential equation into an ordinary differential equation with 
random coefficients goes back to Doss \cite{doss1977liens} and Sussmann \cite{sussmann1978gap}.
An extension to SPDEs was used by Buckdahn and Ma \cite{buckdahn2001stochastic, buckdahn2001stochastic2} to introduce a notion of stochastic viscosity solution for SPDEs and show existence and uniqueness results as well as connections to backward doubly stochastic differential equations. The same approach was employed by Buckdahn and Ma \cite{buckdahn2007pathwise} and Boufoussi et al.\ \cite{boufoussi2007generalized} to study stochastic viscosity solutions of stochastic Hamilton--Jacobi--Bellman (HJB) equations. In this paper, we analyze the regularity properties of such transformations and use them to develop a Monte Carlo method for approximating solutions of Zakai equations. The numerical results in Section \ref{sec:numerics} below
show that it produces accurate results in high dimensions with fast run times.
For different numerical approximation methods for Zakai equations, see e.g., 
\cite{SPDE,Crisan03,CrisanGainesLyons98,GobetPagesPhamPrintemsMVE05,GobetPagesPhamPrintems06}.

\section{Numerical experiments} \label{sec:numerics}

Together with time-discretization, the trapezoidal rule and Monte Carlo sampling,
\cref{thm:main} can be used to approximate
the solution of a given Zakai equation of the form \eqref{eq:Zakai} 
along a realization of the observation process $Z$.
We illustrate this in the following example: Choose $T, \alpha \in (0, \infty)$, 
$\beta \in \R$, $d \in \N$, and let $\sigma \in \R^{d \times d}$ be given 
by $\sigma_{ij} = d^{- \nicefrac{1}{2}}$ for all $i, j \in \{1, \dots, d\}$. Consider 
a $d$-dimensional signal process with dynamics
\begin{equation} \label{eq:Y}
Y_t = Y_0 + \int_0^t \frac{\beta Y_s}{1 + \lVert Y_s \rVert_{\R^d}^2} \,ds 
+ \sigma W_t, \quad \quad t \in [0,T],
\end{equation}
for an ${\cal F}_0$-measurable random initial condition $Y_0 \colon \Omega \to \R^d$ with density 
\[
\varphi(x) = \bigl( \tfrac{\alpha}{2\pi} \bigr)^{\!\nicefrac{d}{2}} \exp\bigl( -\tfrac{\alpha}{2} \lVert x \rVert_{\R^d}^2 \bigr), \quad x \in \R^d, \] defined on a filtered probability space 
$(\Omega, {\cal F}, ({\cal F}_t)_{t \in [0,T]}, \P)$ satisfying the usual conditions 
and a standard $({\cal F}_t)_{t \in [0,T]}$-Brownian motion 
$W \colon [0, T] \times \Omega \to \R^d$
with continuous sample paths. Assume the observation process is of the form 
\begin{equation} \label{eq:Z}
Z_t = \int_0^t \gamma Y_s\,ds + V_t, \quad t \in [0,T],
	\end{equation} 
for a constant $\gamma \in \R$ and a standard $({\cal F}_t)$-Brownian motion 
$V \colon [0, T] \times \Omega \to \R^d$ with continuous 
sample paths independent of $W$.
Let $U \colon [0, T] \times \Omega \to \R^d$ be another standard Brownian motion 
with continuous sample paths and consider stochastic processes 
$R^{z, t, x} \colon [0, t] \times \Omega \to \R^d$, 
$z \in C([0, T], \R^d)$, $t \in [0, T]$, $x \in \R^d$ satisfying
\[ \label{def:mathcal-Y-Zakai-Original}
	R^{z, t, x}_s 
	= 
	x + \int_0^s \edg{\gamma \sigma \sigma^T z(t-r) - \frac{\beta R^{z, t, x}_r}
	{1 + \lVert R^{z, t, x}_r \rVert_{\R^d}^2}} dr 
	+ \sigma U_s 
	\]
	for all $z \in C([0, T], \R^d)$, $t \in [0, T]$, $s \in [0, t]$ and $x \in \R^d$. 
	Let the mappings $B_{z}, u_{z} \colon [0, T]\times \R^d \to \R$, 
	$z \in C([0, T], \R^d)$ be given by 
\[ \label{eq:mathfrak-B-Original}
\begin{split}
B_z(t, x) = &
\tfrac{\gamma^2}{2} \langle \sigma^T z(t) + x, \sigma^T z(t) - x\rangle_{\R^d} 
	- \beta \gamma (1+\lVert x \rVert_{\R^d}^2)^{-1} \langle x, z(t) \rangle_{\R^d} 
	\\
	& - d \beta (1+\lVert x \rVert_{\R^d}^2)^{-1}
	+ 2 \beta \lVert x \rVert_{\R^d}^2 (1+\lVert x \rVert_{\R^d}^2)^{-2}
	\end{split}
	\]
	and
	\[ \label{nonlinearRPDE-Original}
	\begin{split}
	&u_z(t,x)
	= \E\biggl[
	\varphi \Big(R^{z,t,x}_t \Big) \exp \Bigl(
	{\textstyle\int\limits_0^t} B_{z}(t-s,R^{z,t,x}_s) \,ds	
	\Bigr)
	\biggr],	
	\end{split} 
	\]
	$z \in C([0, T], \R^d)$, $t \in [0,T]$, $x \in \R^d$. By \cref{thm:main}, 
	\be \label{eq:repX}
	X_t(x, \omega) 
	= 
	u_{Z(\omega)}(t, x) \exp\bigl( \langle \gamma x, Z_t(\omega) \rangle_{\R^d} \bigr),
	\quad  t \in [0,T], \, x \in \R^d, \, \omega \in \Omega,
	\ee
is, up to indistinguishability, the unique random field 
$X \colon [0, T] \times \R^d \times \Omega \to \R$ satisfying
\eqref{it:Zakai-1}--\eqref{it:Zakai-2} of \cref{thm:main} and solving the Zakai equation
\be \label{specialZakai}	
\begin{split}	X_t(x) 
		&= \varphi(x)
		+ 
		\int_0^t X_s(x) \langle \gamma x, dZ_s \rangle_{\R^d} \\
		&
		+ 
		\int_0^t \edg{
		\tfrac12 {\textstyle\sum\nolimits_{i,j=1}^d} \tfrac{\partial^2}{\partial x_i \partial x_j} X_s(x) - 
		{\textstyle\sum\nolimits_{i=1}^d} \tfrac{\partial}{\partial x_i} \brak{
		\frac{\beta x_i X_s(x)}{
		1 + \lVert x \rVert_{\R^d}^2} }}\,ds \quad \mbox{$\P$-a.s.,}
		\end{split}
\ee
$t \in [0,T], \, x \in \R^d$, corresponding to the dynamics \eqref{eq:Y}--\eqref{eq:Z}.

We use representation \eqref{eq:repX} to approximate
$X_T(x)$ for a given realization $(z(t))_{t \in [0,T]}$ of $(Z_t)_{t \in [0,T]}$.
For numerical purposes, we generate a discrete realization of the observation process
by choosing an $N \in \mathbb{N}$ and considering 
the following discretized versions of \eqref{eq:Y}--\eqref{eq:Z}:
\[ 
\mathcal{Y}_0 \sim \mathcal{N}\brak{0, \tfrac{1}{\alpha} I_d}, \quad \mathcal{Y}_{n} 
= \mathcal{Y}_{n-1} 
+ \frac{\beta \mathcal{Y}_{n-1}}{ 
1 + \lVert \mathcal{Y}_{n-1} \rVert_{\R^d}^2} \tfrac{T}{N} + \sigma (W_{nT\slash N} - W_{(n-1)T\slash N}), 
	\]
\[
\mathcal{Z}_0 = 0, \quad	\mathcal{Z}_{n} = \mathcal{Z}_{n-1}
+ \gamma \frac{\mathcal{Y}_{n-1} + \mathcal{Y}_n}{2} 
\frac{T}{N} + (V_{nT\slash N} - V_{(n-1)T\slash N}), \; n \in \{1, \dots, N\}.
\]
Let $U^{(i)} \colon [0, T] \times \Omega \to \R^d$, $ i \in \N$, be i.i.d.\ 
standard Brownian motions independent of $\mathcal{Y}_0$, $W$, $V$, and
consider $R^{(x, i)}_n \colon \Omega \to \R^d$, $n \in \{0, 1, \ldots, N\} $, $x \in \R^d$, $i \in \N$,
given by $R^{(i, x)}_{0} = x$ and 
\be \label{Ri}
\begin{split}
R^{(x, i)}_{n} 
& = R^{(x, i)}_{n-1} + \brak{\gamma \sigma\sigma^T \mathcal{Z}_{N-n+1} 
- \frac{\beta R^{(x, i)}_{n-1}}{ 1 + \lVert R^{(x, i)}_{n-1} \rVert_{\R^d}^2}} 
\frac{T}{N} + \sigma \left(U^{(i)}_{nT\slash N} - U^{(i)}_{(n-1)T\slash N} \right)
	\end{split} 
	\ee
for $x \in \R^d$, $i \in \N$ and $n \in \{1, 2, \ldots, N\}$. Define the mappings 
$B_n, \mathcal{X}^M \colon \Omega \times \R^d \to \R$, $n \in \{0, 1, \ldots, N\}$, 
$M \in \N$, by
\[
\begin{split}
	B_n(x) = &
	\tfrac{\gamma^2}{2} \langle \sigma^{T} \mathcal{Z}_n + x, \sigma^T \mathcal{Z}_n - x\rangle_{\R^d} 
	- \beta \gamma (1 + \lVert x \rVert_{\R^d}^2 )^{-1} \langle x, \mathcal{Z}_n \rangle_{\R^d} 
	\\
	& - d \beta (1+\lVert x\rVert_{\R^d}^2)^{-1}
	+ 2\beta \lVert x \rVert_{\R^d}^2 (1+\lVert x\rVert_{\R^d}^2)^{-2}, 
	\quad n \in \{1, \dots, N\}, \, x \in \R^d,
	\end{split}
	\]
	and
\begin{align} \label{vM}
\mathcal{X}^M(x) & = \frac{1}{M} \sum_{i=1}^{M} \varphi \big( R^{(x, i)}_N \big)\\
\notag
& \quad \times \exp \Biggl(
\sum_{n=1}^{N} \frac{ T }{2 N } \Bigl[ B_{N-n}(R^{(x, i)}_n) 
+ B_{N-n+1}(R^{(x, i)}_{n-1}) \Bigr] + \langle \gamma x, \mathcal{Z}_N \rangle_{\R^d})\Biggr), 
\quad M \in \N, \; x \in \R^d.
\end{align}
It follows from the law of large numbers that, for $M \to \infty$,
	\[ \label{eq:numerics:approximation}
	\mathcal{X}^M(x) 
	\stackrel{\mbox{$\P$-a.s.}}{\longrightarrow} \; \E[ \mathcal{X}^1(x) \mid \mathcal{Z}],
	\]
which approximates $X_T(x)$.

\cref{table:1} below shows point estimates and $95\%$ confidence intervals
for $\E[ \mathcal{X}^1(x) \mid \mathcal{Z}]$ for different realizations
of $(\mathcal{Y}, \mathcal{Z})$, $\alpha=2\pi$, $\beta= \nicefrac{1}{4}$, $\gamma= 1$, 
$T=\nicefrac{1}{2}$, $N = 100$ and $x \in \{\mathcal{Y}_N, 2 \mathcal{Z}_N\}$.
For every $d \in \{1, 2, 5, 10, 20, 25\}$ we simulated five realizations of 
$({\cal Y}, {\cal Z})$ and computed estimates of
$\E[ \mathcal{X}^1(x) \mid \mathcal{Z}]$ for $x \in \{\mathcal{Y}_N , 2 \mathcal{Z}_N\}$
by computing realizations of $\mathcal{X}^{M}(x)$, $x \in \{\mathcal{Y}_N, 2 \mathcal{Z}_N \}$, 
for $M = $ 4,096,000.  Note that in a typical application,
the signal $\mathcal{Y}_N$ is not directly observable, while, in view of \eqref{eq:Z},
$2 \mathcal{Z}_N = \mathcal{Z}_N /(\gamma T)$ is a naive estimate of $\mathcal{Y}_N$ 
based on the observation process $\mathcal{Z}$. As expected, with a few 
exceptions, the values ${\cal X}^M(x)$ reported in Table \ref{table:1} are higher 
for $x = {\cal Y}_N$ than for $x = 2 {\cal Z}_N$.
The 95\% confidence intervals were approximated,
using the central limit theorem, with
\[
\left[\mathcal{X}^{M}(x) - \frac{s_M(x)}{\sqrt{M}} \, q_{0.975} \, ,\, \mathcal{X}^{M}(x) 
+ \frac{s_M(x)}{\sqrt{M}} \, q_{0.975} \right], 
\]
where $q_{0.975}$ is the $97.5 \%$-quantile of the standard normal distribution and 
$s^2_M(x)$ the sample variance of \eqref{vM} given by 
\[
\begin{split}
s^2_M(x) & = \frac{1}{M-1} \sum_{i=1}^{M} \bigg\{ \varphi \left( R^{(x, i)}_N \right)\\
& \qquad \times 
\exp \left( \sum_{n=1}^{N} \frac{ T }{2 N } \left[ B_{N-n} \left(R^{(x, i)}_n \right) 
	+ B_{N-n+1} \left(R^{(x, i)}_{n-1} \right) \right] + \langle \gamma x, \mathcal{Z}_N \rangle_{\R^d})
	\right) - \mathcal{X}^M(x) \bigg\}^2.
\end{split}
\]
The reported runtimes are averages of the ten times needed to compute
$\mathcal{X}^M(x)$, $x \in \{\mathcal{Y}_N, 2 \mathcal{Z}_N \}$, for five 
different realizations of $({\cal Y}, {\cal Z})$.
\begin{table}
{\small 
\begin{center}
\begin{tabular}{|c|c|c|c|c|c|c|c|c|}
\hline\\[-2.8mm]	
$ d $ & $X_{\nicefrac{1}{2}}({\cal Y}_N)$ & $95\%$ CI for $X_{\nicefrac{1}{2}}({\cal Y}_N)$ & $X_{\nicefrac{1}{2}}(2{\cal Z}_N)$ & $95\%$ CI for $X_{\nicefrac{1}{2}}(2{\cal Z}_N)$ & 
\makecell{Avg.\\ run\\ time}\\[1mm]
\hline
 & 0.6238245 & [0.6230741, 0.6245749] & 0.0000552 & [0.0000471, 0.0000632] &\\ 
 & 0.4336298 & [0.4333329, 0.4339267] & 0.2250949 & [0.2247668, 0.2254231] &\\ 
1 & 0.1863724 & [0.1851538, 0.1875910] & 0.0000330 & [0.0000271, 0.0000390] & 8.1s \\ 
 & 0.4666271 & [0.4662590, 0.4669951] & 0.3562083 & [0.3558604, 0.3565562] &\\ 
 & 0.3245020 & [0.3241600, 0.3248440] & 0.4065303 & [0.4062038, 0.4068568] &\\
 \hline 
 & 0.1076552 & [0.1075091, 0.1078014] & 0.1004743 & [0.1003647, 0.1005839] &\\ 
 & 0.3076132 & [0.3073077, 0.3079187] & 0.0656050 & [0.0654037, 0.0658064] &\\ 
2 & 0.0251895 & [0.0251477, 0.0252314] & 0.0532161 & [0.0531611, 0.0532710] & 8.4s \\ 
 & 0.1698629 & [0.1696636, 0.1700622] & 0.0000018 & [0.0000018, 0.0000018] &\\ 
 & 0.2222962 & [0.2220877, 0.2225048] & 0.2758624 & [0.2755767, 0.2761480] & \\ 
 \hline
 & 0.0365342 & [0.0364497, 0.0366187] & 0.0000000 & [0.0000000, 0.0000000] &\\ 
 & 0.0450055 & [0.0449406, 0.0450704] & 0.0000366 & [0.0000366, 0.0000367] &\\ 
5 & 0.0128435 & [0.0128251, 0.0128619] & 0.0000000 & [0.0000000, 0.0000000] & 9.4s \\ 
 & 0.2135126 & [0.2121792, 0.2148459] & 0.0000000 & [0.0000000, 0.0000000] &\\ 
 & 0.0581450 & [0.0580636, 0.0582264] & 0.0000001 & [0.0000001, 0.0000001] & \\ 
 \hline
 & 0.0071417 & [0.0071233, 0.0071601] & 0.0000000 & [0.0000000, 0.0000000] & \\ 
 & 0.0053026 & [0.0052896, 0.0053156] & 0.0000000 & [0.0000000, 0.0000000] &\\ 
10 & 0.0009165 & [0.0009142, 0.0009189] & 0.0000000 & [0.0000000, 0.0000000] & 9.3s \\ 
 & 0.0006419 & [0.0006407, 0.0006430] & 0.0000000 & [0.0000000, 0.0000000] &\\ 
 & 0.0043062 & [0.0042948, 0.0043176] & 0.0000000 & [0.0000000, 0.0000000] &\\ 
 \hline
 & 0.0000355 & [0.0000355, 0.0000356] & 0.0000000 & [0.0000000, 0.0000000] &\\ 
 & 0.0001140 & [0.0001137, 0.0001142] & 0.0000000 & [0.0000000, 0.0000000] &\\ 
20 & 0.0004798 & [0.0004777, 0.0004819] & 0.0000000 & [0.0000000, 0.0000000] &9.2s \\ 
 & 0.0049785 & [0.0049678, 0.0049893] & 0.0000000 & [0.0000000, 0.0000000] &\\ 
 & 0.0018053 & [0.0017985, 0.0018122] & 0.0000000 & [0.0000000, 0.0000000] &\\ 
 \hline
 & 0.0000179 & [0.0000177, 0.0000182] & 0.0000000 & [0.0000000, 0.0000000] &\\ 
 & 0.0000457 & [0.0000452, 0.0000461] & 0.0000000 & [0.0000000, 0.0000000] &\\ 
25 & 0.0000011 & [0.0000011, 0.0000011] & 0.0000000 & [0.0000000, 0.0000000] & 9.8s \\ 
 & 0.0000047 & [0.0000047, 0.0000048] & 0.0000000 & [0.0000000, 0.0000000] & \\ 
 & 0.0001012 & [0.0000992, 0.0001032] & 0.0000000 & [0.0000000, 0.0000000] & \\ 
 \hline
\end{tabular}
		\caption{Estimates of $X_{\nicefrac{1}{2}}(\mathcal{Y}_N)$ and 
		$X_{\nicefrac{1}{2}}(2 \mathcal{Z}_N)$ together with 95\% confidence intervals
		for the Zakai equation \eqref{specialZakai} for five different realizations of 
		$(\mathcal{Y}, \mathcal{Z})$ in each of the cases $d \in \{1, 2, 5, 10, 20, 25\}$}
		\label{table:1}
	\end{center}}
\end{table}

To approximate the whole function $x \mapsto X_T(x)$, our algorithm can be 
run simultaneously for different $x \in \R^d$. To produce the plots in Figures
\ref{fig:1}--\ref{fig:3}, we divided the time interval into $N = 20$ subintervals 
and computed $\mathcal{X}^{M}(x)$ for a given realization of $({\cal Y}, {\cal Z})$
and different $x \in \R^d$ based on $M = $ 102,400 independent copies of \eqref{Ri},
which we generated simultaneously for different $x \in \R^d$ using 
the same simulated Brownian increments.
The first plot in Figure \ref{fig:1} shows ${\cal X}^M(x)$ for
$x$ on a regular grid with 1024 grid points in the interval 
$[\mathcal{Y}_N-5, \mathcal{Y}_N + 5]$, while
the second plot in Figure \ref{fig:1} shows ${\cal X}^M(x)$ for
$x$ on a regular grid with $128^2$ grid points in the square 
$[\mathcal{Y}_{N,1}-5, \mathcal{Y}_{N,1} + 5] 
\times [\mathcal{Y}_{N,2} -5, \mathcal{Y}_{N,2} + 5]$.
Figures \ref{fig:2}--\ref{fig:3} show ${\cal X}^M(x, \mathcal{Y}_{N,2}, \dots, \mathcal{Y}_{N,d})$ 
for $x$ on a regular grid with 1024 grid points in the interval 
$[\mathcal{Y}_{N,1}-5, \mathcal{Y}_{N,1} + 5]$
for $d \in \{5, 10, 20, 25\}$.
The computation times for the results depicted in Figures 
\ref{fig:1}--\ref{fig:3} were 1.1s, 28.4s, 7.4s, 14.5s, 28.1s,
34.1s, respectively.

The numerical experiments presented in this section were implemented in Python using TensorFlow on a NVIDIA GeForce RTX 2080 Ti GPU. The underlying system was an AMD Ryzen 9 3950X CPU with 64 GB DDR4 memory running Tensorflow 2.1 on Ubuntu 19.10. The Python source codes  
can be found in the GitHub repository \url{https://github.com/seb-becker/zakai}.

\begin{figure} 
\includegraphics[width=0.45\textwidth, trim=0mm 0mm 0mm 0mm, clip]{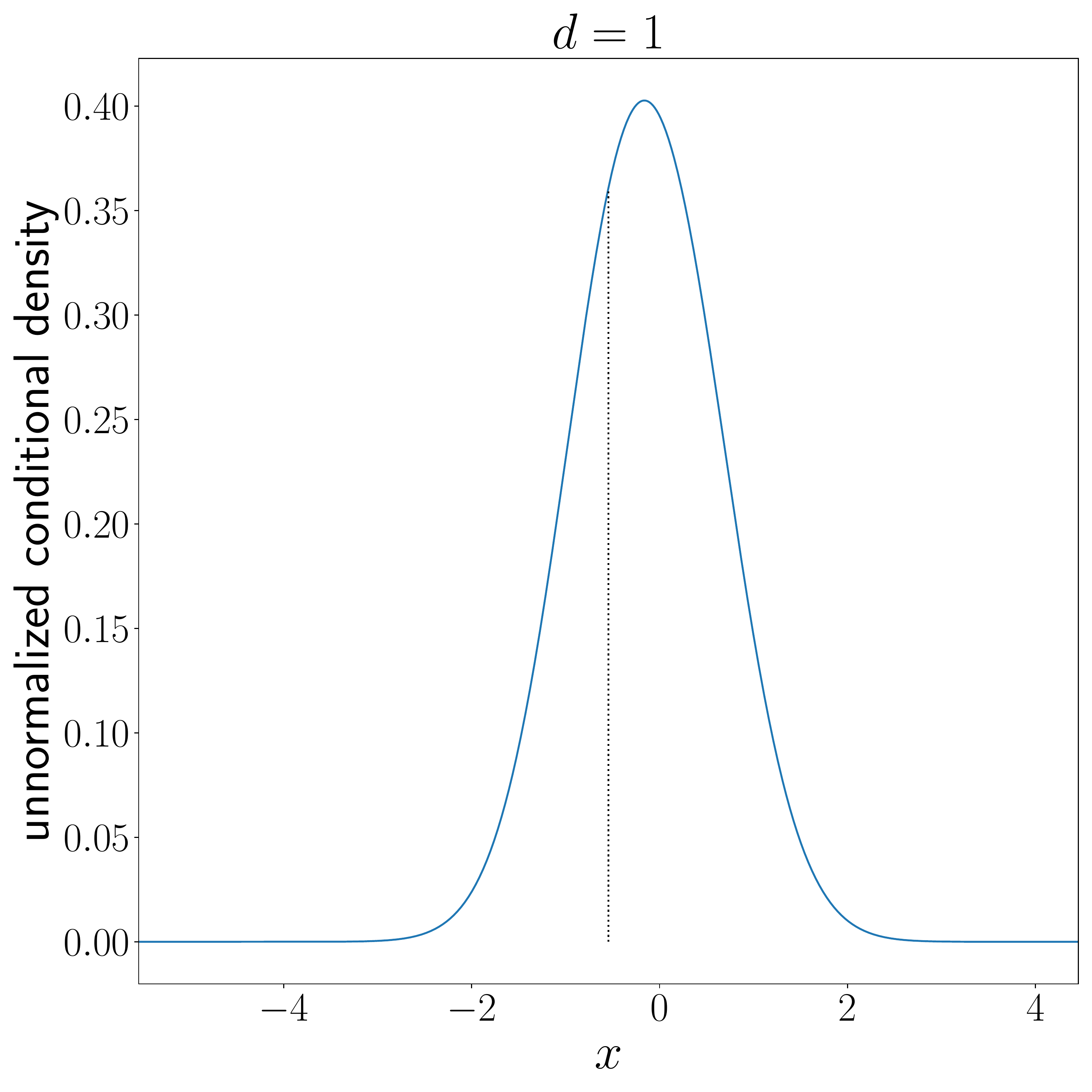}
\qquad
\includegraphics[width=0.45\textwidth, height=0.36\textheight, trim=10mm 20mm 0mm 0mm, clip]{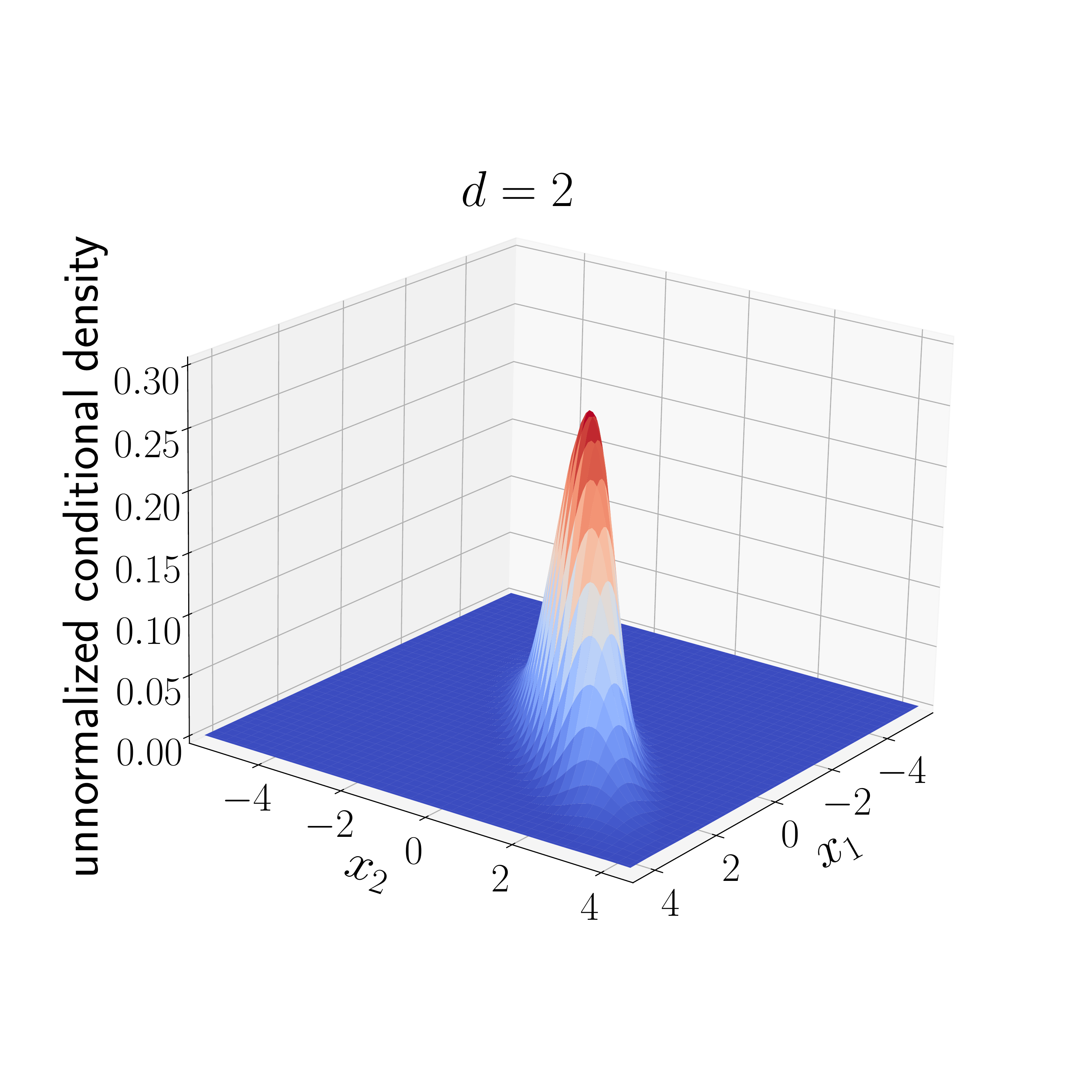}
\caption{Approximations of $X_{\nicefrac{1}{2}}(x)$ for $x \in [\mathcal{Y}_N-5, \mathcal{Y}_N + 5]$
(left) and $x \in [\mathcal{Y}_{N,1}-5, \mathcal{Y}_{N,1} + 5] \times 
[\mathcal{Y}_{N,2} -5, \mathcal{Y}_{n,2} + 5]$ (right). The vertical line in the 
picture at the left shows the location of ${\cal Y}_N$.
}
\label{fig:1}
\end{figure}

\begin{figure} 
\includegraphics[width=0.45\textwidth, trim=0mm 0mm 0mm 0mm, clip]{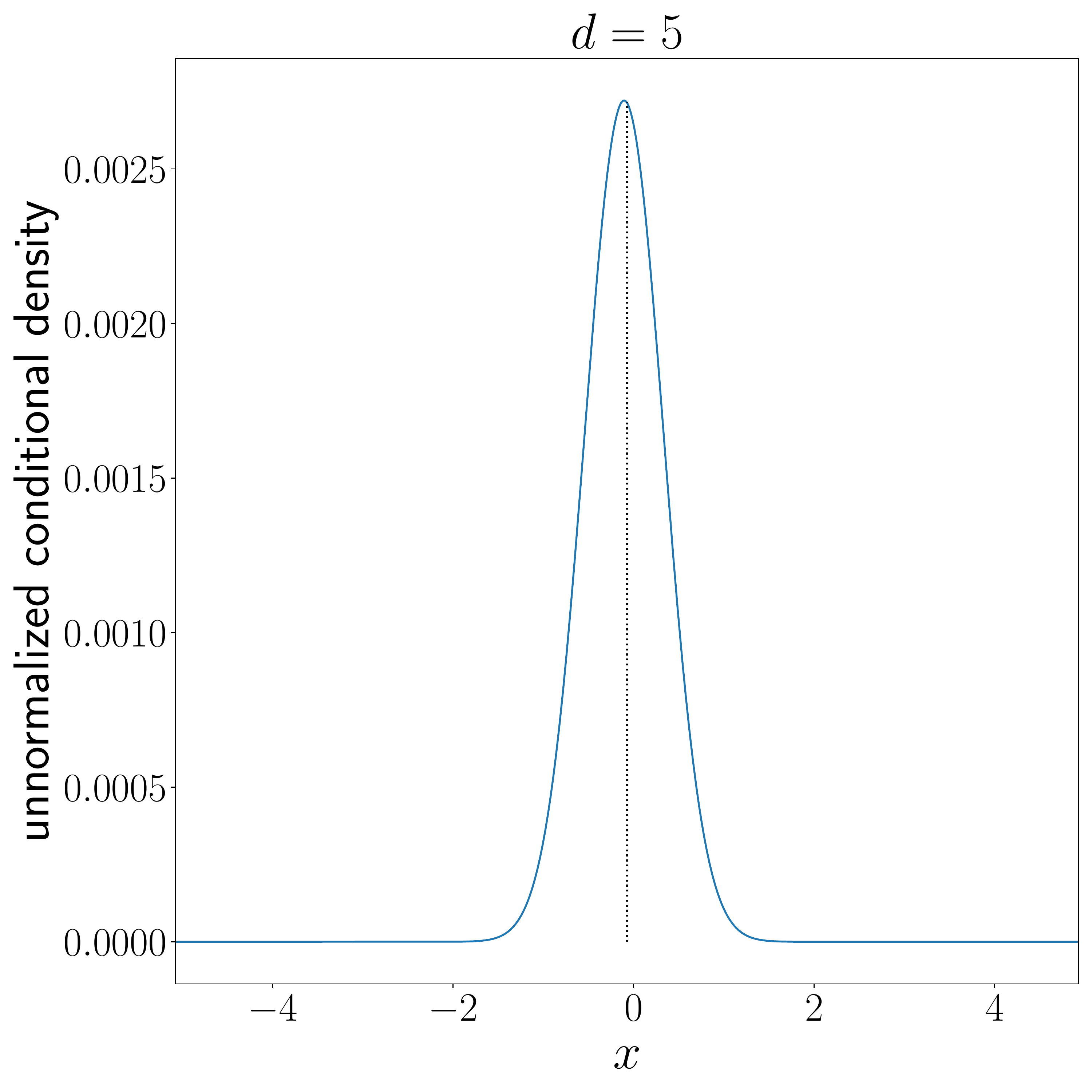}
\qquad
\includegraphics[width=0.45\textwidth, trim=0mm 0mm 0mm 0mm, clip]{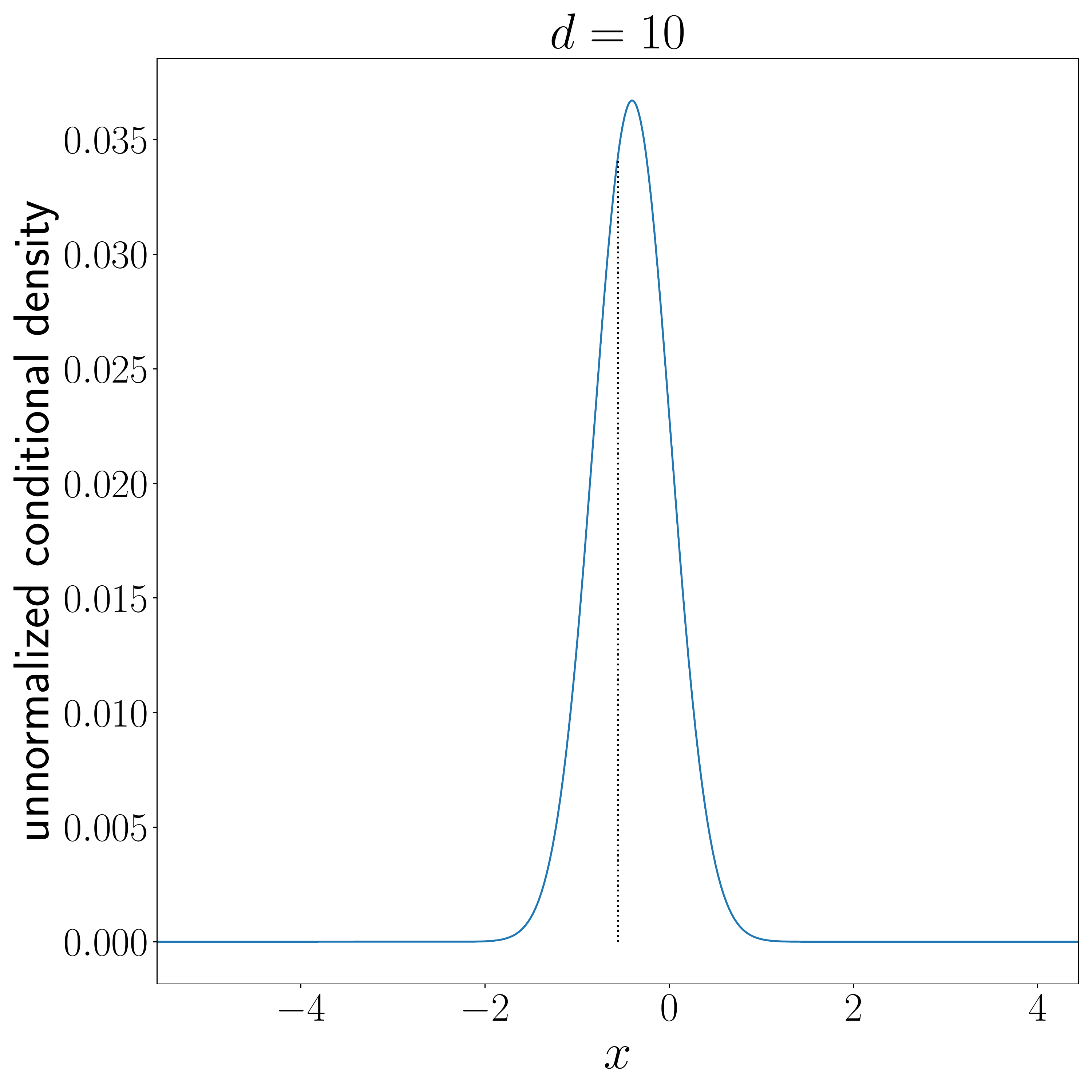}
\caption{Approximations of 
$X_{\nicefrac{1}{2}}(x, {\cal Y}_{N,2} \dots, {\cal Y}_{N,d})$, 
$x \in [{\cal Y}_{N,1}-5, {\cal Y}_{N,1} + 5]$, for $d \in \{5, 10\}$.
The vertical lines show the location of ${\cal Y}_{N,1}$.}
\label{fig:2}
\end{figure}

\begin{figure} 
\includegraphics[width=0.45\textwidth, trim=0mm 0mm 0mm 0mm, clip]{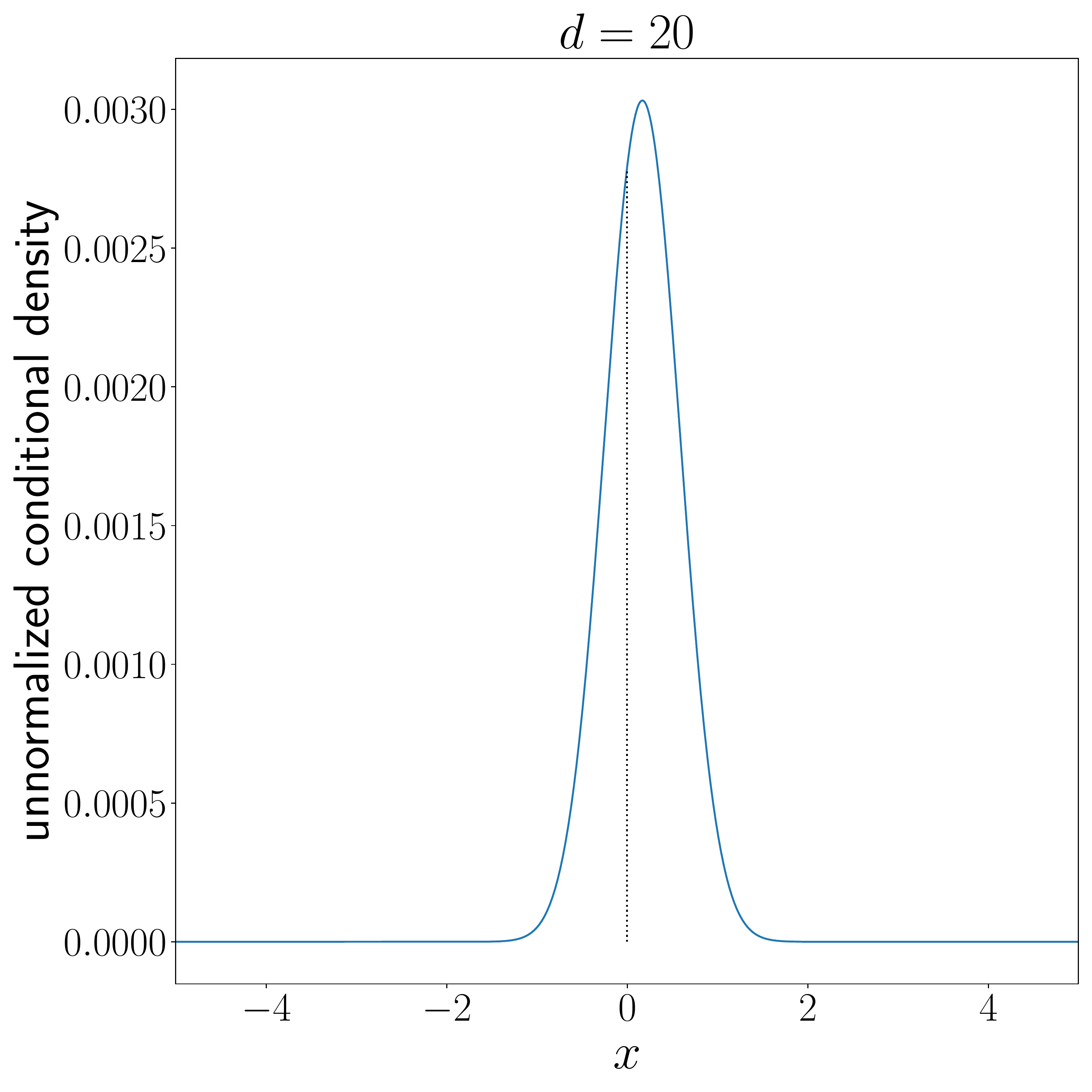}
\qquad
\includegraphics[width=0.45\textwidth, trim=0mm 0mm 0mm 0mm, clip]{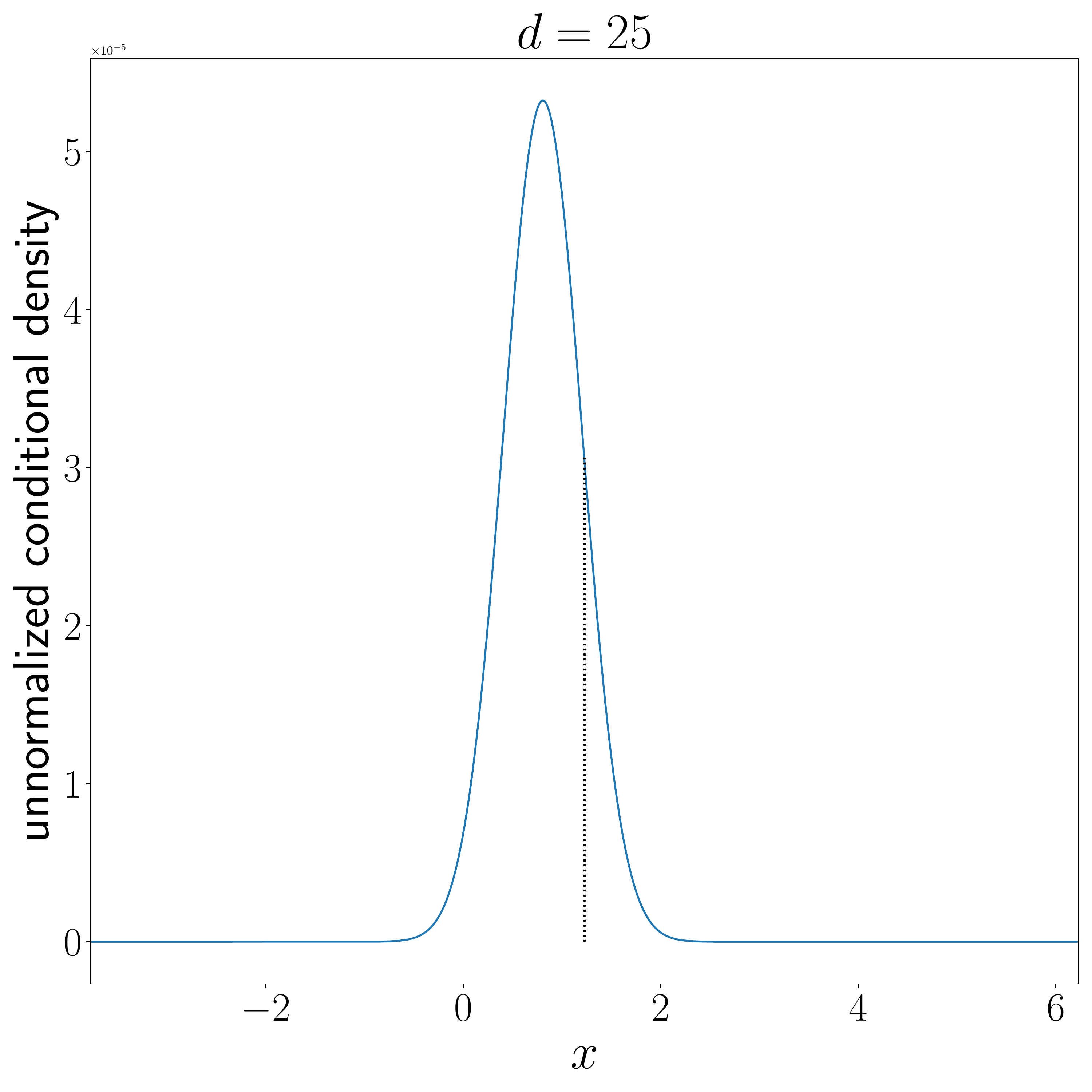}
\caption{Approximations of 
$X_{\nicefrac{1}{2}}(x, {\cal Y}_{N,2} \dots, {\cal Y}_{N,d})$, 
$x \in [{\cal Y}_{N,1}-5, {\cal Y}_{N,1} + 5]$, for $d \in \{20, 25\}$.
The vertical lines show the location of ${\cal Y}_{N,1}$.}
		\label{fig:3}
\end{figure}

\section{Conclusion}

In this paper we have introduced a Monte Carlo method for approximating 
solutions of certain Zakai equations in high dimensions. It is based on a 
Doss--Sussmann-type transformation which 
transforms a Zakai SPDE into a PDE with random coefficients.
This makes it possible to apply the Feynman--Kac formula to obtain a 
Monte Carlo approximation of the solution of a given Zakai equation.
The numerical experiments in \cref{sec:numerics} show that the proposed 
method achieves good results in up to 25 dimensions with fast run times.

\section*{Acknowledgements}
A.J. gratefully acknowledges the Cluster of Excellence EXC 2044-390685587, 
Mathematics Münster: Dynamics-Geometry-Structure funded by the 
Deutsche Forschungsgemeinschaft (DFG, German Research Foundation).
A.N. acknowledges funding by the Nanyang Assistant Professorship Grant 
(NAP Grant) Machine Learning Based Algorithms in Finance and Insurance.

\begin{appendix}
\section*{Appendix. Proof of the main result}
\label{app}

\setcounter{section}{1} 

In this appendix we derive approximation, stability, integrability as well as
regularity results and use them to prove Theorem \ref{thm:main}.

\subsection{Approximation and mollification results for at most polynomially growing functions} \label{subsec:approximation_and_mollification}

\begin{lemma} \label{le:Mollification}
	Let $c,p \in [0, \infty)$, 
	$d \in \N$, 
	$\alpha,T \in (0, \infty)$, and consider at most polynomially growing functions
	$G \in C([0, T]\times\R^d, \R)$ and $H \in C(\R^d, \R)$.
	Moreover, assume that 
	\be \label{eq:Mollification:hoelder-assumption}
	|G(t, x) - G(s, x) | 
	\leq 
	c \, ( 1 + \Vert x\Vert_{\R^d} )^p \, | t - s |^{\alpha} 
	\quad \mbox{for all $s,t \in [0, T]$ and $x \in \R^d$,}
	\ee
	and let $G_n \in C([0, T]\times\R^d, \R)$ and
	$H_n \in C(\R^d, \R)$ for all $n \in \N$, 
	$t \in [0, T]$ and $x \in \R^d$ be given by
	\[
	G_n(t, x) = \left(\tfrac{n}{2\pi}\right)^{\nicefrac{1}{2}} \int_{-\infty}^{\infty} 
	G(\min\{T, \max\{s, 0\} \}, x) \exp\!\big( - \tfrac{n}{2}(t-s)^2\big) \,ds
	\]
	and 
	\[
	H_n(x) = 
	\big(\tfrac{n}{{2\pi}}\big)^{\nicefrac{d}{2}}
	\int_{\R^d}
	H(y) 
	\exp\!\big(-\tfrac{n}{2} \|x - y \|_{\R^d}^2 \big)
	\,dy.
	\]
	Then 
	\begin{enumerate}[{\rm (i)}]
		\item \label{it:Mollifier:convergence-Bn_B} 
		$ \limsup\nolimits_{n \to \infty} \sup_{t \in [0, T]} \sup_{x \in \R^d} \frac{ |G_n(t, x) - G(t, x) | }{ ( 1 + \Vert x \Vert_{\R^d} )^p } = 0 $ 
		and 
		\item \label{it:Mollifier:convergence-Gn-G} 
		$ \limsup\nolimits_{n \to \infty} \sup_{ x \in [-q,q]^d} |H_n(x) - H(x) | = 0 $
		for all $q \in (0, \infty)$.
	\end{enumerate}
\end{lemma}

\begin{proof}[Proof]
	From \eqref{eq:Mollification:hoelder-assumption} and the fact that
	$|\!\min\{T, \max\{s, 0 \} \} - t| \leq |s-t|$ for all $s \in \R$ and $t \in [0, T]$, we obtain
	\[
	\begin{split} 
	\frac{ |G_n(t, x) - G(t, x) | }{ ( 1 + \Vert x \Vert_{\R^d} )^p } 
	& \leq 
	\big( \tfrac{n}{2\pi} \big)^{\!\nicefrac12} 
	\int_{-\infty}^{\infty} 
	\frac{ | G(\min\{T, \max\{s, 0 \} \}, x) - G(t, x) | }{ ( 1 + \Vert x \Vert_{\R^d} )^p } 
	\exp( -\tfrac{n(t-s)^2}{2} )\,ds 
	\\
	& \leq 
	\big( \tfrac{n}{2\pi} \big)^{\!\nicefrac12} 
	\int_{-\infty}^{\infty} c\left| \min\{T, \max\{s, 0\} \} - t\right|^{\alpha} \exp( -\tfrac{n(t-s)^2}{2} )\,ds 
	\\
	& \leq \big( \tfrac{n}{2\pi} \big)^{\!\nicefrac12} 
	\int_{-\infty}^{\infty} c  \left| s - t \right|^{\alpha} \exp( -\tfrac{n(t-s)^2}{2} )\,ds 
	= c \,
	\big( \tfrac{1}{2\pi n^{\alpha}} \big)^{\!\nicefrac12}
	\int_{-\infty}^{\infty} |z|^{\alpha} \exp(-\tfrac{z^2}{2})\,dz
	\end{split}
	\]
	for all $t \in [0,T]$ and $x \in \R^d$. In particular, 
	$ \limsup\nolimits_{n \to \infty} 
	\sup_{t \in [0, T]} \sup_{x \in \R^d} \frac{ | G_n(t, x) - G(t, x) | }{ ( 1 
	+ \Vert x \Vert_{\R^d} )^p } = 0 $, which shows \eqref{it:Mollifier:convergence-Bn_B}. 
Next, note that one has 
	\begin{align}
	\label{eq:G-n-to-G-molli}
	|H_n(x)- H(x)| 
	& \leq
	\big(\tfrac{n}{2\pi}\big)^{\nicefrac{d}{2}}
	\int_{\R^d}
	|H(y)-H(x)|
	\exp( -\tfrac{n}{2} \|x - y \|_{\R^d}^2 )
	\,dy
	\\ \notag
	& =
	\big(\tfrac{1}{2\pi}\big)^{\nicefrac{d}{2}}
	\int_{\R^d} 
	\big|H\big(x+\tfrac{z}{\sqrt{n}}\big)-H\big(x\big)\big|
	\exp( -\tfrac{1}{2} \|z \|_{\R^d}^2 )
	\,dz \quad \mbox{for all $n \in \N$ and $x \in \R^d$,}
	\end{align}
	and the assumption that $H \in C(\R^d, \R)$ is at most polynomially growing 
	implies that 
	\begin{equation} \label{eq:Gcont}
	\sup\nolimits_{n \in \N} \int_{\R^d} 
	\sup\nolimits_{x\in [-q,q]^d} 
	\big|H\big(x+\tfrac{z}{\sqrt{n}}\big)- H\big(x\big)\big|^q
	\exp(-\tfrac{1}{2} \|z \|_{\R^d}^2)
	\,dz<\infty \quad \mbox{for all $q \in (0,\infty)$.}
	\end{equation}
	Combining \eqref{eq:G-n-to-G-molli}, \eqref{eq:Gcont}, the assumption 
	that $H \in C(\R^d, \R)$, the de la Vall\'ee Poussin theorem (cf., e.g., \cite[Corollary~6.21]{Klenke_2014}),
	and the Vitali convergence theorem (cf., e.g., \cite[Theorem~6.25]{Klenke_2014}) 
	yields that 
	$\limsup_{n \to \infty} \sup_{ x \in [-q,q]^d} |H_n(x) - H(x) |  = 0$
	for all $q \in (0, \infty)$.
	This establishes \eqref{it:Mollifier:convergence-Gn-G} and completes the proof of the lemma.
\end{proof}

\begin{lemma} \label{le:loc-unif-conv}
	Let $d, m \in \N$, $T \in (0, \infty)$, and consider two families of functions
	$f^{n,t,x} \in C([0,t], \R^m)$ and $g_n \in C(\R^m, \R)$, 
	$n \in \N_0$, $t \in [0, T]$, $x \in \R^d$, such that for all $q \in (0, \infty)$,
	\bea \label{eq:f0}
	&& \sup_{t \in [0, T]} 
	\sup_{s \in [0, t]} 
	\sup_{x \in [-q, q]^d}
	\lVert f^{0, t, x}(s) \rVert_{\R^m} 
	< \infty, \\
	\label{eq:loc-unif-conv-Y_n-Y_0}
	&& \limsup\nolimits_{n \to \infty} 
	\sup\nolimits_{t \in [0,T]} 
	\sup\nolimits_{s \in [0, t]} 
	\sup\nolimits_{x \in [-q, q]^d} 
	\Vert f^{n,t,x}(s) - f^{0,t,x}(s) \Vert_{\R^m} = 0 \mbox{ and}\\
	 \label{eq:loc-unif-conv-f_n-f_0}
	&& \limsup\nolimits_{n \to \infty} 
	\sup\nolimits_{y \in [-q, q]^m} 
	| g_n(y) - g_0(y) | = 0. 
	\eea
	Then, one has for all $q \in (0, \infty)$,
	\[
	\limsup\nolimits_{n \to \infty} 
	\sup\nolimits_{t \in [0, T]} 
	\sup\nolimits_{s \in [0, t]}
	\sup\nolimits_{x \in [-q, q]^d} 
	| g_n(f^{n,t,x}(s)) - g_0(f^{0,t,x}(x)) | = 0.
	\]
\end{lemma} 

\begin{proof}[Proof]
It follows from \eqref{eq:f0} and \eqref{eq:loc-unif-conv-Y_n-Y_0} that there exist 
$n_q, N_q \in \N$, $q \in (0, \infty)$ such that 
	\begin{equation} \label{eq:loc-unif-conv-unif-bound-Y_n}
	\sup\nolimits_{n \in \{0\}\cup(\N\cap [n_q, \infty))} 
	\sup\nolimits_{t \in [0, T]} 
	\sup\nolimits_{s \in [0, t]}
	\sup\nolimits_{x \in [-q, q]^d} \lVert f^{n, t, x}(s) \rVert_{\R^m} 
	\leq N_q < \infty 
	\end{equation} 
for all $q \in (0,\infty)$.
So one obtains from the assumption that $g_0 \in C(\R^m, \R)$
and \eqref{eq:loc-unif-conv-Y_n-Y_0} that	
\[
	\limsup\nolimits_{n \to \infty} 
	\sup\nolimits_{t \in [0, T]} 
	\sup\nolimits_{s \in [0, t]}
	\sup\nolimits_{x \in [-q, q]^d} | g_0(f^{n, t, x}(s)) - g_0(f^{0, t, x}(s)) | = 0
	\quad \mbox{for all $q \in (0,\infty)$.}
	\]
Combining this with \eqref{eq:loc-unif-conv-f_n-f_0}, \eqref{eq:loc-unif-conv-unif-bound-Y_n}, and the triangle inequality yields
	\[
	\begin{split}
	&
	\limsup\nolimits_{n \to \infty} \bigl[ 
	\sup\nolimits_{t \in [0, T]} 
	\sup\nolimits_{s \in [0, t]} 
	\sup\nolimits_{x \in [-q, q]^d} 
	| g_n(f^{n, t, x}(s)) - g_0(f^{0, t, x}(s)) | \bigr] 
	\\
	& \leq 
	\limsup\nolimits_{n \to \infty} 
	\sup\nolimits_{t \in [0, T]} 
	\sup\nolimits_{s \in [0, t]} 
	\sup\nolimits_{x \in [-q, q]^d} 
	| g_n(f^{n, t, x}(s)) - g_0(f^{n, t, x}(s)) | 
	\\
	& \quad + \limsup\nolimits_{n \to \infty} 
	\sup\nolimits_{t \in [0, T]} 
	\sup\nolimits_{s \in [0, t]} 
	\sup\nolimits_{x \in [-q, q]^d} 
	| g_0(f^{n, t, x}(s)) - g_0(f^{0, t, x}(s)) | 
	\\
	& 
	\leq \limsup\nolimits_{n \to \infty} 
	\sup\nolimits_{t \in [0, T]} 
	\sup\nolimits_{s \in [0, t]} 
	\sup\nolimits_{y \in [- N_q, N_q]^m} | g_n(y) - g_0(y) | 
	= 0
	\end{split}
	\]
	for all $q \in (0, \infty)$, 
	which completes the proof of the lemma.
\end{proof}

\subsection{Integrability and regularity properties for solutions of ODEs with additive noise} 
\label{subsec:apriori_estimates}

The following is a consequence of \cite[Corollary 3.2]{LedTal}. For more details; see e.g.,
\cite[Lemma 6.3]{jentzen2019strong}. 

\begin{lemma} \label{le:integrability-Brownian-motion}
Let $T,c \in [0,\infty)$, $\alpha \in [0,2)$ and $d \in \N$.
Let $W \colon [0, T]\times \Omega \to \R^d$ be a standard Brownian motion
with continuous sample paths defined on a 
probability space $(\Omega,\F,\P)$. Then 
\begin{enumerate}[{\rm (i)}]
 \item \label{it:integrability-Brownian-motion-measurability} 
 the mapping $\Omega \ni \omega \mapsto \sup_{t \in [0,T]} \lVert{W_t(\omega)\rVert}_{\R^d}^{\alpha} \in \R$ is 
 $\F$/$\Borel(\R)$-measurable and 
 \item \label{it:integrability-Brownian-motion-exponential-integrability} 
 $ \E \left[\exp \brak{c \sup_{t\in [0,T]} \lVert{W_t\rVert}_{\R^d}^{\alpha}} \right]<\infty$.
\end{enumerate}
\end{lemma}

\begin{lemma} \label{le:sde-solution-exponential-integrability}
Let $T, c\in(0,\infty)$, $d\in\N$ and
 $\sigma\in\R^{d\times d}$. Let
$b \colon [0,T]\times\R^d \to \R^d$ be a Borel measurable function
satisfying $\lVert b(t, x) \rVert_{\R^d} \leq c ( 1 + \lVert x \rVert_{\R^d}) $
for all $t \in [0, T]$ and $x \in \R^d$.
Let $(\Omega,\F,\P)$ be a probability space supporting a standard Brownian motion
$U \colon [0,T]\times\Omega\to\R^d$ with continuous sample paths. Consider stochastic processes 
$R^{t,x} =(R^{t,x}_s)_{s\in[0,t]} \colon [0,t]\times\Omega\to\R^d$, $x\in\R^d$, $t\in[0,T]$
satisfying
    \be \label{eq:sde-solution-exponential-integrability-sde}
     R^{t,x}_s = x + \int_0^s b(t-r, \cY^{t,x}_r)\,dr + \sigma U_r \quad
     \mbox{for all $t\in[0,T]$, $s \in[0,t]$ and $x \in \R^d$.}
    \ee
Then
\begin{enumerate}[{\rm (i)}]
	\item \label{it:sde-solution-exponential-integrability:item1}
	$\mbox{}$\\[-8mm]
		\[
		\begin{split}
		& 
		\E\Bigl[ 
			\exp\Bigl( p 
			\sup\nolimits_{t \in [0, T]} 
			\sup\nolimits_{s \in [0, t]} 
			\lVert \cY^{t,x}_s \rVert_{\R^d} \Bigr) \Bigr] 
		\\
		& 
		\leq 
			\exp\bigl( p e^{cT} \lVert x \rVert_{\R^d} \bigr)
			\exp\bigl( p c T e^{cT} \bigr)\,
		\E\Bigl[ 
			\exp\bigl( p e^{cT} \sup\nolimits_{t \in [0, T]} \lVert \sigma U_t \rVert_{\R^d} \bigr)
			\Bigr] 
		< \infty 
		\end{split}
		\]
		for all $p \in (0, \infty)$ and $x \in \R^d$, and
	\item \label{it:sde-solution-exponential-integrability:item2} $\mbox{}$\\[-8mm]
	    \[ \label{eq:sde-solution-exponential-integrability-claim}
	    \E\!\left[\exp\Bigl(p\big[
	    	\sup\nolimits_{t\in [0,T]} \sup\nolimits_{s \in [0,t]}
	    	\sup\nolimits_{x\in [-p,p]^d} 
	    	\lVert\cY^{t,x}_s \rVert_{\R^d}\big]\Bigr)\right]<\infty
		\quad \mbox{for all $p \in (0, \infty)$.}
	    \]
	\end{enumerate}
\end{lemma}

\begin{proof}[Proof]
It follows from \eqref{eq:sde-solution-exponential-integrability-sde} ,
the triangle inequality and the assumption that 
$\lVert b(t, x) \rVert_{\R^d} \leq c ( 1 + \lVert x \rVert_{\R^d}) $
for all $t\in [0, T]$ and $x \in \R^d$ that
	\[
	\begin{split}
	\Vert R^{t,x}_s \Vert_{\R^d} 
	&\leq 
	\Vert x \Vert_{\R^d} 
	+
	\int_0^s c \big(1+	\Vert R^{t,x}_r \Vert_{\R^d}\big)\,dr + \sup\nolimits_{r \in [0,T]} \Vert 
	\sigma U_r \Vert_{\R^d}
	\\&
	\leq 
	\lVert{ x \rVert}_{\R^d} + cT + \sup\nolimits_{r \in [0,T]} \lVert \sigma U_r \rVert_{\R^d} 
	+ 
	c \int_0^s \lVert R^{t,x}_r \rVert_{\R^d}\, dr
	\end{split}
	\]
for all $x\in \R^d$, $t \in [0,T]$ and $s \in [0,t]$. Hence, one obtains from
Gronwall's integral inequality (cf., e.g., \cite[Lemma~2.11]{grohs2018proof}) that 
	\[
	\Vert R^{t,x}_s \Vert_{\R^d} 
	\leq
	\left(\lVert x \rVert_{\R^d} + c T +\sup\nolimits_{r \in [0,T]} \lVert \sigma U_r \rVert_{\R^d} \right) 
	e^{c T}
	\]
for all $x\in \R^d$, $t \in [0,T]$ and $s \in [0,t]$. So it follows from
\cref{le:integrability-Brownian-motion} that
	\[
	\begin{split}
	& 
	\E\Bigl[ 
	\exp\Bigl( p 
	\sup\nolimits_{t \in [0, T]} \sup\nolimits_{s \in [0, t]} 
	\lVert R^{t,x}_s \rVert_{\R^d} \Bigr) \Bigr] 
	\\
	& 
	\leq 
	\exp\bigl( p e^{cT} \lVert x \rVert_{\R^d} \bigr)
	\exp\bigl( p c T e^{cT} \bigr)\,
	\E\Bigl[ 
	\exp\Bigl( p e^{cT} \sup\nolimits_{r \in [0, T]} \lVert \sigma U_r \rVert_{\R^d} \Bigr)
	\Bigr] 
	< \infty
	\end{split}
	\]
	for all $p \in (0, \infty)$ and $x \in \R^d$, as well as
	\[ \label{eq:sde-solution-exponential-integrability}
	\begin{split} &
	    \E\!\left[\exp\Bigl(p\big[
	    	\sup\nolimits_{t\in [0,T]} 
	    	\sup\nolimits_{x\in [-p,p]^d} 
	    	\sup\nolimits_{s \in [0,t]}\lVert R^{t,x}_s \rVert_{\R^d}\big]\Bigr)\right] \\
	& 
	\leq 
	\exp\bigl( p e^{cT} \sqrt{d} p \bigr)
	\exp\bigl( p c T e^{cT} \bigr)\,
	\E\Bigl[ 
	\exp\Bigl( p e^{cT} \sup\nolimits_{r \in [0, T]} \lVert \sigma U_r \rVert_{\R^d} \Bigr)
	\Bigr] 
	< \infty
	\end{split}
	    \]
	    for all $p \in (0, \infty)$, which proves the lemma.
\end{proof} 

\begin{lemma} \label{le:Y-differentiability-properties}
Let $T \in (0,\infty)$,  $d \in \N$ and $\sigma \in \R^{d\times d}$. Denote by 
$e_i$, $i \in \{1, 2, \ldots, d\}$, the standard unit vectors in $\R^d$. Let
$b \in C^{0,2}([0,T]\times\R^d,\R^d)$ have bounded partial derivatives
of first and second order with respect to the $x$-variables. Let
$(\Omega, \mathcal{F}, \P)$ be a probability space supporting a
standard Brownian motion $U \colon [0,T]\times\Omega\to\R^d$ with 
continuous sample paths, and let $R^{t,x} = (R^{t,x}_s)_{s\in[0,t]} \colon[0,t]\times\Omega \to \R^d$, 
$t\in [0,T]$, $x \in \R^d$, be stochastic processes satisfying
	\be \label{def:mathcal-Y-smooth-x-Y}
	R^{t,x}_s
	=
	x
	+ 
	\int_0^s b(t-r ,R^{t,x}_r) \,dr 
	+ 
	\sigma U_s \quad \mbox{for all 
	$t \in[0,T]$, $s \in [0,t]$ and $x \in \R^d$.}
	\ee
Then
	\begin{enumerate}[{\rm (i)}]
		\item\label{item:Y-C02} 
		for all $t \in [0,T]$ and $\omega \in \Omega$,
the mapping $(s,x) \mapsto R^{t,x}_s(\omega)$
is in $C^{0,2}([0,t]\times\R^d,\R^d)$,
		\item \label{eq:partial-x-mathcalY} $\mbox{}$\\[-9mm]
            \[
			\tfrac{\partial}{\partial x_i} R^{t,x}_s
            =
            e_i + \int_0^s \big[D_x b(t-r,R^{t,x}_r)\big]
            \!\big(\tfrac{\partial}{\partial x_i} R^{t,x}_r \big)\, dr
            \]
            for all $i \in \{1, 2, \dots, d\}$, $t \in [0,T]$, $s \in [0,t]$ and $x \in \R^d$,             
		\item  \label{eq:delaVaPou-x-cY} $\mbox{}$\\[-9mm]
            \[
            \sup\nolimits_{t \in [0,T]} 
            \sup\nolimits_{s \in [0,t]} 
            \sup\nolimits_{x \in \R^d} 
            \lVert\tfrac{\partial}{\partial x_i} R^{t,x}_s \rVert_{\R^d}
            \leq \exp\Bigl(
            T 
            \sup\nolimits_{(t,x)\in [0,T]\times \R^d}\Vert D_x b(t,x)\Vert_{L(\R^d,\R^d)}
            \Bigr)
            <\infty
            \]
            for all\footnote{For $d,k,n \in \N$ we denote by $L^{(n)}(\R^d, \R^k)$ the set of all
            continuous $n$-linear functions from $(\R^d)^n$ to $\R^k$. By $\lVert \cdot \rVert_{L^{(n)}(\R^d, \R^k)}$ we denote the operator norm on $L^{(n)}(\R^d, \R^k)$ given by 
 $\lVert f \rVert_{L^{(n)}(\R^d, \R^k)} = \sup_{v_1, v_2, \ldots, v_n \in \R^d \setminus\{0\}} \frac{\lVert f(v_1, v_2, \ldots, v_n) \rVert_{\R^k}}{\lVert v_1 \rVert_{\R^d} \lVert v_2 \rVert_{\R^d} \ldots \lVert v_n \rVert_{\R^d}}$.
 For simplicity, we set $L(\R^d, \R^k) = L^{(1)}(\R^d, \R^k)$ and denote the norm 
      $\lVert \cdot \rVert_{L^{(1)}(\R^d, \R^k)}$ by $\lVert \cdot \rVert_{L(\R^d,\R^k)}$.}
            $i \in \{1, 2, \dots, d\}$, 
		\item \label{eq:partial-x-x-mathcalY} $\mbox{}$\\[-8mm]
            \[
            \begin{split}
            &
            \tfrac{\partial^2}{\partial x_i \partial x_j} R^{t,x}_s
            \\&
            = 
            \int_0^s
            \big[ 
            D^2_x b(t-r,R^{t,x}_r)\big]\!\big(\tfrac{\partial}{\partial x_i} 
            R^{t,x}_r, \tfrac{\partial}{\partial x_j} R^{t,x}_r\big)
            \,dr
             +
            \int_0^s
            \big[
            D_x b(t-r,R^{t,x}_r)
            \big]\!\big(\tfrac{\partial^2}{\partial x_i \partial x_j}  R^{t,x}_r \big)
            \,dr
            \end{split}
            \]
            for all $i,j \in \{1, \dots, d\}$, $t \in [0,T]$, $s \in [0,t]$ and $x \in \R^d$,
        \item \label{eq:delaVaPou-x-x-cY} $\mbox{}$\\[-13mm]
            \begin{align*}
            \sup\nolimits_{t \in [0,T]} 
            \sup\nolimits_{s \in [0,t]} 
            \sup\nolimits_{x \in \R^d}
            \lVert \tfrac{\partial^2}{\partial x_i \partial x_j} R^{t,x}_s \rVert_{\R^d} 
            \leq
            T 
            \Bigl[
            \sup\nolimits_{(t,x)\in [0,T]\times \R^d}\big\Vert D^2_x b(t,x)\big\Vert_{L^{(2)}(\R^d,\R^d)}
            \Bigr]
            \\
            \times
            \exp\!\Big(
            3T \sup\nolimits_{(t,x)\in [0,T]\times \R^d}\Vert D_x b(t,x)\Vert_{L(\R^d,\R^d)} \Big)
            <\infty
            \end{align*}
            for all $i,j \in \{1,2,\dots,d\}$.
	\end{enumerate}
\end{lemma}

\begin{proof}[Proof] 
Since, by assumption, $b \in C^{0,2}([0,T]\times\R^d,\R^d)$ has bounded partial 
derivatives with respect to the $x$-variables, it follows from 
\eqref{def:mathcal-Y-smooth-x-Y}, the de la Vall\'ee Poussin theorem 
(cf., e.g., \cite[Corollary~6.21]{Klenke_2014})
and the Vitali convergence theorem (cf., e.g., \cite[Theorem~6.25]{Klenke_2014})
that for all $i \in \{1, 2, \dots, d\}$, $t \in [0,T]$ and $\omega \in \Omega$,
the mapping 
 $(s,x) \mapsto R^{t,x}_s(\omega)\in \R^d$ is in $C^{0,1}([0,t]\times\R^d,\R^d)$ 
and satisfies
	\[
	\label{eq:partial-x-mathcalY-2}
	\tfrac{\partial}{\partial x_i} R^{t,x}_s
	= e_i + \int_0^s
	\big[ D_x b(t-r,R^{t,x}_r)\big]\!\big(\tfrac{\partial}{\partial x_i} 
	R^{t,x}_r \big)\,dr \quad \mbox{for all $s \in [0,t]$, $x \in \R^d$.}
	\]
(cf.\ also, e.g., \cite[Theorem~4.6.5]{kunita1997stochastic}).
This shows \eqref{eq:partial-x-mathcalY}.
	
Now, \eqref{eq:partial-x-mathcalY} and the assumption that 
$b \in C^{0,2}([0,T]\times\R^d,\R^d)$ has bounded partial derivatives 
with respect to the $x$-variables imply that for all 
    $i \in \{1,2,\dots,d\}$, 
    $x\in \R^d$, 
    $t \in [0,T]$ and
    $s \in [0,t]$, one has
	\[
	\lVert \tfrac{\partial}{\partial x_i} R^{t,x}_s \rVert_{\R^d}
	\leq 
	1 
	+ \int_0^s 
	\sup\nolimits_{(w,y) \in [0,T]\times \R^d} \Vert D_x b(w,y)\Vert_{L(\R^d,\R^d)}
	\lVert \tfrac{\partial}{\partial x_i}R^{t,x}_r \rVert_{\R^d} \,dr,
	\]
which, by Gronwall's integral inequality (cf., e.g., \cite[Lemma~2.11]{grohs2018proof}), yields
	\[
	\sup\nolimits_{t \in [0,T]} 
	\sup\nolimits_{s \in [0,t]} 
	\sup\nolimits_{x \in \R^d} 
	\lVert\tfrac{\partial}{\partial x_i} R^{t,x}_s \rVert_{\R^d}
	\leq \exp\Bigl(
	T \sup\nolimits_{(t,x)\in [0,T]\times \R^d}\Vert D_x b
	(t,x)\Vert_{L(\R^d,\R^d)}	\Bigr)
	<\infty,
	\]
	for all $i \in \{1,2,\dots,d\}$. This establishes \eqref{eq:delaVaPou-x-cY}.

Next, we observe that \eqref{def:mathcal-Y-smooth-x-Y},  \eqref{eq:partial-x-mathcalY}, 
the assumption that $b \in C^{0,2}([0,T]\times \R^d,\R^d)$ has bounded partial derivatives
of second order with respect to the $x$-variables, 
the de la Vall\'ee Poussin theorem (cf., e.g., \cite[Corollary~6.21]{Klenke_2014}), 
and the Vitali convergence theorem (cf., e.g., \cite[Theorem~6.25]{Klenke_2014}) ensure that 
for all
    $i,j \in \{1, 2, \dots, d\}$, $t \in [0,T]$
    and $\omega \in \Omega$, the mapping
    $(s,x) \mapsto R^{t,x}_s(\omega)\in \R^d$ is in 
    $C^{0,2}([0,T]\times\R^d,\R^d)$ and satisfies
    \[
	 \label{eq:partial-x-x-mathcalY-2}
	 \tfrac{\partial^2}{\partial x_i \partial x_j} R^{t,x}_s
	 = 
	 \int_0^s
	 \big[ 
	 D^2_x b(t-r,R^{t,x}_r)\big]\!
	 \big(\tfrac{\partial}{\partial x_i} R^{t,x}_r, \tfrac{\partial}{\partial x_j} R^{t,x}_r \big)
	 \,dr
	 +
	 \int_0^s
	 \big[
	 D_x b (t-r,R^{t,x}_r)
	 \big]\!
	 \big( \tfrac{\partial^2}{\partial x_i \partial x_j} R^{t,x}_r \big)
	 \,dr
     \]
     for all $s \in [0,t]$ and $x \in \R^d$ 
(cf.\ also, e.g., \cite[Theorem~4.6.5]{kunita1997stochastic}).
This shows \eqref{item:Y-C02} and \eqref{eq:partial-x-x-mathcalY}.
	 
Since $b \in C^{0,2}([0,T]\times \R^d,\R^d)$ has bounded partial derivatives 
of second order with respect to the $x$-variables, it follows from
\eqref{eq:delaVaPou-x-cY}--\eqref{eq:partial-x-x-mathcalY} that
for all $i,j \in \{1,2,\dots,d\}$, $x\in \R^d$, $t\in [0,T]$ and $s \in [0,t]$, one has
	\[
	\begin{split}
	\lVert \tfrac{\partial^2}{\partial x_i \partial x_j} R^{t,x}_s \rVert_{\R^d} 
    &\leq
	T 
	\Bigl[\sup\nolimits_{(w,y)\in [0,T]\times \R^d}\big\Vert D^2_x b(w,y)
	\big\Vert_{L^{(2)}(\R^d,\R^d)}\Bigr]
    \\
    & 
    \quad \times \exp\!\bigg(
	2T \sup\nolimits_{(w,y)\in [0,T]\times \R^d}\Vert D_x b
	(w,y)\Vert_{L(\R^d,\R^d)}
    \bigg)\\
    & 
	+ 
	\int_0^s  
    \Bigl[\sup\nolimits_{(w,y)\in [0,T]\times \R^d}\Vert D_x b (w,y)\Vert_{L(\R^d,\R^d)}\Bigr]
    \lVert \tfrac{\partial^2}{\partial x_i \partial x_j} R^{t,x}_r\rVert_{\R^d}
	 \, dr,
	 \end{split}
	 \]
which, by Gronwall's integral inequality (cf., e.g., \cite[Lemma~2.11]{grohs2018proof}), 
implies that
    \[
	\begin{split}
    &
    \sup\nolimits_{t \in [0,T]} 
	\sup\nolimits_{s \in [0,t]} 
	\sup\nolimits_{x \in \R^d}
    \lVert \tfrac{\partial^2}{\partial x_i \partial x_j} R^{t,x}_s \rVert_{\R^d} 
    \\
    &\leq
    T 
    \Bigl[\sup\nolimits_{(t,x)\in [0,T]\times \R^d}\big\Vert D^2_x b
    (t,x)\big\Vert_{L^{(2)}(\R^d,\R^d)}\Bigr]\!
    \exp\Bigl(
	3T \sup\nolimits_{(t,x)\in [0,T]\times \R^d}\Vert
	D_x b(t,x)\Vert_{L(\R^d,\R^d)}
    \Bigr)<\infty,
    \end{split}
	\]
for all  $i,j \in \{1,2,\dots,d\}$. This 
establishes \eqref{eq:delaVaPou-x-x-cY} and completes the proof of the lemma.	
\end{proof}

\subsection{Stability properties of solutions of ODEs with additive noise} \label{subsec:stability}

\begin{lemma}\label{le:Y_processes_stability}
Let $T \in (0, \infty)$, $d \in \N$, $\sigma \in \R^{d\times d}$, and consider
two mappings $b, \mathfrak{b} \in C^{0, 2}([0,T]\times\R^d, \R^d)$ that have bounded 
	partial derivatives of first and second order
	with respect to the $x$-variables. Let
	$(\Omega, \mathcal F, \P)$ be a probability space supporting a 
	standard Brownian motion $U \colon [0, T] \times \Omega \to \R^d$
	with continuous sample paths, and consider stochastic processes 
	$R^{t,x}, \cR^{t,x}  \colon [0, t] \times \Omega \to \R^d$, $t \in [0, T]$,
	$x \in \R^d$, satisfying
		\be \label{le:Y_processes_stability:def_mc_Y_mb_Y}
		R^{t, x}_s = x + \int_0^s b(t- r, R^{t,x}_r)\,dr + \sigma U_s
		\quad\text{and}\quad 
		\mathcal{R}^{t, x}_s = x + \int_0^s \mathfrak{b}(t-r, \mathcal{R}^{t,x}_r)\,dr + \sigma U_s
		\ee
		for all $t \in [0, T]$, $s \in [0, t]$ and $x \in \R^d$.
	Then 
	\begin{enumerate}[{\rm (i)}]
		\item \label{le:Y_processes_stability:item1} $\mbox{}$\\[-8.5mm]
			\[
			\begin{split}
			&\sup\nolimits_{s \in [0, t]} \Vert \cR^{t,x}_s - R^{t,x}_s \Vert_{\R^d}
			\\
			&  
			\leq T 
			\sup\nolimits_{(r, y)\in [0,T]\times \R^d}  \Vert \fb(r, y)- b(r, y) \Vert_{\R^d} 
			\exp\!\left(T \sup\nolimits_{(r, y)\in [0,T] \times \R^d} 
			\big\Vert D_x b (r, y)\Vert_{L(\R^d,\R^d)}
			\right)\!, 
			\end{split}
			\]
			for all $t \in [0,T]$ and $x \in \R^d$,
		\item \label{le:Y_processes_stability:item2} $\mbox{}$\\[-11.5mm]
			\begin{align*} 
			\nonumber  
			& 
			\sup\nolimits_{s \in [0, t]} \Vert \tfrac{\partial}{\partial x_i}
			\cR^{t, x}_s - \tfrac{\partial}{\partial x_i} R^{t, x}_s \Vert_{\R^d}
			\\
			& \leq 
			T 
			\sup\nolimits_{s \in [0,t]} 
			\big\Vert D_x \fb (t-s,R^{t,x}_s)
			- D_x b(t-s,\cR^ {t,x}_s)\big\Vert_{L(\R^d,\R^d)} 
			\\ \nonumber
			& 
			\times \exp\!\bigg(
			T \Big[\sup\nolimits_{(r, y)\in [0,T]\times \R^d}
			\Vert D_x \fb (r, y)\Vert_{L(\R^d,\R^d)} 
			+ \sup\nolimits_{(r, y)\in[0,T]\times \R^d}\big\Vert D_x b(r, y)\big\Vert_{L(\R^d,\R^d)}
			\Big]\bigg),
			\end{align*}
	for all $i \in \{1, 2, \ldots, d\}$, $t \in [0, T]$, $x \in \R^d$, and
		\item \label{le:Y_processes_stability:item3} $\mbox{}$\\[-8mm]
			\[
			\begin{split}
			&\sup\nolimits_{s \in [0, t]} \lVert 
			\tfrac{\partial^2}{\partial x_i \partial x_j} \cR^{t,x}_s
			- \tfrac{\partial^2}{\partial x_i \partial x_j} R^{t,x}_s \rVert_{\R^d} 
			\\& \leq 
			\exp\Bigl(T \sup\nolimits_{(r, y) \in [0, T]\times\R^d} \lVert 
			D_x \fb(r, y) \rVert_{L(\R^d, \R^d)} \Bigr) 
			\\
			&  
			\times \bigg[
			2 T^2 
			\sup\nolimits_{(r,y)\in[0,T]\times\R^d} \lVert 
			D^2_x \mathfrak{b}(r, y) \rVert_{L^{(2)}(\R^d, \R^d)} 
			\\
			& \,\,\,\, \times 
			 \exp\Bigl( 3T \sup\nolimits_{(r, y) \in [0, T]\times\R^d} 
			\max\!\Big\{ \lVert D_x \fb (r, y) \rVert_{L(\R^d, \R^d)},  
			\lVert D_x b(r, y) \rVert_{L(\R^d, \R^d)} \Big\} \Bigr) 
			\\
			& \,\,\,\, \times 
		 \sup\nolimits_{s \in [0, t]} \lVert D_x \mathfrak{b}(t-s, \mathcal{R}^{t,x}_s) - 
		 D_x b (t-s, R^{t,x}_s) \rVert_{L(\R^d, \R^d)} 
			\\
			& \,\, + 
			T 
			\exp\Bigl( 2T \sup\nolimits_{(r, y) \in [0, T]\times\R^d} 
			\lVert D_x b(r, y) \rVert_{L(\R^d, \R^d)} \Bigr)
			\\
			& \,\,\,\, \times 
			 \sup\nolimits_{s \in [0, t]} \lVert D^2_x \mathfrak{b}(t-s, \mathcal{R}^{t,x}_s) 
			- D^2_x b(t-s, R^{t,x}_s) \rVert_{L^{(2)}(\R^d, \R^d)} 
			\\
			& 
			\,\, + T^2 
			\sup\nolimits_{(r,y)\in [0,T]\times \R^d} \lVert D^2_x b(r, y) \rVert_{L^{(2)}(\R^d, \R^d)} 
			\\
			& \,\,\,\, \times 
			\exp\Bigl( 3T \sup\nolimits_{(r, y)\in [0, T]\times\R^d} \lVert D_x \mathfrak{b}(r, y) \rVert_{L(\R^d, \R^d)} \Bigr) 
			\\
			& \,\,\,\, \times 
			\sup\nolimits_{s \in [0, t]} \lVert D_x \mathfrak{b}(t-s, \mathcal{R}^{t,x}_s) 
			- D_x b (t-s, R^{t,x}_s) \rVert_{\R^d} \bigg].
			\end{split}
			\]
			for all $i, j \in \{1, 2, \ldots, d\}$, $t \in [0, T]$ and $x \in \R^d$.
				\end{enumerate}
\end{lemma}

\begin{proof}[Proof] 
Throughout this proof we assume without loss of generality that 
	\be \label{eq:Y_processes_stability_wlog}
	\sup\nolimits_{(t,x) \in [0,T]\times \R^d} \Vert \mathfrak{b}(t,x) - b(t,x) \Vert_{\R^d} < \infty.
	\ee
From \eqref{le:Y_processes_stability:def_mc_Y_mb_Y}
and the triangle inequality we obtain
\[
	\begin{split}
	& \lVert \mathcal{R}^{t,x}_s - R^{t,x}_s \rVert_{\R^d} 
	\leq 
	\int_0^s \lVert \mathfrak{b}(t-r,\mathcal{R}^{t,x}_r) - b(t-r, R^{t,x}_r) \rVert_{\R^d} \, dr \\ 
	& \leq
	\int_0^s \Big[ 
	\lVert \mathfrak{b}(t-r,\mathcal{R}^{t,x}_r) - b(t-r,\mathcal{R}^{t,x}_r) \rVert_{\R^d} 
	+
	\lVert b(t-r,\mathcal{R}^{t,x}_r) - b(t-r,R^{t,x}_r) \rVert_{\R^d} \Big] \,dr
	\end{split}
	\]
	for all $t \in [0,T]$, $s \in [0,t]$ and $x \in \R^d$.
Combining this with \eqref{eq:Y_processes_stability_wlog} and the assumption that 
$b$ has bounded partial derivatives with respect to the $x$-variables yields
	\[
	\begin{split} 
	\Vert \mathcal{R}^{t,x}_s - R^{t,x}_s \Vert_{\R^d}
	&\leq  
	T \sup\nolimits_{(r, y) \in [0, T]\times \R^d}  \Vert \mathfrak{b}(r,y)- b(r,y) \Vert_{\R^d}\\ 
	& 
	+
	\sup\nolimits_{(r, y) \in [0, T]\times \R^d} 
	\big\Vert D_x b(r, y)\Vert_{L(\R^d,\R^d)}
	\int_0^s \Vert
	\mathcal{R}^{t,x}_r - R^{t,x}_r \Vert_{\R^d}
	\,dr
	\end{split}
	\]
for all $t \in [0, T]$, $s \in [0, t]$ and $x \in \R^d$. Therefore, we obtain from
Gronwall's integral inequality (cf., e.g., \cite[Lemma~2.11]{grohs2018proof}) that
	\[
	\begin{split}
	&\sup\nolimits_{s \in [0,t]}
	\Vert \mathcal{R}^{t,x}_s - R^{t,x}_s \Vert_{\R^d} \\
	&\leq 
	T \sup\nolimits_{(r, y)\in [0,T]\times \R^d}  \Vert \mathfrak{b}(r, y)- b(r, y) \Vert_{\R^d}
	\exp\Bigl(T \sup\nolimits_{(r, y)\in [0,T]\times \R^d} 
	\big\Vert D_x b(r, y)\Vert_{L(\R^d,\R^d)}
	\Bigr)
	\end{split}
	\]
	for all $t\in [0,T]$ and $x \in \R^d$, which
establishes \eqref{le:Y_processes_stability:item1}. 

Next, observe that \eqref{le:Y_processes_stability:def_mc_Y_mb_Y},  
\cref{le:Y-differentiability-properties}.\eqref{eq:partial-x-mathcalY} and the triangle inequality imply that 
	\[
	\begin{split}
	\lVert \tfrac{\partial}{\partial x_i} \mathcal{R}^{t,x}_s - \tfrac{\partial}{\partial x_i} 
	R^{t,x}_s \rVert_{\R^d}
	& \leq 
	\int_0^s \big \lVert 
	\big[D_x \mathfrak{b}(t-r, \mathcal{R}^{t,x}_r)\big](\tfrac{\partial}{\partial x_i} 
	\mathcal{R}^{t,x}_r) - \big[D_x b(t-r, R^{t,x}_r) \big](\tfrac{\partial}{\partial x_i} R^{t,x}_r)
	\big \rVert_{\R^d}\,dr	
	\\
	& \leq 
	\int_0^s \lVert D_x \mathfrak{b}(t-r, \mathcal{R}^{t,x}_r) \rVert_{L(\R^d,\R^d)} 
	\lVert \tfrac{\partial}{\partial x_i}\mathcal{R}^{t,x}_r - \tfrac{\partial}{\partial x_i} R^{t,x}_r \rVert_{\R^d} \,dr 
	\\
	& + 
	\int_0^s \lVert D_x (t-r, \mathcal{R}^{t,x}_r)
	- D_x b(t-r,R^{t,x}_r) \rVert_{L(\R^d, \R^d)} \lVert \tfrac{\partial}{\partial x_i} 
	R^{t,x}_r \rVert_{\R^d} \,dr
	\end{split}
	\]
for all $i \in \{1,2,\dots,d\}$, $t \in [0,T]$, $s \in [0,t]$ and $x \in \R^d$. Therefore,
we obtain from another application of 
\cref{le:Y-differentiability-properties}.\eqref{eq:delaVaPou-x-cY} that
	\[
	\begin{split}
	\lVert \tfrac{\partial}{\partial x_i} \mathcal{R}^{t,x}_s 
	- \tfrac{\partial}{\partial x_i} R^{t,x}_s \rVert_{\R^d}
	& \leq 
	\sup\nolimits_{(r, y) \in [0, T]\times\R^d} \big \Vert D_x \mathfrak{b}(r, y) \big \Vert_{L(\R^d,\R^d)} 
	\int_0^s \lVert \tfrac{\partial}{\partial x_i}\mathcal{R}^{t,x}_r - \tfrac{\partial}{\partial x_i} R^{t,x}_r \rVert_{\R^d} \,dr 
	\\[1ex]
	& + 
	T  \sup\nolimits_{r \in [0, t]} \big \Vert D_x \mathfrak{b}(t-r, \mathcal{R}^{t,x}_r) 
	- D_x b(t-r,R^{t,x}_r) \rVert_{L(\R^d, \R^d)} 
	\\
	& \quad \times
	\exp\!\bigg( T \sup\nolimits_{(r, y)\in [0,T]\times \R^d}\Vert D_x b
	(r, y)\Vert_{L(\R^d,\R^d)} \bigg)  
	\end{split}
	\]
for all $i \in \{1, 2, \ldots, d\}$, $t \in [0, T]$, $s \in [0, t]$ and $x \in \R^d$.
Gronwall's inequality (cf., e.g., \cite[Lemma~2.11]{grohs2018proof}) hence ensures that
	\begin{align*}
	\lVert \tfrac{\partial}{\partial x_i} \mathcal{R}^{t,x}_s - \tfrac{\partial}{\partial x_i} R^{t,x}_s \rVert_{\R^d} 
	\leq
	T 
	\sup\nolimits_{r \in [0,t]} 
	\big\Vert D_x \mathfrak{b}(t-r,\mathcal{R}^{t,x}_r)-
	D_x b(t-r,R^ {t,x}_r)\big\Vert_{L(\R^d,\R^d)} \\
	\times 
	\exp\!\bigg(
	T \Big[ \sup\nolimits_{(r, y) \in [0,T]\times \R^d}\Vert D_x b
	(r, y)\Vert_{L(\R^d,\R^d)}
	+
	\sup\nolimits_{(r, y)\in[0,T]\times \R^d}\big \Vert D_x \mathfrak{b}(r, y)\big\Vert_{L(\R^d,\R^d)}
	\Big]\bigg)
	\end{align*}
for all $i \in \{1, 2, \dots, d\}$, 
$t \in [0, T]$, $s \in [0,t]$ and $x \in \R^d$. This shows \eqref{le:Y_processes_stability:item2}. 

Now, note that it follows from \eqref{le:Y_processes_stability:def_mc_Y_mb_Y} and 
\cref{le:Y-differentiability-properties}.\eqref{eq:partial-x-x-mathcalY} that
	\[
	\begin{split}
	\tfrac{\partial^2}{\partial x_i \partial x_j} \mathcal{R}^{t,x}_s
	-
	\tfrac{\partial^2}{\partial x_i \partial x_j} R^{t,x}_s 
	& =
	\int_0^s
	\big[ D^2_x \mathfrak{b}(t-r, \mathcal{R}^{t,x}_r) \big]\!
	\big(\tfrac{\partial}{\partial x_i} \mathcal{R}^{t,x}_r, 
	\tfrac{\partial}{\partial x_j} \mathcal{R}^{t,x}_r - \tfrac{\partial}{\partial x_j} R^{t,x}_r \big)
	\, dr 
	\\
	& + 
	\int_0^s 
	\big[ D^2_x \mathfrak{b}(t-r, \mathcal{R}^{t,x}_r) \big]\!
	\big(\tfrac{\partial}{\partial x_i} \mathcal{R}^{t,x}_r 
	- \tfrac{\partial}{\partial x_i} R^{t,x}_r, \tfrac{\partial}{\partial x_j} R^{t,x}_r \big)\,dr
	\\
	& + 
	\int_0^s 
	\big[ D^2_x \mathfrak{b}(t-r, \mathcal{R}^{t,x}_r) 
		- D^2_x b(t-r, R^{t,x}_r) \big]\!
	\big(\tfrac{\partial}{\partial x_i} R^{t,x}_r, 
	\tfrac{\partial}{\partial x_j} R^{t,x}_r \big) \, dr
	\\ 
	& 
	+ 
	\int_0^s 
	\big[ D_x \mathfrak{b}(t-r, \mathcal{R}^{t,x}_r) \big]
	\!\big(\tfrac{\partial^2}{\partial x_i \partial x_j} \mathcal{R}^{t,x}_r
	-	\tfrac{\partial^2}{\partial x_i \partial x_j} R^{t,x}_r \big) \, dr
	\\
	&
	+ 
	\int_0^s 
	\big[ D_x \mathfrak{b}(t-r, \mathcal{R}^{t,x}_r) 
		- D_x b(t-r, R^{t,x}_r) \big]\!
	\big( \tfrac{\partial^2}{\partial x_i \partial x_j} R^{t,x}_r \big)
	\, dr
	\end{split}
	\]
for all $i,j \in \{1, 2, \dots, d\}$, 
$t \in [0, T]$, $s \in [0,s]$ and $x \in \R^d$.
\eqref{le:Y_processes_stability:item2} together with \eqref{eq:delaVaPou-x-cY} 
and \eqref{eq:delaVaPou-x-x-cY} of \cref{le:Y-differentiability-properties}
therefore yield
\[
	\begin{split}
	& 
	\lVert 
	\tfrac{\partial^2}{\partial x_i \partial x_j} \mathcal{R}^{t,x}_s
	- \tfrac{\partial^2}{\partial x_i \partial x_j} R^{t,x}_s \rVert_{\R^d} 
	\\
	& \leq
	2 T^2 
	\sup\nolimits_{(r,y)\in[0,T]\times\R^d} \lVert D^2_x \mathfrak{b}(r, y) \rVert_{L^{(2)}(\R^d, \R^d)} 
	\\
	& \quad \times
	\exp\Bigl( 3T \sup\nolimits_{(r, y) \in [0, T]\times\R^d} 
	\max\!\Big\{ \lVert D_x \mathfrak{b}(r, y) \rVert_{L(\R^d, \R^d)},  
			\lVert D_x b(r, y) \rVert_{L(\R^d, \R^d)} \Big\} \Bigr)
	\\
	& \quad \times
	 \sup\nolimits_{r \in [0, t]} \lVert D_x \mathfrak{b}(t-r, \mathcal{R}^{t,x}_r) - D_x b(t-r, R^{t,x}_r) \rVert_{L(\R^d, \R^d)} 
	\\
	& + 
	T 
	\exp\Bigl( 2T \sup\nolimits_{(r, y) \in [0, T]\times\R^d} \lVert D_x b(r, y) \rVert_{L(\R^d, \R^d)} \Bigr) 
	\\
	& \quad \times
	\sup\nolimits_{r \in [0, t]} \lVert D^2_x \mathfrak{b}(t-r, \mathcal{R}^{t,x}_r) 
	- D^2_x b(t-r, R^{t,x}_r) \rVert_{L^{(2)}(\R^d, \R^d)} 
	\\
	&
	+ T^2 
	\sup\nolimits_{(r,y)\in [0,T]\times \R^d} \lVert D^2_x b(r, y) \rVert_{L^{(2)}(\R^d, \R^d)}
	\\
	& \quad \times \exp\Bigl( 3T \sup\nolimits_{(r, y)\in [0, T]\times\R^d} \lVert D_x \mathfrak{b}(r, y) \rVert_{L(\R^d, \R^d)} \Bigr) 
	\\
	& \quad \times
	\sup\nolimits_{r \in [0, t]} \lVert D_x \mathfrak{b}(t-r, \mathcal{R}^{t,x}_r) 
	- D_x b(t-u, R^{t,x}_u) \rVert_{\R^d} 
	\\
	& 
	+ 
	 \sup\nolimits_{(r, y) \in [0, T]\times\R^d} \lVert D_x \mathfrak{b}(r, y) \rVert_{L(\R^d, \R^d)} 
	\int_0^s 
	\lVert \tfrac{\partial^2}{\partial x_i \partial x_j} 
	\mathcal{R}^{t,x}_r - \tfrac{\partial^2}{\partial x_i \partial x_j} R^{t,x}_r \rVert_{\R^d} \, dr
	\end{split}
	\]
for all $i, j \in \{1, 2, \dots, d\}$, 
$t \in [0, T]$, $s \in [0, t]$ $x \in \R^d$. So it follows from
	Gronwall's integral inequality (cf., e.g., \cite[Lemma~2.11]{grohs2018proof}) that	
	\[
	\begin{split}
	&\lVert 
	\tfrac{\partial^2}{\partial x_i \partial x_j} \mathcal{R}^{t,x}_s
	- \tfrac{\partial^2}{\partial x_i \partial x_j} R^{t,x}_s \rVert_{\R^d} 
	\\& \leq 
	\exp\Bigl(T \sup\nolimits_{(r, y) \in [0, T]\times\R^d} \lVert D_x \mathfrak{b}(r, y) \rVert_{L(\R^d, \R^d)} \Bigr)
	\\
	& \times \Bigg[ 
	2 T^2 
	\sup\nolimits_{(r,y)\in[0,T]\times\R^d} \lVert D^2_x \mathfrak{b}(r, y) \rVert_{L^{(2)}(\R^d, \R^d)} 
	\\
	& \qquad \times 
	 \exp\Bigl( 3T \sup\nolimits_{(r, y) \in [0, T]\times\R^d} 
	\max\!\Big\{ \lVert D_x \mathfrak{b}(r, y) \rVert_{L(\R^d, \R^d)},  
	\lVert D_x b(r, y) \rVert_{L(\R^d, \R^d)} \Big\} \Bigr) 
	\\
	& \qquad \times
	 \sup\nolimits_{s \in [0, t]} \lVert D_x \mathfrak{b}(t-s, \mathcal{R}^{t,x}_s) 
	 - D_x b(t-s, R^{t,x}_s) \rVert_{L(\R^d, \R^d)} 
	\\
	& \quad + 
	T 
	\exp\Bigl( 2T \sup\nolimits_{(r, y) \in [0, T]\times\R^d} \lVert D_x b(r, y) \rVert_{L(\R^d, \R^d)} \Bigr)
	\\
	& \qquad \times
	 \sup\nolimits_{r \in [0, t]} \lVert D^2_x \mathfrak{b}(t-r, \mathcal{R}^{t,x}_r) 
	- D^2_x b(t-r, R^{t,x}_r) \rVert_{L^{(2)}(\R^d, \R^d)} 
	\\
	& \quad
	+ T^2 
	\sup\nolimits_{(r,y)\in [0,T]\times \R^d} \lVert D^2_x b(r, y) 
	\rVert_{L^{(2)}(\R^d, \R^d)} 
	\\
	& \qquad \times \exp\Bigl( 3T \sup\nolimits_{(r, y)\in [0, T]\times\R^d} 
	\lVert D_x \mathfrak{b}(r, y) \rVert_{L(\R^d, \R^d)} \Bigr) 
	\\
	& \qquad \times
	\sup\nolimits_{r \in [0, t]} \lVert D_x \mathfrak{b}(t-r, \mathcal{R}^{t,x}_r) 
	- D_x b(t-r, R^{t,x}_r) \rVert_{\R^d} \Bigg]
	\end{split} 
	\]
for all $i,j \in \{ 1, 2, \dots, d\}$, 
$t \in [0, T]$, $s \in [0, t]$ and $x \in \R^d$. This shows \eqref{le:Y_processes_stability:item3}
and completes the proof of the lemma.	
\end{proof}

\subsection{Differentiability properties of certain random fields defined in terms of
ODEs with additive noise} 
\label{subsec:differentiability}

\begin{lemma} \label{le:V}
Let $T \in (0,\infty)$, $d \in \N$, $\sigma\in \R^{d\times d}$, and consider a 
function $b \in C^{0,2}([0,T]\times\R^d,\R^d)$ with bounded partial derivatives
of first and second order with respect to the $x$-variables. Let
$\varphi \in C^2(\R^d,[0,\infty))$ and
$B \in C^{0,2}([0,T]\times \R^d,\R)$ such that
$\sup_{(t,x)\in [0,T]\times \R^d} \tfrac{B(t,x)}{1+\Vert x \Vert_{\R^d}} <\infty$
and the first and second order partial derivatives of $B$ with respect to the 
$x$-variables are at most polynomially growing. Let
$(\Omega, \mathcal{F}, \P)$ be a probability space supporting a 
standard Brownian motion $U \colon [0,T]\times\Omega\to\R^d$ 
with continuous sample paths. Consider stochastic processes
$R^{t,x} \colon [0,t] \times\Omega \to \R^d$, $t\in [0,T]$, $x \in \R^d$, 
satisfying
	\begin{equation}\label{def:mathcal-Y-smooth-x-V}
	R^{t,x}_s
	=
	x
	+ 
	\int_0^s b(t-r,R^{t,x}_r) \,dr 
	+ 
	\sigma U_r
	\end{equation}
	for all $x\in \R^d$, $t\in[0,T]$ and $s \in [0,t]$. Let the function
	$\mathcal{U} \colon [0,T]\times\R^d\times \Omega \to \R$
be given by
	\be
	\label{mathfrak-V-x-V}
	\begin{split}
	&\mathcal{U} (t,x)
	= \varphi(R^{t,x}_t)\exp\!\Big({\textstyle\int\limits_0^t} B(t-s,R^{t,x}_s)\,ds\Big)
	\end{split}
	\ee
for all $t \in [0,T]$ and $x \in \R^d$.
Then
	\begin{enumerate}[{\rm (i)}] 
	\item\label{V-smooth-smooth-x-V} 
for all $\omega \in \Omega$, the mapping
    $(t,x)\mapsto \mathcal{U}(t,x,\omega)$ is in $C^{0,2}([0,T]\times \R^d,\R)$,  
    \item \label{eq:partial-x-V} $\mbox{}$\\[-13mm]
        \begin{align*}
        & \tfrac{\partial}{\partial x_i} \mathcal{U}(t,x)
        =
        \exp\!\Big({\textstyle\int\limits_0^t} B(t-s,R^{t,x}_s)\,ds\Big)
        \\
        & \times
        \Big[\big[\nabla \varphi (R^{t,x}_t) \big]^T
        \!\big(\tfrac{\partial}{\partial x_i} R^{t,x}_t\big)
        +
        \varphi(R^{t,x}_t)
               \int_0^t \big[ D_x B(t-s, R^{t,x}_s)\big]\!\big(\tfrac{\partial}{\partial x_i} 
        R^{t,x}_s\big)\,ds  \Big]
        \end{align*}
for all $i\in \{1,2,\dots,d\}$, $t \in [0,T]$, $x \in \R^d$, and
    \item\label{eq:partial-x-x-V} $\mbox{}$\\[-9mm]
        \[
        \begin{split}
        \tfrac{\partial^2}{\partial x_i \partial x_j} \mathcal{U}(t,x)
        & =  
        \exp\!\Big({\textstyle\int\limits_0^t} B(t-s,R^{t,x}_s)\,ds\Big)
        \bigg[
        \big[D^2 \varphi(R^{t,x}_t)\big] 
        \!\big(\tfrac{\partial}{\partial x_i} R^{t,x}_t, \tfrac{\partial}{\partial x_j} R^{t,x}_t \big) 
        \\
        & \quad +
        \big[\nabla \varphi(R^{t,x}_t)\big]^T
        \!\big(\tfrac{\partial^2}{\partial x_i \partial x_j}R^{t,x}_t\big)
        \\
        & 
        \quad +
        \big[\nabla \varphi(R^{t,x}_t)\big]^T
        \!\big(\tfrac{\partial}{\partial x_j}R^{t,x}_t\big)
        \int_0^t 
        \big[ D_x B(t-s,R^{t,x}_s)\big]\!\big(\tfrac{\partial}{\partial x_i} R^{t,x}_s\big)\,ds
        \\ 
        & 
        \quad +
        \big[\nabla \varphi(R^{t,x}_t)\big]^T \!\big(\tfrac{\partial}{\partial x_i}R^{t,x}_t\big)
        \int_0^t \big[ D_x B(t-s,R^{t,x}_s)\big]\!\big(\tfrac{\partial}{\partial x_j} R^{t,x}_s\big)\,ds
        \\
        & 
        \quad +
        \varphi(R^{t,x}_t) 
        \int_0^t 
        \big[ D_x B(t-s,R^{t,x}_s)\big]\!\big(\tfrac{\partial}{\partial x_j} R^{t,x}_s\big)\,ds
        \\
        &
        \qquad \times
        \int_0^t 
        \big[ D_x B(t-s,R^{t,x}_s)\big]\!\big(\tfrac{\partial}{\partial x_i} R^{t,x}_s\big)\,ds
        \\
        &
        \quad +
        \varphi(R^{t,x}_t)
        \int_0^t 
        \big[ D^2_x B(t-s,R^{t,x}_s) \big]\!\big(\tfrac{\partial}{\partial x_i} R^{t,x}_s,\tfrac{\partial}{\partial x_j} R^{t,x}_s\big)\,ds
        \\
      & \quad +
   \varphi(R^{t,x}_t)
        \int_0^t 
        \big[ D_x B(t-s,R^{t,x}_s)\big]\!\big(\tfrac{\partial^2}{\partial x_i \partial x_j} R^{t,x}_s\big) ds
        \bigg]
        \end{split}
        \]
        for all $i,j\in \{1,2,\dots,d\}$, $t \in [0,T]$ and $x \in \R^d$.
        \end{enumerate}
\end{lemma}

\begin{proof}[Proof]
It follows from \cref{le:Y-differentiability-properties}.\eqref{item:Y-C02},
our assumptions on $b$, $\varphi$ and $B$, the chain rule, the fundamental theorem of calculus, 
\eqref{def:mathcal-Y-smooth-x-V} and \eqref{mathfrak-V-x-V} that for all
$t \in [0, T]$ and $\omega \in \Omega$, the mapping 
$x \mapsto \mathcal{U}(t,x,\omega)$ belongs to $C^1(\R^d, \R)$
and
\[
    \begin{split}
    \tfrac{\partial}{\partial x_i} \mathcal{U}(t,x)
    &=
    \big[\nabla \varphi(R^{t,x}_t)\big]^T \!\big(\tfrac{\partial}{\partial x_i}R^{t,x}_t\big)
    \exp\!\Big({\textstyle\int\limits_0^t} B(t-s,R^{t,x}_s)\,ds\Big)\\
    & 
    +
    \varphi(R^{t,x}_t)
    \exp\!\Big({\textstyle\int\limits_0^t} B(t-s,R^{t,x}_s)\,ds\Big)
    \int_0^t \big[ D_x B(t-s, R^{t,x} _s)\big]\!\big(\tfrac{\partial}{\partial x_i} R^{t,x}_s\big) \,ds
    \end{split}
    \]
  for all $i \in \{1, 2, \dots, d\}$, $t \in [0, T]$ and $x \in \R^d$.
  This shows \eqref{eq:partial-x-V}.

Similarly, it follows from \cref{le:Y-differentiability-properties}.\eqref{item:Y-C02},
\eqref{eq:partial-x-V}, and the assumptions that for all
$t \in [0, T]$ and $\omega \in \Omega$, the mapping
$x \mapsto \mathcal{U}(t,x,\omega)$ is in $C^2(\R^d, \R)$ and
    \[    
    \begin{split}
    &
    \exp\!\Big(\!-{\textstyle\int\limits_0^t} B(t-s,R^{t,x}_s)\,ds\Big)
    \tfrac{\partial^2}{\partial x_i \partial x_j} \mathcal{U}(t,x)\\
    &=
    \big[D^2 \varphi (R^{t,x}_t)\big] 
    \!\big(\tfrac{\partial}{\partial x_i}R^{t,x}_t, \tfrac{\partial}{\partial x_j}R^{t,x}_t \big) 
    +
    \big[\nabla \varphi (R^{t,x}_t)\big]^T \big(\tfrac{\partial^2}{\partial x_i \partial x_j}R^{t,x}_t\big)
    \\ 
    & 
    +
    \big[\nabla \varphi (R^{t,x}_t) \big]^T
    \!\big(\tfrac{\partial}{\partial x_j} R^{t,x}_t\big)
    \int_0^t 
    \big[ D_x B(t-s,R^{t,x}_s)\big]\!\big(\tfrac{\partial}{\partial x_i} R^{t,x}_s\big)\,ds
\\ 
    & 
    +
    \big[\nabla \varphi (R^{t,x}_t)\big]^T \big(\tfrac{\partial}{\partial x_i}R^{t,x}_t\big)
    \int_0^t \big[ D_x B(t-s,R^{t,x}_s)\big]\!\big(\tfrac{\partial}{\partial x_j} R^{t,x}_s\big)\,ds
    \\
    & 
    +
    \varphi(R^{t,x}_t) 
    \int_0^t 
    \big[ D_x B(t-s,R^{t,x}_s)\big]\!\big(\tfrac{\partial}{\partial x_j} R^{t,x}_s\big)\,ds
    \int_0^t 
    \big[ D_x B(t-s,R^{t,x}_s)\big]\!\big(\tfrac{\partial}{\partial x_i} R^{t,x}_s\big)\,ds
    \\
    &
    +
    \varphi(R^{t,x}_t) 
    \int_0^t 
    \big[ D^2_x B(t-s,R^{t,x}_s) \big]\!\big(\tfrac{\partial}{\partial x_i} R^{t,x}_s,\tfrac{\partial}{\partial x_j} R^{t,x}_s\big)\,ds
    \\
& 
    + 
    \varphi(R^{t,x}_t) 
    \int_0^t
    \big[ D_x B(t-s,R^{t,x}_s)\big]\!\big(\tfrac{\partial^2}{\partial x_i \partial x_j} R^{t,x}_s\big) \,ds
\end{split}
\]
for all $i,j\in \{1,2,\dots,d\}$, $t \in [0,T]$ and $x \in \R^d$.
This shows \eqref{V-smooth-smooth-x-V} and \eqref{eq:partial-x-x-V}
and completes the proof of the lemma.
\end{proof}

\begin{lemma}
\label{le:Zakai-smooth-x}
Let $T \in (0,\infty)$, $d\in \N$, $\sigma \in \R^{d\times d}$, and consider a 
function $b \in C^{0,2}([0,T]\times\R^d,\R^d)$ with bounded partial derivatives
of first and second order with respect to the $x$-variables. Let
$\varphi \in C^2(\R^d,[0,\infty))$ and $B \in C^{0,2}([0,T]\times \R^d,\R)$ 
have at most polynomially growing partial derivatives of first and 
second order with respect to the $x$-variables and assume that 
		$\sup_{(t,x)\in [0,T]\times \R^d} \frac{B(t,x)}{1+\lVert x \rVert_{\R^d}} 
		 <\infty$.
Let $(\Omega, \mathcal{F}, \P)$ be a probability space supporting a 
standard Brownian motion $U \colon [0,T]\times\Omega\to\R^d$ 
with continuous sample paths. Consider stochastic processes
$R^{t,x} \colon[0,t]\times\Omega \to \R^d$, 
	$t\in [0,T]$, $x \in \R^d$, satisfying
		\[ \label{def:mathcal-Y-smooth-x}
		R^{t,x}_s
		=
		x
		+ 
		\int_0^s
		b(t-r ,R^{t,x}_r) \,dr 
		+ 
		\sigma U_s \quad \mbox{for all
		$t \in [0,T]$, $s \in [0,t]$ and $x \in \R^d$.}
		\]
	Let the functions $\mathcal{U} \colon [0,T] \times\R^d\times \Omega \to \R$ and 
	$u \colon [0,T] \times\R^d \to \R$ be given by 
		\[
		\label{eq:mathfrak-V-x}
		\begin{split}
		&\mathcal{U}(t,x)
		= \varphi(R^{t,x}_t)\exp\!\Big({\textstyle\int\limits_0^t} B(t-s,R^{t,x}_s)\,ds\Big) 
		\quad \mbox{and} \quad u(t,x) = \E\big[\mathcal{U}(t,x)\big]
		\end{split}
		\]
		for $x\in \R^d$ and $t\in[0,T]$.
	Then
		\begin{enumerate}[{\rm (i)}]
		\item\label{mathfrak-v-smooth-smooth-x}
			$u \in C^{0,2}([0,T]\times \R^d,\R)$,
		\item\label{mathfrak-v-Dx}
			$\tfrac{\partial}{\partial x} u(t,x) = \E\big[\tfrac{\partial}{\partial x} \mathcal{U}
			(t,x)\big]$
		for all $t \in [0,T]$, $x \in \R^d$, and
		\item\label{mathfrak-v-D-2-xx}
			$\tfrac{\partial^2}{\partial x^2} u(t,x)
			= \E\big[\tfrac{\partial^2}{\partial x^2} \mathcal{U}(t,x)\big]$
	for all $t\in [0,T]$ and $x \in \R^d$.
		\end{enumerate}
\end{lemma}

\begin{proof}[Proof]
Since, by assumption, $\varphi \in C^2(\R^d,[0,\infty))$ has at most polynomially growing partial 
derivatives of first and second order, we obtain from
\cref{le:sde-solution-exponential-integrability} that 
	\begin{multline} \label{eq:Zakai-smooth-x:phi:polynomial}
	\E\Bigl[ 
		\sup\nolimits_{t \in [0, T]}
		\sup\nolimits_{s \in [0, t]}
		\sup\nolimits_{x \in [-p, p]^d} 
		\bigl( 
			|\varphi(R^{t,x}_s)|^p 
			+ \lVert \nabla \varphi(R^{t,x}_s)\rVert_{L(\R^d, \R)}^p
			\\
			+ \lVert D^2 \varphi (R^{t,x}_s)\rVert_{L^{(2)}(\R^d, \R)}^p 
		\bigr) 
		\Bigr] < \infty
	\end{multline}
	for all $p \in (0, \infty)$.
	Similarly, it follows from the assumption 
	that $B \in C^{0,2}([0,T]\times \R^d,\R)$ has at most polynomially growing 
	partial derivatives of first and second order with respect to the $x$-variables and
	\cref{le:sde-solution-exponential-integrability} that
	\begin{multline} \label{eq:Zakai-smooth-x:B:polynomial}
	\E\Bigl[ 
	\sup\nolimits_{t \in [0, T]}
	\sup\nolimits_{s \in [0, t]}
	\sup\nolimits_{x \in [-p,p]^d} 
	\bigl( 
	| B(t-s, R^{t,x}_s) |^p 
	+
	\lVert D_x B(t-s, R^{t,x}_s)\rVert_{L(\R^d, \R)}^p
	\\
	+ \lVert D^2_x B(t-s, R^{t,x}_s)\rVert_{L^{(2)}(\R^d, \R)}^p 
	\bigr) \Bigr] < \infty
	\end{multline}
for all $p \in (0, \infty)$. The assumption that
$\sup_{(t,x)\in [0,T]\times \R^d} \frac{B(t,x)}{1+\Vert x \Vert_{\R^d}}  < \infty$
and \cref{le:sde-solution-exponential-integrability} guarantee
that	\be \label{eq:Zakai-smooth-x:B:integral}
	\E\Bigl[ 
	\sup\nolimits_{t \in [0, T]}
	\sup\nolimits_{s \in [0, t]}
	\sup\nolimits_{x \in [-p,p]^d} 
	\bigl| \exp\!\big( {\textstyle \int\limits_0^s} B(t-r, R^{t,x}_r)\,dr \big) \bigr|^p \Bigr] < \infty
	\ee
	for all $p \in (0,\infty)$.
Combining \eqref{eq:Zakai-smooth-x:phi:polynomial}--\eqref{eq:Zakai-smooth-x:B:integral}
with H\"older's inequality shows that 
\[
\E \bigg[ \sup_{t \in [0, T]} \sup_{x \in \R^d} | \mathcal{U}(t, x) |^p \bigg] < \infty \quad 
\mbox{for all } p \in (0, \infty),
\]
which together with the de la Vall\'ee Poussin theorem (cf., e.g., \cite[Corollary~6.21]{Klenke_2014}),
the Vitali convergence theorem (cf., e.g., \cite[Theorem~6.25]{Klenke_2014}), and
\cref{le:V}.\eqref{V-smooth-smooth-x-V} implies that $u \in C([0,T]\times\R^d,\R)$.
Next, we note that \eqref{eq:Zakai-smooth-x:phi:polynomial}--\eqref{eq:Zakai-smooth-x:B:integral},    \cref{le:V}.\eqref{eq:partial-x-V}, \cref{le:Y-differentiability-properties}.\eqref{eq:delaVaPou-x-cY}
and H\"older's inequality yield that 
	\[
	\E\Bigl[ 
	\sup\nolimits_{t \in [0, T]} 
	\sup\nolimits_{x \in [-p,p]^d}
	\lVert \tfrac{\partial}{\partial x} \mathcal{U}(t, x) \rVert_{L(\R^d, \R)}^p 
	\Bigr] 
	< \infty
	\]
	for all $p \in (0, \infty)$. Therefore, one obtains from
\cref{le:V}.\eqref{V-smooth-smooth-x-V},
the de la Vall\'ee Poussin theorem (cf., e.g., \cite[Corollary~6.21]{Klenke_2014}),
the Vitali convergence theorem (cf., e.g., \cite[Theorem~6.25]{Klenke_2014}), 
and
the fundamental theorem of calculus that

\begin{enumerate}[(a)]
	\item\label{eq:v-C-0-1-Zakai} 
for all $t \in [0,T]$, the mapping
$x \mapsto u(t,x)$ is in $C^1(\R^d,\R)$,
	\item \label{eq:v-C-0-1-b-Zakai}
for all $i \in \{1,2,\dots,d\}$, the mapping
$(t,x) \mapsto \tfrac{\partial}{\partial x_i} u(t,x)$
is in $C([0,T]\times \R^d,\R)$, and
\item \label{eq:v-C-0-1-x} 
$\tfrac{\partial}{\partial x} u(t,x)=\E[\tfrac{\partial}{\partial x} 
\mathcal{U}(t,x)]$
for all $x\in \R^d$ and $t \in [0,T]$.
\end{enumerate}
This establishes \eqref{mathfrak-v-Dx}. 
	
Next, note that it follows from \eqref{eq:Zakai-smooth-x:phi:polynomial}--\eqref{eq:Zakai-smooth-x:B:integral}, \cref{le:V}.\eqref{eq:partial-x-x-V}, items~\eqref{eq:delaVaPou-x-cY} and \eqref{eq:delaVaPou-x-x-cY} of \cref{le:Y-differentiability-properties} and H\"older's inequality that
	\[
	\E\biggl[ 
		\sup_{t \in [0,T]} 
		\sup_{x \in [-p,p]} 
		\lVert \tfrac{\partial^2}{\partial x^2} \mathcal{U}(t, x) \rVert_{L^{(2)}(\R^d, \R)}^p 
	\biggr] < \infty
	\]
for all $p \in (0, \infty)$. Hence, one obtains from 
\cref{le:V}.\eqref{V-smooth-smooth-x-V}, \eqref{eq:v-C-0-1-Zakai}--\eqref{eq:v-C-0-1-b-Zakai} above,
the de la Vall\'ee Poussin theorem (cf., e.g., \cite[Corollary~6.21]{Klenke_2014}),
the Vitali convergence theorem (cf., e.g., \cite[Theorem~6.25]{Klenke_2014}),
and the fundamental theorem of calculus that
\begin{enumerate}[(A)]
	\item\label{eq:v-C-0-2-Zakai} 
	for all $t \in [0,T]$, the mapping 
		$x \mapsto u(t,x)$ is in $C^2(\R^d,\R)$,
	\item \label{eq:v-C-0-2-b-Zakai}
	for all $i,j \in \{1, 2, \dots, d\}$, the mapping
		$(t,x)\mapsto \tfrac{\partial^2}{\partial x_i  \partial x_j} u(t,x)\in \R$
		is in $C([0,T]\times \R^d,\R)$,
	\item \label{eq:v-C-0-2-x} 
		$\tfrac{\partial^2}{\partial x^2} u(t,x)
		=\E[\tfrac{\partial^2}{\partial x^2} \mathcal{U}(t,x)]$ for all $x\in \R^d$  and
		$t \in [0,T]$.
\end{enumerate}
\eqref{eq:v-C-0-2-x} directly establishes \eqref{mathfrak-v-D-2-xx}. 
Moreover, \eqref{eq:v-C-0-1-Zakai}--\eqref{eq:v-C-0-1-x}, 
\eqref{eq:v-C-0-2-Zakai}--\eqref{eq:v-C-0-2-x}, and the fact 
that $u \in C([0,T] \times \R^d, \R)$ 
imply \eqref{mathfrak-v-smooth-smooth-x}, which completes the proof of the lemma.
\end{proof}

\subsection{Feynman--Kac representations for linear PDEs} \label{subsec:FK}

The following lemma is a stepping stone towards the proof of the Feynman--Kac 
representation in \cref{prop:Zakai-approx} below. It makes stronger regularity 
assumptions on the coefficients.
\cref{prop:Zakai-approx} can be derived from it by mollifying the coefficients.

\begin{lemma} \label{le:Zakai-Smooth}
	Let $T \in (0,\infty)$, 
	$d \in \N$, $\sigma \in \R^{d\times d}$, and consider a function
	$b \in C^{1,2}([0,T]\times\R^d,\R^d)$ with a bounded partial derivative
	with respect to $t$ and bounded partial derivatives of first and second order with 
	respect to the $x$-variables. Let $\varphi \in C^3(\R^d,[0,\infty))$ have at most 
	polynomially growing partial derivatives of first, second and third order.
	Let $B \in C^{1,2}([0,T]\times \R^d,\R)$ have an at most polynomially 
	growing partial derivative with respect to $t$ and at most polynomially 
	growing partial derivatives of first and 
	second order with respect to the $x$ variables. In addition, assume that 
	$\sup_{(t,x)\in [0,T]\times \R^d} \frac{B(t,x)}{1+\Vert x \Vert_{\R^d}} < \infty$,
	and let $(\Omega, \mathcal{F}, \P)$ be a probability space supporting a 
	standard Brownian motion $U \colon [0,T]\times\Omega\to\R^d$ 
	with continuous sample paths. Consider stochastic processes
	$R^{t,x}\colon[0,t]\times\Omega \to \R^d$, $t\in [0,T]$, $x \in \R^d$, 
	satisfying
	\begin{equation}\label{def:mathcal-Y-Smooth}
	R^{t,x}_s
	=
	x
	+ 
	\int_0^s
	b(t-r,R^{t,x}_r) \,dr 
	+ 
	\sigma U_s \quad \mbox{for all 
	$t\in[0,T]$, $s \in [0,t]$ and $x \in \R^d$,}
	\end{equation}
	and let the mappings $\mathcal{U} \colon[0,T]\times \R^d \times \Omega \to \R$,
	$u \colon [0,T]\times\R^d \to \R$ be given by
	\begin{equation}\label{mathfrak-V-Smooth}
	\mathcal{U}(t,x)=\varphi(R^{t,x}_t)\exp\!\Big({\textstyle\int\limits_0^t} B(t-s,R^{t,x}_s)\,ds\Big) \, \mbox{ and } \, u(t,x) = \E\big[\mathcal{U}(t,x)\big] 
	\end{equation}
for $t \in [0,T]$ and $x \in \R^d$.
	Then
	\begin{enumerate}[{\rm (i)}]
		\item\label{mathfrak-v-Smooth} 
		$u \in C^{1,2}([0,T]\times \R^d,\R)$, and
		\item \label{PDE-Smooth-item}
		$\mbox{}$\\[-8.6mm]
		\[ \label{PDE-Smooth}
		\begin{split}
		&u(t,x)
		\\
		&=
		\varphi(x)
		+ 
		\int_0^t
		\bigg[ 
		\tfrac{1}{2} \operatorname{Trace}_{\R^d}\!\big( \sigma\sigma^{T}\operatorname{Hess}_x (u)(s,x) \big)
		+
		\left\langle b(s,x), \nabla_x u(s,x)\right\rangle_{\Rd}
		+ B(s,x) u(s,x)
		\bigg] \,ds
		\end{split}
		\]
		for all $t \in [0,T]$ and $x \in \R^d$.
	\end{enumerate}
\end{lemma}

\begin{proof}[Proof]
Consider the mapping $\mathcal{V}\colon [0,T]\times \R^d\times \Omega \to \R$ given by
	\be \label{eq:Ito-cV}
	\begin{split}
	\mathcal{V}(t,x) = \varphi(x)
	+ 
	\int_0^t 
	\exp\!\Big({\textstyle\int\limits_0^s} B(t-r,R^{t,x}_r)\,dr \Big) 
	\biggl[
	\tfrac{1}{2} \operatorname{Trace}_{\R^d}\!\big( \sigma \sigma^{T} 
	\operatorname{Hess}(\varphi)(R^{t,x}_s) \big) 
	\\
	+
	\langle \nabla\varphi(R^{t,x}_s), b(t-s,R^{t,x}_s) \rangle_{ \R^d }
	+ 	
	B(t-s,R^{t,x}_s) \varphi(R^{t,x}_s) 
	\biggr] \,ds \quad \mbox{for $t \in [0,T]$ and $x \in \R^d$.}
	\end{split}
	\ee
	It follows from \eqref{def:mathcal-Y-Smooth} by It\^o's formula that
	\be
	\label{eq:Ito-cY}
	\begin{split}
	\varphi(R^{t,x}_s)
	&= 
	\varphi(x)
	+ 
	\int_0^s \langle \nabla \varphi(R^{t,x}_r), \sigma\, dU_r \rangle_{ \R^d } \\
	& 
	+ 
	\int_0^s \Big[ \langle \nabla \varphi(R^{t,x}_r), b(t-r, R^{t,x}_r) \rangle_{ \R^d }
	+
	\tfrac{1}{2} \operatorname{Trace}_{\R^d}
	\!\big(\sigma \sigma^{T} \operatorname{Hess} (\varphi)(R^{t,x}_r)\big) \Big] \,dr 
	\quad \mbox{$\P$-a.s.}
	\end{split}
	\ee
for all $t \in [0,T]$, $s \in [0,t]$ and $x \in \R^d$.
	In addition, one has
	\begin{equation}
	\label{eq:Ito-fB}
	\begin{split}
	\exp\!\Big({\textstyle\int\limits_0^s} B(t-r,R^{t,x}_r)\,dr \Big) 
	= 1 + \int_0^s \exp\!\Big({\textstyle\int\limits_0^r} B(t-w,R^{t,x}_w)\,dw\Big)\, 
	B(t-r,R^{t,x}_r)\,dr 
	\end{split}
	\end{equation}
for all $t \in [0,T]$, $s \in [0,t]$ and $x \in \R^d$.
\eqref{mathfrak-V-Smooth}--\eqref{eq:Ito-fB} and another application of
It\^o's formula give	
\be \label{eq:Ito-V}
	\begin{split}
	\mathcal{U}(t,x)
	&=
	\varphi(R^{t,x}_t) \exp\!\Big({\textstyle\int\limits_0^t} B(t-s,R^{t,x}_s)\,ds\Big)\\
	&= \varphi(x)
	+ 
	\int_0^t \exp\!\Big({\textstyle\int\limits_0^s} B(t-r,R^{t,x}_r)\,dr\Big)
	\langle \nabla\varphi (R^{t,x}_s), \sigma\,dU_s \rangle_{\R^d}
	\\
	& \quad + 
	\int_0^t\exp\!\Big({\textstyle\int\limits_0^s} B(t-r,R^{t,x}_r)\,dr \Big) 
	\langle \nabla\varphi(R^{t,x}_s), b(t-s, R^{t,x}_s) \rangle_{ \R^d } \,ds
	\\
	& \quad
	+
	\tfrac{1}{2} \int_0^t 
	\exp\!\Big({\textstyle\int\limits_0^s} B(t-r,R^{t,x}_r)\,dr\Big)
	\operatorname{Trace}_{\R^d}\!\big( \sigma \sigma^{T}  
	\operatorname{Hess}(\varphi)(R^{t,x}_s)\big) \,ds\\
	& \quad + \int_0^t 
	\exp\!\Big({\textstyle\int\limits_0^s} B(t-r,R^{t,x}_r)\,dr \Big)
	B(t-s,R^{t,x}_s)
	\varphi(R^{t,x}_s) \,ds
	\\
	&= 
	\mathcal{V}(t,x) 
	+ 	
	\int_0^t \exp\!\Big({\textstyle\int\limits_0^s} B(t-r,R^{t,x}_r)\,dr \Big)
	\langle \nabla\varphi(R^{t,x}_s), \sigma \,dU_s \rangle_{ \R^d } \quad 
	\mbox{$\P$-a.s.}
	\end{split}
	\ee
	for all $t \in [0,T]$ and $x \in \R^d$.
Now, observe that \cref{le:sde-solution-exponential-integrability} and the assumption 
that $\varphi \in C^3(\R^d,[0,\infty))$ has at most polynomially growing partial derivatives 
of first, second and third order imply that 
	\begin{multline} \label{eq:Zakai-Smooth:phi:polynomial}
	\E\Bigl[ 
	\sup\nolimits_{t \in [0, T]}
	\sup\nolimits_{s \in [0, t]}
	\sup\nolimits_{x \in [-p, p]^d} 
	\bigl( |\varphi(R^{t,x}_s)|^p + \lVert \nabla
	\varphi (R^{t,x}_s)\rVert_{L(\R^d, \R)}^p
	\\
	+ \lVert D^2 \varphi (R^{t,x}_s)\rVert_{L^{(2)}(\R^d, \R)}^p 
	+ \lVert D^3 \varphi (R^{t,x}_s)\rVert_{L^{(3)}(\R^d, \R)}^p 
	\bigr) \Bigr] < \infty \quad \mbox{for all $p \in (0, \infty)$.}
	\end{multline}
Moreover, \cref{le:sde-solution-exponential-integrability} and the assumption that 
the partial derivative of $B \in C^{1,2}([0,T]\times \R^d,\R)$
with respect to $t$ as well as its first and second order partial 
derivatives with respect to the $x$-variables are at most polynomially growing ensure that
	\begin{multline} \label{eq:Zakai-Smooth:B:polynomial}
	\E\Bigl[ 
	\sup\nolimits_{t \in [0, T]}
	\sup\nolimits_{s \in [0, t]}
	\sup\nolimits_{x \in [-p,p]^d} 
	\Big( | B(t-s, R^{t,x}_s) |^p + | \tfrac{\partial}{\partial t}B (t-s, R^{t,x}_s) |^p 
	\\
	+ \lVert D_x B (t-s, R^{t,x}_s)\rVert_{L(\R^d, \R)}^p
	+ \lVert D^2_x B (t-s, R^{t,x}_s)\rVert_{L^{(2)}(\R^d, \R)}^p \Big)
	\Bigr] < \infty \quad \mbox{for all $p \in (0,\infty)$.}
	\end{multline} 
Similarly, \cref{le:sde-solution-exponential-integrability} and the assumption 
$\sup_{(t,x)\in [0,T]\times \R^d} \frac{B(t,x)}{1+\Vert x \Vert_{\R^d}} < \infty$ imply that	
\begin{equation} \label{eq:Zakai-Smooth:exp:polynomial}
	\E\biggl[ 
	\sup\nolimits_{t \in [0, T]}
	\sup\nolimits_{s \in [0, t]}
	\sup\nolimits_{x \in [-p,p]^d} 
	\Bigl| \exp\!\Big( {\textstyle \int\limits_0^s} B(t-r, R^{t,x}_r)\,dr \Big) \Bigr|^p 
	\biggr] < \infty \quad \mbox{for all $p \in (0, \infty)$.}
	\end{equation} 
From \eqref{eq:Zakai-Smooth:phi:polynomial}, \eqref{eq:Zakai-Smooth:exp:polynomial}
and H\"older's inequality one obtains
	\begin{equation} \label{eq:Zakai-Smooth:L2-martingale}
	\E\bigg[\int_0^t \Big\lVert \exp\!\Big({\textstyle\int\limits_0^s} 
	B(t-r,R^{t,x}_r)\,dr \Big)
	\sigma^{T} \nabla\varphi (R^{t,x}_s)\Big\rVert_{\R^d}^2\,ds \bigg]<\infty
	\quad \mbox{for all $t \in [0,T]$ and $x \in \R^d$.}
	\end{equation}
Next, note that it follows from the assumptions on $b \in C^{1,2}([0,T]\times \R^d,\R^d)$
that it grows at most linearly. Therefore, we obtain from 
\eqref{mathfrak-V-Smooth} and 
\eqref{eq:Ito-V}--\eqref{eq:Zakai-Smooth:L2-martingale} together with Fubini's theorem that
	\begin{align} \notag
	u(t,x)
	&=
	\E\big[\mathcal{U}(t,x)\big]
	= \E\big[\mathcal{V}(t,x)\big]
	\\ \notag
	&= \varphi(x)
	+ 
	\int_0^t
	\E\Bigl[
	\exp\!\Big({\textstyle\int\limits_0^s} B(t-r,R^{t,x}_r)\,dr\Big) 
	\langle \nabla \varphi (R^{t,x}_s), b(t-s, R^{t,x}_s) \rangle_{\R^d}
	\Bigr]
	\,ds\\ \label{eq:Ito-V-Fubini}
	& 
	+
	\tfrac{1}{2} \int_0^t 
	\E\Bigl[\exp\!\Big({\textstyle\int\limits_0^s} B(t-r,R^{t,x}_r)\,dr\Big)
	\operatorname{Trace}_{\R^d}\!\big(\sigma\sigma^{T}
	\operatorname{Hess}(\varphi)(R^{t,x}_s)\big)\Bigr] 
	\,ds\\ \notag
	& 
	+ \int_0^t 
	\E\Bigl[
	\exp\!\Big({\textstyle\int\limits_0^s} B(t-r,R^{t,x}_r)\,dr \Big)\, 
	B(t-s,R^{t,x}_s) \varphi(R^{t,x}_s)
	\Bigr]
	\,ds \quad \mbox{for all $t \in [0,T]$ and $x \in \R^d$.}
	\end{align}
It follows from \eqref{def:mathcal-Y-Smooth}, the assumption that 
$b \in C^{1,2}([0,T]\times\R^d,\R^d)$ has bounded partial derivatives, 
the de la Vall\'ee Poussin theorem (cf., e.g., \cite[Corollary~6.21]{Klenke_2014}),
and the Vitali convergence theorem (cf., e.g., \cite[Theorem~6.25]{Klenke_2014})
that for all $\omega \in \Omega$ and $s \in [0,T]$, the mapping 
$(t,x) \mapsto R^{t,x}_s$ is in $C^{1,0}([s,T] \times\R^d,\R^d)$ with
	\be
	\label{eq:partial-t-mathcalY}
	\tfrac{\partial}{\partial t} R^{t,x}_s
	=
	\int_0^s \Big[ \tfrac{\partial}{\partial t} b (t-r, R^{t,x}_r) 
	+ \big[ D_x b (t-r, R^{t,x}_r)\big]\!\big(\tfrac{\partial}{\partial t} 
	R^{t,x}_r \big) \Big] \,dr
	\ee
	(cf.\ also, e.g., \cite[Theorem~4.6.5]{kunita1997stochastic}).
	This and the assumption that $b \in C^{1,2}([0,T]\times\R^d,\R^d)$ 
	has bounded partial derivatives imply that
	\[
	\begin{split}
	& \lVert \tfrac{\partial}{\partial t} R^{t,x}_s \rVert_{\R^d}
	\leq 
	T \sup\nolimits_{(r, y) \in [0,T] \times \R^d} \Vert \tfrac{\partial}{\partial t} b(r, y)\Vert_{\R^d} 
	\\
	& + \sup\nolimits_{(r, y)\in [0,T]\times \R^d}\Vert \tfrac{\partial}{\partial x} b(r, y)\Vert_{L(\R^d,\R^d)} 
	\int_0^s 
	\lVert \tfrac{\partial}{\partial t}R^{t,x}_r \rVert_{\R^d} \,dr
	\quad \mbox{for all $t \in [0,T]$, $s \in [0,t]$ and $x \in \R^d$,}
	\end{split} 
	\]
	which, by Gronwall's integral inequality (cf., e.g., \cite[Lemma~2.11]{grohs2018proof}), yields
	\be \label{eq:delaVaPou-t-cY}
	\begin{split}
	&
	\sup\nolimits_{t \in [0,T]} 
	\sup\nolimits_{s \in [0,t]} 
	\sup\nolimits_{x \in \R^d} \lVert\tfrac{\partial}{\partial t} R^{t,x}_s\rVert_{\R^d}
	\\
	&\leq 
	T \bigl[ \sup\nolimits_{(r, y) \in [0,T]\times \R^d}\Vert \tfrac{\partial}{\partial t}b (r, y)\Vert_{\R^d} \bigr]
	\exp\Bigl( T \sup\nolimits_{(r, y)\in [0,T]\times \R^d}\Vert \tfrac{\partial}{\partial x}b (r, y)\Vert_{L(\R^d,\R^d)}
	\Bigr)
	<\infty.
	\end{split}
	\ee
	Next observe that \eqref{eq:partial-t-mathcalY}, 
	\cref{le:sde-solution-exponential-integrability},
	our assumptions on $b, \varphi, B$,
	the de la Vall\'ee Poussin theorem (cf., e.g., \cite[Corollary~6.21]{Klenke_2014}),
	and 
	the Vitali convergence theorem (cf., e.g., \cite[Theorem~6.25]{Klenke_2014})
	imply that
	\begin{enumerate}[(a)]
		\item \label{it:Zakai-Smooth:partial-t-V-1} $\mbox{}$\\[-8mm]
		\[
		\label{eq:partial-t-V-1}
		\begin{split}
		&\frac{d}{dt}\bigg(\int_0^t
		\exp\!\Big({\textstyle\int\limits_0^s} B(t-r,R^{t,x}_r)\,dr\Big) 
		\big[\tfrac{\partial}{\partial x}\varphi (R^{t,x}_s)\big]\!\big(b(t-s,R^{t,x}_s)\big) \,ds\bigg)\\
		&=
		\exp\!\Big({\textstyle\int\limits_0^t} B(t-r,R^{t,x}_r)\,dr\Big) 
		\big[ \tfrac{\partial}{\partial x}\varphi (R^{t,x}_t)\big]\!\big(b(0,R^{t,x}_t)\big)\\
		& 
		+
		\int_0^t 
		\exp\!\Big({\textstyle\int\limits_0^s} B(t-r,R^{t,x}_r)\,dr\Big) 
		\big[ \tfrac{\partial}{\partial x}\varphi (R^{t,x}_s)\big]\!\big(b(t-s,R^{t,x}_s)\big)
		\\
		& \quad
		\times
		{\textstyle \int\limits_0^s}
		\Big[
		\tfrac{\partial}{\partial t} B(t-r,R^{t,x}_r)
		+ 
		\big[\tfrac{\partial}{\partial x}B(t-r,R^{t,x}_r)\big]\!\big(\tfrac{\partial}{\partial t}
		R^{t,x}_r \big) \Big] \, dr \, ds
		\\
		& 
		+
		\int_0^t 
		\exp\!\Big({\textstyle\int\limits_0^s} B(t-r,R^{t,x}_r)\,dr \Big)
		\big[\tfrac{\partial^2}{\partial x^2}\varphi(R^{t,x}_s) \big]\!\big(\tfrac{\partial}{\partial t}
		R^{t,x}_s, \fb(t-s,R^{t,x}_s) \big)
		\,ds\\
		&  
		+ \int_0^t
		\exp\!\Big({\textstyle\int\limits_0^s} B(t-r,R^{t,x}_r)\,dr \Big)
		\\
		& \quad \times
		\big[\tfrac{\partial}{\partial x}\varphi(R^{t,x}_s)\big]\!\Big[
		\tfrac{\partial}{\partial t} b(t-s,R^{t,x}_s)
		+
		\big[\tfrac{\partial}{\partial x} b(t-s,\cY^{t,x}_s) \big]
		(\tfrac{\partial}{\partial t}R^{t,x}_s)
		\Big] \,ds
		\end{split}
		\]
	for all $t \in [0,T]$ and $x \in \R^d$,
		\item \label{it:Zakai-Smooth:partial-t-V-2} $\mbox{}$\\[-9mm]
		\[
		\label{eq:partial-t-V-2}
		\begin{split}
		&\frac{d}{dt}\bigg(	\int_0^t 
		\exp\!\Big({\textstyle\int\limits_0^s} B(t-r,R^{t,x}_r)\,dr \Big)
		\operatorname{Trace}_{\R^d}
		\bigl( \sigma \sigma^{T} \operatorname{Hess} (\varphi)(R^{t,x}_s)\bigr)
		\,ds\bigg)\\
		&
		=
		\exp\!\Big({\textstyle\int\limits_0^t} B(t-r,R^{t,x}_r)\,dr\Big)
		\operatorname{Trace}_{\R^d}
		\bigl( \sigma \sigma^{T} \operatorname{Hess}(\varphi)(R^{t,x}_t) \bigr)\\
		& 
		+
		\int_0^t
		\exp\!\Big({\textstyle\int\limits_0^s} B(t-r,R^{t,x}_r)\,dr \Big)
		\operatorname{Trace}_{\R^d}\bigl( \sigma \sigma^{T} 
		\operatorname{Hess}(\varphi)(R^{t,x}_s)\bigr)\\
		& \quad 
		\times
		\bigg[
		{\textstyle \int\limits_0^s}
		\Big( \tfrac{\partial}{\partial t}B(t-r,R^{t,x}_r)
		+ 
		\bigl[\tfrac{\partial}{\partial x}B(t-r,R^{t,x}_r)\bigr]\!\bigl(\tfrac{\partial}{\partial t}
		R^{t,x}_r\bigr)
		\, \Big) \, dr \bigg] \,ds \\
		& 
		+		
		\int_0^t
		\exp\!\Big({\textstyle\int\limits_0^s} B(t-r,R^{t,x}_r)\,dr \Big) 
		\operatorname{Trace}_{\R^d}\bigl(
		\sigma \sigma^{T}
		\bigl[\tfrac{\partial}{\partial x}\operatorname{Hess}(\varphi)(R^{t,x}_s)\bigr]\!
		\bigl(\tfrac{\partial}{\partial t}R^{t,x}_s\bigr)
		\bigr) \,ds
		\end{split}
		\]
		for all $t \in [0,T]$ and $x \in \R^d$, and
		\item \label{it:Zakai-Smooth:partial-t-V-3} $\mbox{}$\\[-9mm]
		\[
		\label{eq:partial-t-V-3}
		\begin{split}
		&\frac{d}{dt}\bigg(
		\int_0^t 
		\exp\!\Big({\textstyle\int\limits_0^s} B(t-r,R^{t,x}_r)\,dr \Big)\,
		B(t-s,R^{t,x}_s)\, \varphi(R^{t,x}_s) \,ds
		\bigg)\\
		&=
		\exp\!\Big({\textstyle\int\limits_0^t} B(t-r,R^{t,x}_r)\,dr\Big) 
		B(0, R^{t,x}_t)\,\varphi(R^{t,x}_t)\\
		& 
		+
		\int_0^t
		\exp\!\Big({\textstyle\int\limits_0^s} B(t-r,R^{t,x}_r)\,dr\Big)
		B(t-s,R^{t,x}_s)\,\varphi(R^{t,x}_s) \\
		& \quad
		\times
		{\textstyle \int\limits_0^s}
		\Big[
		\tfrac{\partial}{\partial t}B (t-r,R^{t,x}_r)
		+ 
	\bigl[\tfrac{\partial}{\partial x}B(t-r,R^{t,x}_r)\bigr]\!\bigl(\tfrac{\partial}{\partial t}
	R^{t,x}_r \bigr) \Big] \, dr \,ds\\
		& 
		+ \int_0^t
		\exp\!\Big({\textstyle\int\limits_0^s} B(t-r,R^{t,x}_r)\,dr\Big)
		\varphi(R^{t,x}_s)\\
		& \quad
		\times
		\bigg[
		\tfrac{\partial}{\partial t} B(t-s,R^{t,x}_s)
		+
		\bigl[\tfrac{\partial}{\partial x} B(t-s, R^{t,x}_s)\bigr]\!
		\bigl(\tfrac{\partial}{\partial t} R^{t,x}_s\bigr)
		\bigg]
		\,ds\\
		& 
		+
		\int_0^t 
		\exp\!\Big({\textstyle\int\limits_0^s} B(t-r,R^{t,x}_r)\,dr\Big)
		B(t-s,\cY^{t,x}_s)
		\bigl[
		\tfrac{\partial}{\partial x}\varphi(R^{t,x}_s)\bigr]\!
		\bigl(\tfrac{\partial}{\partial t}R^{t,x}_s \bigr)
		\,ds
		\end{split}
		\]
		for all $t \in [0,T]$ and $x \in \R^d$.
	\end{enumerate}
	Combining \eqref{eq:Ito-cV},   \eqref{eq:Zakai-Smooth:phi:polynomial}--\eqref{eq:Zakai-Smooth:exp:polynomial}, \eqref{eq:delaVaPou-t-cY} and the assumption that $b \in C^{1,2}([0,T]\times\R^d, \R^d)$ has bounded partial derivatives with \eqref{it:Zakai-Smooth:partial-t-V-1}--\eqref{it:Zakai-Smooth:partial-t-V-3} and H\"older's inequality shows that  
	\[ \label{eq:delaVaPou-t-V}
	\E\Big[
	\sup\nolimits_{t\in [0,T]}
	\sup\nolimits_{x\in [-p,p]^d} 
	| \tfrac{\partial}{\partial t} \mathcal{V}(t,x) |^p
	\Big]<\infty \quad \mbox{for all } p \in (0, \infty).
	\]
	This, \eqref{eq:Ito-V-Fubini}, 
	\eqref{it:Zakai-Smooth:partial-t-V-1}--\eqref{it:Zakai-Smooth:partial-t-V-3},
	the de la Vall\'ee Poussin theorem (cf., e.g., \cite[Corollary~6.21]{Klenke_2014}),	
	the Vitali convergence theorem (cf., e.g., \cite[Theorem~6.25]{Klenke_2014}), 
	and the fundamental theorem of calculus ensure that
	\begin{enumerate}[(A)]
\item\label{eq:v-C-1-0-a-Zakai} for all $x \in \R^d$, the mapping
		$t \mapsto u(t,x)$ is in $C^{1}([0,T],\R)$,
		
		\item\label{eq:v-C-1-0-b-Zakai} 
		the mapping 
		$(t,x)\mapsto \frac{\partial}{\partial t}u(t,x)$ is in
		$C([0,T]\times \R^d,\R)$,
		and
		\item\label{eq:v-C-1-0-c-Zakai} 
		$\frac{\partial}{\partial t}u(t,x) = 	\E[\tfrac{\partial}{\partial t}\mathcal{V}(t,x)]$
		for all $t \in [0,T]$ and $x \in \R^d$,
	\end{enumerate}
	which, together with \cref{le:Zakai-smooth-x}.\eqref{mathfrak-v-smooth-smooth-x},
	implies \eqref{mathfrak-v-Smooth}.
	
	Next, note that it follows from the Markov property of $(R^{t,x}_s)_{s \in [0,t]}$ that
	\[
	\E \Big[G((R^{t,x}_{\delta+s})_{s\in [0,t-\delta]})\mathbbm{1}_{A}(R^{t,x}_{\delta}) \Big] 
	= \int_A \E \Big[G((R^{t-\delta,y}_{s})_{s \in [0,t- \delta]}) \Big] R^{t,x}_{\delta}(\P)(dy)
	\]
	for all Borel subsets $A \subseteq \R^d$, $t \in [0, T] $, $\delta \in [0,t]$, $x \in \R^d$
	and bounded functions $G \in C(C([0,t- \delta],\R^d),\R)$, which together with
	\eqref{mathfrak-V-Smooth}, implies that
	\begin{equation} \label{eq:Markov}
	\begin{split}
	u(t,x) 
	&=\E\bigl[ \mathcal{U}(t, x) \bigr]
	=\E\Bigl[\varphi(\cY^{t,x}_t)\exp\!\Big({\textstyle{\int\limits_0^t}} 
	B(t-s,R^{t,x}_s)\,ds \Big)\Bigr]\\
	&=\E\biggl[
	\exp\!\Big({\textstyle{\int\limits_0^{\delta}}}B(t-s,R^{t,x}_s)\,ds\Big)
	\E\Bigl[\varphi(R^{t,x}_t) \exp\!\Big({\textstyle{\int\limits_{\delta}^t}} B(t-s,R^{t,x}_s)\,ds\Big)
	\,\Big| \, (R^{t,x}_s)_{0 \le s \le \delta} 
	\Bigr]
	\biggr]\\
	&=\E\biggl[
	\exp\!\Big({\textstyle{\int\limits_0^{\delta}}} B(t-s, R^{t,x}_s)\,ds\Big)
	\E\Bigl[\varphi(R^{t,x}_t) \exp\!\Big({\textstyle{\int \limits_{\delta}^t}} B(t-s, R^{t,x}_s)\,ds\Big)
	\,\Big| \, R^{t,x}_{\delta} 
	\Bigr]
	\biggr]\\
	&=\E\biggl[
	\exp\!\Big({\textstyle{\int\limits_0^{\delta}}} B(t-s,R^{t,x}_s)\,ds\Big)
	\E
	\Bigl[\varphi(R^{t-\delta,y}_{t- \delta})\exp\!\Big({\textstyle{\int\limits_0^{t-\delta}}} 
	B(t - \delta -s, R^{t-\delta,y}_s)\,ds\Big)
	\Bigr]\Big|_{R^{t,x}_{\delta} = y}
	\biggr]\\
	&=\E\biggl[ \exp\!\Big({\textstyle{\int\limits_0^{\delta}}} 
	B(t-s, R^{t,x}_s)\,ds\Big) u(t-\delta, R^{t,x}_{\delta}) \biggr].
	\end{split}
	\end{equation}
	Furthermore, \eqref{def:mathcal-Y-Smooth}, \eqref{mathfrak-v-Smooth}, 
	and It\^o's formula assure that 
	\[
	\begin{split}
	u \big(t- \delta, R^{t,x}_{\delta} \big)
	&=
	u(t,x)
	- 
	\int_0^{\delta} \tfrac{\partial}{\partial t} u(t-s, R^{t,x}_s)\,ds
	+ 
	\int_0^{\delta} \langle \nabla_x u(t-s,R^{t,x}_s), b(t-s,R^{t,x}_s) \rangle_{ \R^d } \,ds\\
	&  
	+
	\int_0^{\delta} \langle \nabla_x u(t-s, R^{t,x}_s), \sigma\,dU_s \rangle_{ \R^d }
	+
	\tfrac{1}{2}\int_0^{\delta} \operatorname{Trace}_{\R^d}\bigl(\sigma \sigma^{T} 
	\operatorname{Hess}_x (u)(t-s, R^{t,x}_s)\bigr)
	\,ds \quad \mbox{$\P$-a.s.}
	\end{split}
	\]
	for all $t\in[0,T]$, $\delta \in [0,t]$ and $x\in \R^d$.
	Combining this, \eqref{eq:Ito-fB}, and It\^o's formula gives
		\begin{equation}
	\label{eq:Ito-PDE-derivation}
	\begin{split}
	&\exp\!\Big({\textstyle{\int\limits_0^{\delta}}} B(t-s, R^{t,x}_s)\,ds\Big) u\big(t- \delta,R^{t,x}_h\big)\\
	&= u(t,x)
	-\int_0^{\delta}
	\exp\!\Big({\textstyle{\int\limits_0^s}} B(t-r, R^{t,x}_r)\,dr \Big)
	\tfrac{\partial}{\partial t} u(t-s, R^{t,x}_s)\,ds\\
	& 
	+\int_0^{\delta}
	\exp\!\Big({\textstyle{\int\limits_0^s}} B(t-r, R^{t,x}_r)\,dr \Big) 
	\langle \nabla_x u(t-s, R^{t,x}_s) ,b(t-s, R^{t,x}_s) \rangle_{ \R^d }
	\,ds\\
	&  
	+\int_0^{\delta}
	\exp\!\Big({\textstyle{\int\limits_0^s}} B(t-r,R^{t,x}_r)\,dr \Big) 
	\langle \nabla_x u(t-s, R^{t,x}_s), 
	\sigma\,dU_s \rangle_{ \R^d } \\
	&
	+\tfrac{1}{2}
	\int_0^{\delta} 	\exp\!\Big({\textstyle{\int\limits_0^s}} B(t-r,R^{t,x}_r)\,dr \Big)  
	\operatorname{Trace}_{\R^d}\bigl(\sigma \sigma^{T}
	\operatorname{Hess}_x (u)(t-s,R^{t,x}_s) \bigr) \,ds\\
	&  
	+
	\int_0^{\delta}
	\exp\!\Big({\textstyle{\int\limits_0^s}} B(t-r , R^{t,x}_r)\,dr \Big) 
	B(t-s,\cY^{t,x}_s) \, u(t-s, R^{t,x}_s)
	\,ds \quad \mbox{$\P$-a.s.}
	\end{split}
	\end{equation}
	for all $t\in[0,T]$, $\delta \in [0,t]$ and $x \in \R^d$.
	Moreover, it follows from \cref{le:sde-solution-exponential-integrability}, 
	\cref{le:V}, \cref{le:Zakai-smooth-x}, \eqref{it:Zakai-Smooth:partial-t-V-1}--\eqref{it:Zakai-Smooth:partial-t-V-3}, 
	\eqref{eq:v-C-1-0-a-Zakai}--\eqref{eq:v-C-1-0-c-Zakai}, and our assumptions on 
	$b$, $\varphi$ and $B$ that
	 	\begin{multline} \label{eq:Zakai-Smooth:growth_v}
	\E\biggl[ 
	\sup_{t \in [0, T]} 
	\sup_{s \in [0, t]} 
	\sup_{x \in [-p,p]^d} 
	\Big( |u(t - s, R^{t,x}_s)|^p
	+ \lVert \nabla_x u(t - s, R^{t,x}_s) \rVert_{L(\R^d ,\R)} 
	\\
	+ \lVert D^2_x u(t - s, R^{t,x}_s) \rVert_{L^{(2)}(\R^d, \R)} 
	\Big) \biggr] < \infty
	\quad \mbox{for all $p \in (0, \infty)$.}	
	\end{multline} 
	This ensures that
	\[
	\E\bigg[\int_0^t\Big\lVert
	\exp\!\Big({\textstyle{\int\limits_0^s}} B(t-r, R^{t,x}_r)\,dr\Big) 
	\sigma^{T} \nabla_x u(t-s,R^{t,x}_s) \Big\rVert_{\R^d}^2
	\,ds\Bigg]<\infty \quad \mbox{for all $t \in [0, T]$ and $x \in \R^d$,}
	\]
	which, together with \eqref{eq:Markov}, \eqref{eq:Ito-PDE-derivation}, \eqref{eq:Zakai-Smooth:growth_v} and Fubini's theorem yields
	\[
	\begin{split}
	0
	&=
	\E\biggl[\exp\!\Big({\textstyle{\int\limits_0^{\delta}}} B(t-s,R^{t,x}_s)\,ds\Big)
	u \big(t-\delta,\cY^{t,x}_h\big)- u(t,x)\biggr]\\
	&=
	\int_0^{\delta} 
	\E\biggl[\exp\!\Big({\textstyle{\int\limits_0^s}} B(t-r, R^{t,x}_r)\,dr\Big)
	\Big(  -
	\tfrac{\partial}{\partial t} u(t-s, R^{t,x}_s) +\tfrac{1}{2}\operatorname{Trace}_{\R^d}\bigl(\sigma \sigma^{T} \operatorname{Hess}_x (u)(t-s, R^{t,x}_s)\bigr)
	\\
	& 
	\qquad \qquad 
	+ \langle \nabla_x u(t-s, R^{t,x}_s) , 
	b(t-s,R^{t,x}_s) \rangle_{ \R^d }
	+
	B(t-s, R^{t,x}_s) \, u(t-s, R^{t,x}_s)
	\Big)\biggr]
	\,ds 	\end{split}
	\]
	for all $t \in (0,T]$, $\delta \in [0,t]$ and $x\in \R^d$.
	This, \eqref{mathfrak-v-Smooth}, \eqref{eq:Zakai-Smooth:phi:polynomial}--\eqref{eq:Zakai-Smooth:exp:polynomial}, \eqref{eq:Zakai-Smooth:growth_v}, \cref{le:Y-differentiability-properties}, and \cref{le:V} imply that 
	\begin{align} \label{eq:PDE-derivation-final}
	0&= \lim_{\delta \downarrow 0} \frac{1}{\delta}
	\E\Bigl[\exp\!\Big({\textstyle{\int\limits_0^{\delta}}} 
	B(t-s, R^{t,x}_s)\,ds\Big) u \big(t-\delta, R^{t,x}_{\delta} \big)- u(t,x)\Bigr]
	\\ \nonumber
	& =-
	\tfrac{\partial}{\partial t} u(t,x)
	+
	\tfrac{1}{2}\operatorname{Trace}_{\R^d}\bigl(\sigma \sigma^T \operatorname{Hess}_x (u)(t,x) \bigr)
	+
	\langle b(t,x), \nabla_x u(t,x) \rangle_{ \R^d }
	+
	B(t,x) u(t,x)
	\end{align}
	for all $t \in(0,T]$ and $x\in \R^d$.
	Moreover, it follows from \eqref{def:mathcal-Y-Smooth}--\eqref{mathfrak-V-Smooth} 
	that $u(0,x)=\varphi(x)$ for all $x\in \R^d$.
	Combining this with \eqref{eq:PDE-derivation-final} and \eqref{mathfrak-v-Smooth} 
	proves \eqref{PDE-Smooth-item}, which completes the proof.
	\end{proof}

\begin{proposition} \label{prop:Zakai-approx}
Let $T,c,\alpha \in (0,\infty)$, $d\in \N$, $\sigma \in \R^{d\times d}$, and 
consider a function $b \in C^{0,2}([0,T]\times\R^d,\R^d)$ with bounded partial derivatives
of first and second order with respect to the $x$-variables. Let
$\varphi \in C^2(\R^d,[0,\infty))$ have at most polynomially growing partial derivatives
of first and second order and $B \in C^{0,2}([0,T]\times \R^d,\R)$ at most 
polynomially growing partial derivatives of first and second order with 
	respect to the $x$-variables. In addition, assume that
	\begin{equation} \label{eq:Hoelder-fb}
		\begin{split}
		&\Vert b(t,x)- b(s,x) \Vert_{\R^d} 
		+ 
		\Vert \tfrac{\partial^n}{\partial x^n} b(t,x)- \tfrac{\partial^n}{\partial x^n}
		b(s,x)\Vert_{L^{(n)}(\R^d,\R^d)}
		\leq c |t-s|^\alpha,
		\end{split}
		\end{equation}
		\begin{equation} \label{eq:linear-fB}
		\begin{split}
		& B(t,x) \leq c(1 + \lVert x \rVert_{\R^d}), \mbox{ and}
		\end{split}
		\end{equation}
		\begin{equation} \label{eq:Hoelder-fB}
		\begin{split}
		&|B(t,x)- B(s,x)|
		+
		\Vert \tfrac{\partial^n}{\partial x^n} B(t,x)
		- \tfrac{\partial^n}{\partial x^n} B(s,x)\Vert_{L^{(n)}(\R^d,\R)}
		\leq c (1+\Vert x \Vert_{\R^d})|t-s|^\alpha
		\end{split}
		\end{equation}
		for all $n \in \{1,2\}$, $s,t \in [0,T]$ and $x\in \R^d$.
	Let $(\Omega, \mathcal{F}, \P)$ be a probability space supporting a 
	standard Brownian motion
	$U \colon [0,T]\times\Omega\to\R^d$ with continuous sample paths. Consider
	stochastic processes $R^{t,x} \colon[0,t]\times\Omega \to \R^d$, 
	$t\in [0,T]$, $x \in \R^d$, satisfying
		\[ \label{def:mathcal-Y-approx}
		R^{t,x}_s
		=
		x
		+ 
		\int_0^s b(t-r,R^{t,x}_r) \,dr 
		+ 
		\sigma U_s
		\]
	for all $t\in[0,T]$, $s \in [0,t]$ and $x \in \R^d$,
	and let the function $u \colon [0,T]\times\R^d \to \R$ be given by
	\[
	\label{mathfrak-v-copy}
	\begin{split}
	& u(t,x)
	= \E\bigg[\varphi(R^{t,x}_t)\exp\!\Big({\textstyle\int\limits_0^t} 
	B(t-s,R^{t,x}_s)\,ds\Big)\bigg] \quad \mbox{for $t\in[0,T]$ and $x\in \R^d$.}
	\end{split}
	\]
	Then
	\begin{enumerate}[{\rm (i)}]
		\item\label{mathfrak-v-smooth-approx} 
		$u \in C^{1,2}([0,T]\times \R^d,\R)$, and
		\item\label{mathfrak-v-determ-approx}
	$\mbox{}$\\[-8.6mm]
			\[
			\label{nonlinearRPDE-Det-approx}
			\begin{split}
			&u(t,x)
			\\&
			=
			\varphi(x)
			+ 
			\int_0^t
			\bigg[ \tfrac{1}{2} \operatorname{Trace}_{\R^d}\bigl( \sigma \sigma^T
			\operatorname{Hess}_x (u)(s,x)\bigr) 
			+ 
			\left\langle b(s,x), \nabla_x u (s,x)\right\rangle_{\Rd}
			+ B(s,x) \, u(s,x) \bigg]
			\,ds
			\end{split}
			\]
			for all $t\in[0,T]$ and $x\in \R^d$.
	\end{enumerate}
\end{proposition}

\begin{proof}[Proof]
Throughout this proof we fix a 
	$q \in [1, \infty)$ such that
\begin{equation} \label{eq:at_most_polynomial_growth_of_varphi}
	\sup_{t \in [0, T]} 
	\sup_{x \in \R^d} 
	\left[ \frac{ 
		|\varphi(t,x)| 
		+ \Vert \nabla \varphi(t,x) \Vert_{L(\R^d,\R)} 
		+ \Vert D^2 \varphi(t,x) \Vert_{L^{(2)}(\R^d,\R)} }{ (1 + \Vert x \Vert_{\R^d})^{q} } 
	\right] < \infty
		\end{equation}
		and
	\begin{equation}\label{eq:at_most_polynomial_growth_of_B}
	\sup_{t \in [0, T]} 
	\sup_{x \in \R^d} 
	\left[ \frac{ 
		|B(t,x)| 
		+ \Vert D_x B(t,x) \Vert_{L(\R^d,\R)} 
		+ \Vert D^2_x B(t,x) \Vert_{L^{(2)}(\R^d,\R)} }{ (1 + \Vert x \Vert_{\R^d})^{q} }\right] < \infty.
	\end{equation}
Let the mappings $\varphi_n \colon\R^d \to [0,\infty)$ and 
$b_n, B_n \colon [0,T] \times \R^d \to \R^d$, $n\in \N$, be given by 
\begin{equation}
	\label{eq:def:varphi_n}
	\varphi_n(x)
	=
	\big(\tfrac{n}{2\pi}\big)^{\!\nicefrac{d}{2}}
	\int_{\R^d}	\varphi(y) 
	\exp\bigl(-\tfrac{n}{2}\lVert x-y \rVert_{\R^d}^2\bigr)\\
	\,dy,
	\end{equation}
	\begin{equation} \label{eq:def:b_n}
	b_n(t,x)
	= \bigl(\tfrac{n}{2\pi}\bigr)^{\!\nicefrac{1}{2}}\int_{-\infty}^{\infty}
	b\bigl(\min\{T, \max\{s, 0\}\},x\bigr) \exp\Bigl(\tfrac{-n(t-s)^2}{{2}}\Bigr)\,ds,
	\end{equation}
	\begin{equation} \label{eq:def:B_n}
	B_n(t,x)
	=
	\big(\tfrac{n}{2\pi}\big)^{\!\nicefrac{1}{2}}\int_{-\infty}^{\infty}
	B\bigl(\min\{T,\max\{s, 0\}\},x\bigr) \exp\Bigl(\tfrac{-n(t-s)^2}{{2}}\Bigr)\,ds
	\end{equation}
	for all $n \in \N$, $t \in [0,T]$ and $x \in \R^d$. Consider the stochastic processes
$R^{n,t,x} \colon [0,t]\times\Omega \to \R^d$, 
$n \in \N$, $t \in [0,T]$, $x \in \R^d$, and mappings
$\mathcal{U}_n \colon [0,T]\times \R^d\times \Omega \to \R$, $n\in \N_0$, 
$u_n \colon [0,T]\times \R^d \to \R$, $n\in \N$, given by
	\[ \label{def:mathcal-Y_n}
	R^{n,t,x}_s
	=
	x
	+ 
	\int_0^s
	b_n(t-r,R^{n,t,x}_r) \,dr 
	+ \sigma U_s,
	\]
	\[
	\label{eq:def:V-approx}
	\mathcal{U}_0(t,x)=\varphi(R^{t,x}_t)\exp\!\Big({\textstyle\int\limits_0^t} B(t-s,R^{t,x}_s)\,ds\Big), \quad
	\mathcal{U}_n(t,x)=\varphi_n(R^{n,t,x}_t)\exp\!\Big({\textstyle\int\limits_0^t} 
	B_n(t-s,R^{n,t,x}_s)\,ds \Big), 
	\]
	and
	\[
	\label{eq:def:v-approx}
	u_n(t,x)= \E[\mathcal{U}_n(t,x)]
	\]
	for $n \in \N$, $t \in [0,T]$ and $x \in \R^d$. 
	\eqref{eq:at_most_polynomial_growth_of_varphi} and
	\eqref{eq:def:varphi_n} imply that $(\varphi_n)_{n\in\N} \subseteq C^3(\R^d,\R)$ and
	\begin{equation} \label{eq:uniform_derivative_bound_varphi_n}
	\sup_{n \in \N} \sup_{t \in [0, T]} \sup_{x \in \R^d} 
	\left[ 
	\frac{|\varphi_n(t,x)| + \Vert \nabla \varphi_n(t,x) \Vert_{L(\R^d,\R)} + \Vert D^2
	\varphi_n(t,x) \Vert_{L^{(2)}(\R^d,\R)} }{(1 + \Vert x \Vert_{\R^d})^q}
	\right] < \infty. 
	\end{equation} 
Furthermore, since
	$\exp(\nicefrac{-s^2}{2})$ is even in $s \in \R$, it follows from 
 \eqref{eq:def:b_n} that 
 	\[
	\begin{split}
	\tfrac{\partial}{\partial t} b_n (t, x) 
	&
	= 
	\int_{-\infty}^{\infty} b\bigl( \min\{T, \max\{s, 0\}\}, x \bigr) n(s-t) \exp\bigl( \tfrac{-n(t-s)^2}{2} \bigr) \,ds
	\\&
	= 
	\bigl( \tfrac{n}{2\pi} \bigr)^{\!\nicefrac12}
	\int_{-\infty}^{\infty} \Bigl[ b\bigl( \min\{T, \max\{s, 0\}\}, x \bigr) - b\bigl( t, x \bigr)
	\Bigr] n(s-t) \exp\bigl( \tfrac{-n(t-s)^2}{2} \bigr) \,ds. 
	\end{split}
	\]
Moreover, one obtains from \eqref{eq:Hoelder-fb} and the fact that  
$|\min\{T, \max\{s, 0\}\} - t| \leq |t-s|$ for all $t \in [0,T]$ and $s \in \R$ that
	\[
	\begin{split}
	|\tfrac{\partial}{\partial t} b_n (t, x)|
	&
	\leq  
	\bigl( \tfrac{n}{2\pi} \bigr)^{\!\nicefrac12}
	\int_{-\infty}^{\infty} c \left| \min\{T, \max\{s, 0\}\} - t \right|^{ \alpha } n |s-t| \exp\bigl( \tfrac{-n(t-s)^2}{2} \bigr) \,ds
	\\
	& \leq 
	c
	\bigl( \tfrac{n}{2\pi} \bigr)^{\!\nicefrac12}
	\int_{-\infty}^{\infty} n |s-t|^{1+\alpha} \exp\bigl( \tfrac{-n(t-s)^2}{2} \bigr) \,ds
	\leq 
	\tfrac{c}{\sqrt{2\pi}} n^{\nicefrac{(1-\alpha)}{2}} \int_{-\infty}^{\infty} 
	|z|^{1+\alpha} \exp\bigl( \tfrac{-|z|^2}{2} \bigr) \,dz
	\end{split}
	\]
	for all $n \in \N$, $t \in [0, T]$ and $x \in \R^d$.
Combining this with the assumption that $b \in C^{0, 2}([0, T]\times\R^d, \R^d)$ 
has bounded partial derivatives of first and second order with 
respect to the $x$-variables and \eqref{eq:def:b_n} ensures that 
	\begin{enumerate}[(a)]
		\item \label{it:differentiability_properties_of_bn}
		for all $n \in \N$, $b_n \in C^{1, 2}([0, T] \times \R^d, \R^d)$ 
		has a bounded partial derivative with respect to $t$ and 
		bounded partial derivatives of first and second order with respect to the 
		$x$-variables,
		\item $\mbox{}$\\[-7mm]
			\begin{equation*} \label{eq:uniform_derivative_bounds_bn_b}
			\sup_{n \in \N} 
			\sup_{(t,x)\in [0,T]\times \R^d} \lVert D_x
			b_n(t,x)\rVert_{L(\R^d,\R^d)} 
			\leq 
			\sup_{(t,x)\in [0,T]\times \R^d}
			\lVert D_x b(t,x)\rVert_{L(\R^d,\R^d)} < \infty,
			\end{equation*}
		\item $\mbox{}$\\[-7mm]
			\begin{equation*} \label{eq:uniform_derivative_bounds_bn_b-2}
			\sup_{n \in \N}
			\sup_{(t,x)\in [0,T]\times \R^d} 
			\lVert D^2_x b_n(t, x) \rVert_{L^{(2)}(\R^d, \R^d)}
			\leq
			\sup_{(t,x)\in [0,T] \times \R^d} 
			\lVert D^2_x b (t,x) \rVert_{L^{(2)}(\R^d,\R^d)} < \infty.
			\end{equation*} 
	\end{enumerate}
Next, note that it follows from \eqref{eq:def:B_n} and the assumption that
$B(t,x) \leq c( 1 + \Vert x \Vert_{\R^d})$ for all $t \in [0, T]$ and $x \in \R^d$ that
	\begin{equation} \label{eq:uniform-linear-bound-from-above-for-Bn}
	\sup_{n \in \N}\sup_{t \in [0, T]}\sup_{x \in \R^d} \left[ \frac{B_n(t, x)}{1 + \Vert x \Vert_{\R^d}} \right] 
	\leq 
	\sup_{t \in [0, T]}\sup_{x \in \R^d} \left[ \frac{B(t, x)}{1 + \Vert x \Vert_{\R^d}} \right] < \infty,
	\end{equation} 	
In addition, one obtains from
\eqref{eq:at_most_polynomial_growth_of_B} and \eqref{eq:def:B_n} that 
	$(B_n)_{n \in \N} \subseteq C^{1, 2}([0, T] \times \R^d, \R)$
and 
	\begin{equation} \label{eq:derivative-bounds-Bn}
	\sup_{n \in \N} \sup_{t \in [0, T]} \sup_{x \in \R^d} 
	\left[ 
	\frac{|B_n(t,x)| + \Vert D_x B_n(t,x) \Vert_{L(\R^d,\R)} 
	+ \Vert D^2_x B_n(t,x) \Vert_{L^{(2)}(\R^d,\R)}}{(1 + \Vert x \Vert_{\R^d})^q}
	\right] < \infty. 
	\end{equation} 
Now, note that it follows from \cref{le:Zakai-smooth-x} and the assumptions on $\varphi$, $b$,
and $B$ that 
	\begin{enumerate}[(a')]
	\item \label{mathfrak-v-smooth-smooth-x-Approx} $u \in C^{0,2}([0,T]\times \R^d,\R)$,
		\item\label{mathfrak-v-Dx-Approx}
			$\frac{\partial}{\partial x} u(t,x) = \E[\tfrac{\partial}{\partial x} \mathcal{U}_0(t,x)]$
			for all $t\in [0,T]$ and $x\in \R^d$,
		and
		\item\label{mathfrak-v-D-2-xx-Approx}
			$\frac{\partial^2}{\partial x^2} u(t,x) = \E[\tfrac{\partial^2}{\partial x^2} 
			\mathcal{U}_0(t,x)]$
			for all $t \in [0,T]$ and $x \in \R^d$.
	\end{enumerate}
By \cref{le:Zakai-smooth-x}, \cref{le:Zakai-Smooth},
\eqref{it:differentiability_properties_of_bn}--\eqref{eq:uniform_derivative_bounds_bn_b-2} and
\eqref{eq:uniform_derivative_bound_varphi_n}--\eqref{eq:derivative-bounds-Bn}, one has
	\begin{enumerate}[(A)]
		\item\label{mathfrak-v-Smooth-Approx} 
			$u_n\in C^{1,2}([0,T]\times \R^d,\R)$ for all $n \in \N$,
		\item\label{mathfrak-v-Dx-Smooth-Approx}
			$\frac{\partial}{\partial x} u_n(t,x)=\E[\tfrac{\partial}{\partial x} 
			\mathcal{U}_n(t,x)]$
			for all $n \in \N$, $t\in [0,T]$ and $x\in \R^d$,
		\item\label{mathfrak-v-D-2-xx-Smooth-Approx}
			$\frac{\partial^2}{\partial x^2} u_n(t,x) 
			= \E[\tfrac{\partial^2}{\partial x^2} \mathcal{U}_n(t,x)]$ 
			for all $n \in \N$, $t\in [0,T]$ and $x\in \R^d$,
		and 
		\item\label{PDE-Smooth-item-Approx} $\mbox{}$\\[-13mm]
					\begin{multline}
			u_n(t,x) 
			=
			\varphi_n(x)
			\\
			+ 
			\int_0^t
			\Bigl[ 
			\tfrac{1}{2} \operatorname{Trace}_{\R^d}\bigl( \sigma \sigma^{*}(\operatorname{Hess}_x u_n)(s,x)\bigl) +
			\left\langle b_n(s,x),(\nabla_x u_n)(s,x)\right\rangle_{\Rd}
			+ B_n(s,x) \, u_n(s,x)
			\Bigr]
			\,ds
			\end{multline}
	\end{enumerate}
	for all $n \in \N$, $t \in [0,T]$ and $x\in \R^d$.
From \cref{le:Mollification}.\eqref{it:Mollifier:convergence-Bn_B} 
together with \eqref{eq:Hoelder-fb}, \eqref{eq:def:b_n}, \eqref{eq:uniform_derivative_bounds_bn_b} and \eqref{eq:uniform_derivative_bounds_bn_b-2} one obtains
\begin{enumerate}[(A')]
		\item \label{it:lim-b_n-b}
		$\limsup\nolimits_{n \to \infty} 
		\sup_{(t,x)\in [0,T]\times \R^d}  \Vert b_n(t,x)- b(t,x) \Vert_{\R^d} = 0$, 
		\item \label{it:lim-D_x-b_n-D_x-b}
		$\limsup\nolimits_{n \to \infty} \sup_{(t,x)\in [0,T]\times \R^d} 
		\Vert D_x b_n(t,x) - D_x b(t,x)\Vert_{L(\R^d,\R^d)} =0$,
		and 
		\item \label{it:lim-D_xx-b_n-D_xx-b}
		$\limsup\nolimits_{n \to \infty} \sup_{(t,x)\in [0,T]\times \R^d} 
		\Vert D^2_x b_n(t,x)
		- D^2_x b (t,x)\Vert_{L^{(2)}(\R^d,\R^d)} =0$.			
	\end{enumerate}
	Combining \cref{le:Y_processes_stability}.\eqref{le:Y_processes_stability:item1} with
\eqref{it:lim-b_n-b} gives
	\begin{equation}\label{eq:lim-Y_n-Ynn}
	\begin{split}
	&\limsup\nolimits_{n \to \infty} 
		\sup\nolimits_{t \in [0,T]} 
		\sup\nolimits_{s \in [0,t]}
		\sup\nolimits_{x \in \R^d}
	\Vert R^{n,t,x}_s - R^{t,x}_s \Vert_{\R^d} \\
	&\leq 
	T 
	\exp\bigl(T \sup\nolimits_{(r, y)\in [0,T]\times \R^d} 
	\big\Vert D_x b(r, y)\Vert_{L(\R^d,\R^d)}
	\bigr) \\
	& \quad \times
	\limsup\nolimits_{n\to\infty}
	\sup\nolimits_{(r, y)\in [0,T]\times \R^d}  \Vert b_n(r, y)- b(r, y) \Vert_{\R^d} 
	= 0.
	\end{split}
	\end{equation}
	Hence, we obtain from \cref{le:sde-solution-exponential-integrability} that	%
	\begin{multline}\label{eq:delaVaPou-cY-n-Approx}
	\sup\nolimits_{n\in \N}
	\E\Bigl[\exp\bigl(p 
		\sup\nolimits_{t\in [0,T]}
		\sup\nolimits_{s\in [0,t]}
		\sup\nolimits_{x\in [-p,p]^d}
		\lVert \cY^{n,t,x}_s \rVert_{\R^d} \bigr)\Bigr]\\
	\leq 
	\sup\nolimits_{n\in \N}
	\E\Bigl[\exp\bigl(p 
		\sup\nolimits_{t\in [0,T]}
		\sup\nolimits_{s\in [0,t]}
		\sup\nolimits_{x\in [-p,p]^d}
		\lVert \cY^{n,t,x}_s -\cY^{t,x}_s \rVert_{\R^d} \bigr) 
	\\
	\times \exp\bigl(p
		\sup\nolimits_{t\in [0,T]}
		\sup\nolimits_{s\in [0,t]}
		\sup\nolimits_{x\in [-p,p]^d}
		\lVert R^{t,x}_s \rVert_{\R^d}\bigr)\Bigr]
	<\infty \quad \mbox{for all $p \in (0,\infty)$.}
	\end{multline}
Moreover, \cref{le:loc-unif-conv}, \eqref{it:lim-D_x-b_n-D_x-b}, and \eqref{eq:lim-Y_n-Ynn} 
yield 
	\begin{equation}
	\label{eq:lim-bn-x-b-x}
	\limsup_{n \to \infty} 	
	\sup_{t \in [0, T]} 
	\sup_{s \in [0, t]} 
	\sup_{x \in [-p, p]^d} 
	\big\Vert D_x b_n(t-s, R^{n, t,x}_s) 
	- D_x b(t-s,R^ {t,x}_s)\big\Vert_{L(\R^d,\R^d)}
	= 0
	\end{equation}
	 for all $p \in (0,\infty)$, which together with
\cref{le:Y_processes_stability}.\eqref{le:Y_processes_stability:item2},\eqref{eq:uniform_derivative_bounds_bn_b} and \eqref{eq:delaVaPou-cY-n-Approx} 
shows that 
	\begin{equation} \label{eq:lim_Dx_Yn-Dx_Y}
	\limsup_{n \to \infty} 
	\sup_{t \in [0, T]} 
	\sup_{s \in [0, t]} 
	\sup_{x \in [-p, p]^d} 
	\lVert \tfrac{\partial}{\partial x_i} \cY^{n, t, x}_s - \tfrac{\partial}{\partial x_i} R^{t, x}_s \rVert_{\R^d}
	= 0
	\end{equation} 
for all $i \in \{1, 2, \dots, d\}$ and $p \in (0, \infty)$.
Furthermore, it follows from \cref{le:Y-differentiability-properties}.\eqref{eq:delaVaPou-x-cY} 
and \eqref{eq:uniform_derivative_bounds_bn_b} that
	\begin{equation}
	\label{eq:partial-Y-n-Y-sum}
	\begin{split}
	&
	\sup\nolimits_{n \in \N}
	\sup\nolimits_{t \in [0,T]}
	\sup\nolimits_{s \in [0,t]} 
	\sup\nolimits_{x \in \R^d} 
	\Bigl[
	\lVert \tfrac{\partial}{\partial x_i} R^{t,x}_s \rVert_{\R^d}
	+
	\lVert \tfrac{\partial}{\partial x_i} \cY^{n,t,x}_s \rVert_{\R^d}
	\Bigr]
	\\
	& \leq 
	2 \exp\bigl(
	T \sup\nolimits_{(t,x)\in [0,T]\times \R^d}
	\Vert D_x b(t,x)\Vert_{L(\R^d,\R^d)} \bigr)
	<\infty \quad \mbox{for all $i \in \{1, 2, \dots, d\}$.}
	\end{split}
	\end{equation} 
By \cref{le:loc-unif-conv}, \eqref{it:lim-D_xx-b_n-D_xx-b}, and \eqref{eq:lim-Y_n-Ynn},
we obtain that	
\[
	\limsup_{n \to \infty} 	
	\sup_{t \in [0, T]} 
	\sup_{s \in [0, t]} 
	\sup_{x \in [-p, p]^d} 
	\big\Vert D^2_x b_n(t-s, R^{n, t, x}_s) 
	- D^2_x b (t-s,R^ {t,x}_s)\big\Vert_{L^{(2)}(\R^d,\R^d)} = 0
	\]
	for all $p \in (0, \infty)$.
\cref{le:Y_processes_stability}.\eqref{le:Y_processes_stability:item3} \eqref{eq:uniform_derivative_bounds_bn_b}, \eqref{eq:uniform_derivative_bounds_bn_b-2}, and \eqref{eq:lim-bn-x-b-x} hence assure that 
	\begin{equation} 
	\label{eq:lim_Dxx_Yn-Dxx_Y}
	\limsup_{n \to \infty} 
		\sup_{t \in [0, T]} 
		\sup_{s \in [0, t]} 
		\sup_{x \in [-p, p]^d} 
		\lVert \tfrac{\partial^2}{\partial x_i\partial x_j} R^{n,t,x}_s 
		- \tfrac{\partial^2}{\partial x_i\partial x_j} R^{t,x}_s \rVert_{\R^d} = 0
	\end{equation} 
	for all $i,j \in \{1, 2, \ldots, d\}$ and $p \in (0, \infty)$.
	Moreover, it follows from \cref{le:Y-differentiability-properties}.\eqref{eq:delaVaPou-x-x-cY},
	\eqref{eq:uniform_derivative_bounds_bn_b} and \eqref{eq:uniform_derivative_bounds_bn_b-2}
	that	\begin{equation}
	\label{eq:delaVaPou-x-x-cY-n-approx}
	\begin{split}
	&
	\sup\nolimits_{n \in \N} 
	\sup\nolimits_{t \in [0,T]}
	\sup\nolimits_{s \in [0,t]} 
	\sup\nolimits_{x \in \R^d} 
	\bigl[ \lVert \tfrac{\partial^2}{\partial x_i \partial x_j} \cY^{t,x}_s \rVert_{\R^d}
	+
	\lVert \tfrac{\partial^2}{\partial x_i \partial x_j} \cY^{n,t,x}_s \rVert_{\R^d}
	\bigr]
	\\
	& \leq
	2T 
	\bigl[
	\sup\nolimits_{(t,x)\in [0,T] \times \R^d} 
		\lVert D^2_x b(t,x) \rVert_{L^{(2)}(\R^d,\R^d)}
	\bigr]
	\exp\bigl( 3T \sup\nolimits_{(t,x)\in [0,T]\times \R^d}\Vert D_x
	b(t,x)\Vert_{L(\R^d,\R^d)} \bigr)
	\\
	&<\infty \quad \mbox{for all $i,j \in \{1,2, \dots, d\}$.}
	\end{split}
	\end{equation}
	Next, note that \cref{le:Mollification}, \cref{le:loc-unif-conv}, \eqref{eq:lim-Y_n-Ynn}, 
	and \eqref{eq:delaVaPou-cY-n-Approx}
	 ensure that 
	\begin{equation} \label{eq:V_n-V}
	\limsup_{n \to \infty} 
		\sup_{t \in [0, T]} 
		\sup_{x \in [-p, p]^d} 
		|\mathcal{U}_n(t, x) - \mathcal{U}_0(t, x) | = 0 \quad \mbox{for all $p \in (0, \infty)$.}
	\end{equation}
Moreover, it follows from \cref{le:sde-solution-exponential-integrability}  together with \eqref{eq:at_most_polynomial_growth_of_varphi},  
\eqref{eq:uniform_derivative_bound_varphi_n},
	\eqref{eq:uniform-linear-bound-from-above-for-Bn}
	 and 
	\eqref{eq:delaVaPou-cY-n-Approx} that
	\[
\sup_{n \in \N} \E \left[ \sup_{t \in [0,T]} | \mathcal{U}_n(t, x) - 
\mathcal{U}_0(t, x) |^p \right] < \infty \quad
\mbox{for all $x \in \R^d$ and $p \in (0, \infty)$.}
\]
Combining this and \eqref{eq:V_n-V} with the de la Vall\'ee Poussin theorem 
(cf., e.g., \cite[Corollary~6.21]{Klenke_2014}) and the Vitali convergence theorem (cf., e.g., \cite[Theorem~6.25]{Klenke_2014}) implies that 
	\begin{equation}
	\label{eq:lim-v_n-v}
	\limsup_{n\to \infty} 
		\sup_{t \in [0,T]} 
			| u_n(t, x) - u(t, x) | 
	\leq 
	\limsup_{n\to \infty} 
	\E\left[ \sup\nolimits_{t \in [0,T]} |\mathcal{U}_n(t, x) - \mathcal{U}_0(t, x) | \right] 
	= 0 
	\end{equation}
	for all for $x\in \R^d$.
Next, note that \cref{le:Mollification}, \cref{le:loc-unif-conv}, \cref{le:V}.\eqref{eq:partial-x-V},
\eqref{eq:Hoelder-fB}, \eqref{eq:at_most_polynomial_growth_of_varphi}, \eqref{eq:lim-Y_n-Ynn}, and \eqref{eq:lim_Dx_Yn-Dx_Y} ensure that 
	\begin{equation} \label{eq:lim_Dx_Vn-Dx_V}
	\limsup_{n \to \infty} 
		\sup_{t \in [0, T]} 
		\sup_{x \in [-p, p]^d} 
	\lVert \tfrac{\partial}{\partial x}\mathcal{U}_n(t, x) 
	- \tfrac{\partial}{\partial x} \mathcal{U}_0(t, x) \rVert_{L(\R^d,\R)}
	= 0 \quad \mbox{for all } p \in (0,\infty).
	\end{equation} 
Moreover, \eqref{eq:at_most_polynomial_growth_of_varphi},  \eqref{eq:uniform_derivative_bound_varphi_n},
\eqref{eq:uniform-linear-bound-from-above-for-Bn}, 
\eqref{eq:delaVaPou-cY-n-Approx}, \eqref{eq:partial-Y-n-Y-sum}, \cref{le:sde-solution-exponential-integrability}, and 
\cref{le:V}.\eqref{eq:partial-x-V} imply that
\[
\sup_{n \in \N} \E \left[ \sup_{t \in [0,T]} \lVert \tfrac{\partial}{\partial x}
\mathcal{U}_n(t, x) - \tfrac{\partial}{\partial x} \mathcal{U}_0(t, x) 
\rVert_{L(\R^d,\R)}^p \right] < \infty \quad
\mbox{for all $x \in \R^d$ and $p \in (0,\infty)$.}
\]
It therefore follows from
\eqref{mathfrak-v-smooth-smooth-x-Approx}--\eqref{mathfrak-v-Dx-Approx},
\eqref{mathfrak-v-Smooth-Approx}--\eqref{mathfrak-v-Dx-Smooth-Approx},
\eqref{eq:lim_Dx_Vn-Dx_V}, 
the de la Vall\'ee Poussin theorem (cf., e.g., \cite[Corollary~6.21]{Klenke_2014}),
and 
the Vitali convergence theorem (cf., e.g., \cite[Theorem~6.25]{Klenke_2014}) that
	\begin{equation}
	\label{eq:lim-gradient-v_n-v}
	\limsup_{n\to \infty}
		\sup_{t \in [0, T]} 
		\lVert \tfrac{\partial}{\partial x} u_n(t,x)
		- \tfrac{\partial}{\partial x} u (t,x)\rVert_{\R^d}
		= 0 \quad \mbox{for all $x \in \R^d$.}
	\end{equation}
Similarly, we obtain from
\cref{le:Mollification}, \cref{le:loc-unif-conv}, \cref{le:V}.\eqref{eq:partial-x-x-V},
\eqref{eq:Hoelder-fB}, \eqref{eq:at_most_polynomial_growth_of_varphi}, \eqref{eq:lim-Y_n-Ynn}, \eqref{eq:lim_Dx_Yn-Dx_Y},
and
\eqref{eq:lim_Dxx_Yn-Dxx_Y} that
	\begin{equation} \label{eq:lim_Dxx_Vn-Dxx_V}
	\limsup_{n \to \infty} 
		\sup_{t \in [0, T]} 
		\sup_{x \in [-p, p]^d} 
		\lVert \tfrac{\partial^2}{\partial x^2}\mathcal{U}_n(t, x) 
		- \tfrac{\partial^2}{\partial x^2}\mathcal{U}_0(t, x) \rVert_{L^{(2)}(\R^d, \R^d)} 
	= 0 \quad \mbox{for all $p \in (0,\infty)$.}
	\end{equation} 
Moreover, observe that \eqref{eq:at_most_polynomial_growth_of_varphi},  
\eqref{eq:uniform_derivative_bound_varphi_n}, 
\eqref{eq:uniform-linear-bound-from-above-for-Bn}, 
\eqref{eq:delaVaPou-cY-n-Approx}, \eqref{eq:partial-Y-n-Y-sum}, 
\eqref{eq:delaVaPou-x-x-cY-n-approx}, 	
\cref{le:sde-solution-exponential-integrability}, and \cref{le:V}.\eqref{eq:partial-x-x-V} show that 
	\[
	\sup_{n \in \N} \E \left[
	\sup_{t \in [0, T]} \lVert
	\tfrac{\partial^2}{\partial x^2} \mathcal{U}_n(t, x) 
	- \tfrac{\partial^2}{\partial x^2} \mathcal{U}_0(t, x) \rVert_{L^{(2)}(\R^d, \R^d)}^p
	\right] \allowbreak < \infty \quad \mbox{for all $x \in \R^d$ and $p \in (0,\infty)$.}
	\]
\eqref{mathfrak-v-smooth-smooth-x-Approx}, \eqref{mathfrak-v-D-2-xx-Approx},
\eqref{mathfrak-v-Smooth-Approx}, \eqref{mathfrak-v-D-2-xx-Smooth-Approx}, 
\eqref{eq:lim_Dxx_Vn-Dxx_V}, 
	the de la Vall\'ee Poussin theorem (cf., e.g., \cite[Corollary~6.21]{Klenke_2014}),
	and 
	the Vitali convergence theorem (cf., e.g., \cite[Theorem~6.25]{Klenke_2014}) 
	hence imply that 	\begin{equation}
	\label{eq:lim-x-x-v_n-v}
	\limsup_{n\to \infty}
	\sup_{t\in [0,T]}
		\lVert \tfrac{\partial^2}{\partial x^2} u_n(t,x)
		- \tfrac{\partial^2}{\partial x^2} u(t,x)\rVert_{L^{(2)}(\R^d, \R^d)} 
	= 0 \quad \mbox{for all } x \in \R^d.
	\end{equation}
Moreover, note that \eqref{eq:def:varphi_n} and \cref{le:Mollification}.\eqref{it:Mollifier:convergence-Gn-G} ensure that
	\[
	\limsup_{n \to \infty} |\varphi_n (x) - \varphi(x) | = 0 \quad \mbox{for all $x \in \R^d$.}
	\]
It therefore follows from 
\eqref{PDE-Smooth-item-Approx}, 
	\eqref{eq:lim-v_n-v}, \eqref{eq:lim-gradient-v_n-v}, and \eqref{eq:lim-x-x-v_n-v} that 
	\[
	\label{PDE-v-Approx}
	\begin{split}
	&u(t,x)\\
	&=
	\varphi(x)
	+ 
	\int_0^t
	\Bigl[ 
	\tfrac{1}{2} \operatorname{Trace}_{\R^d}\bigl( \sigma \sigma^{T} \operatorname{Hess}_x 
	(u)(s,x) \bigr)
	+
	\left \langle u(s,x), \nabla_x u(s,x)\right\rangle_{\Rd}
	+ B(s, x) \, u(s, x) 
	\Bigr] \,ds 
	\end{split}
	\]
	for all $t\in [0,T]$ and $x\in \R^d$,
which establishes \eqref{mathfrak-v-determ-approx}.
Moreover, \eqref{mathfrak-v-determ-approx} and 
\eqref{mathfrak-v-smooth-smooth-x-Approx} imply \eqref{mathfrak-v-smooth-approx},
which finishes the proof of the proposition.
\end{proof}

\subsection{Proof of \cref{thm:main}}
\label{sec:proofthm}

Consider the random field
\[
	I(t, x) = \langle h(x), Z_t \rangle_{\R^k}, \quad t \in [0,T], \, x \in \R^d,
	\]
and define the operator $\mathcal{L}\colon C^{2}(\R^d,\R) \to C(\R^d,\R)$ by
	\begin{equation} \label{eq:def-L-operator}
	\cL w(x)
	= \tfrac{1}{2}\operatorname{Trace}_{\R^d}\bigl( \sigma \sigma^{T} 
	\operatorname{Hess} (w)(x) \bigr) 
	-
	\langle \mu(x), \nabla w(x) \rangle_{\R^d}, \quad
	w \in C^2(\R^d, \R), \, x \in \R^d.
	\end{equation}
Let the mappings $b_z \colon [0,T]\times \R^d \to \R^d$, $z \in C([0,T],\R^d)$, be given by
	\begin{equation} \label{eq:def:fb}
	b_z(t,x) = \sigma \sigma^T [D h(x)]^T z(t) -\mu(x), \quad 
	z \in C([0,T],\R^k), \, t \in [0,T], \, x\in \R^d,
	\end{equation}	
	and define the random field
	\[
	\fu = (\fu(t, x))_{t \in [0, T], x \in \R^d} 
	= (\fu(t, x, \omega))_{t \in [0, T], x \in \R^d, \omega\in\Omega} 
	\colon [0,T]\times \R^d \times \Omega \to \R	\]
 by
	\[
	\fu(t, x, \omega) = u_{ Z(\omega)}(t, x), \quad t \in [0,T], \, x \in \R^d, \,
	\omega \in \Omega,
	\]
where $u_z$ is given by \eqref{eq:u}.
	Observe that it follows from \eqref{eq:YZ}--\eqref{eq:Zakai:definition_X}
	that for all $t\in [0,T]$ and $x\in \R^d$, $X_t(x) = \fu(t, x) e^{I(t,x)}$ 
	is $\F_t$/$\B(\R)$-measurable. So it satisfies \eqref{it:Zakai-1}.
	
	Next, note that it follows from the assumptions on 
	$\mu\in C^3(\R^d,\R^d)$ and $h\in C^4(\R^d,\R^d)$  
	together with \eqref{eq:mathfrak-B} 
	that 
	\begin{enumerate}[(a)]
	\item \label{it:bv_C02_with_bounded_derivatives}
		for every $z \in C([0,T],\R^k)$,
		$b_z \in C^{0,2}([0,T]\times\R^d, \R^d)$ has bounded partial derivatives
		of first and second order with respect to the $x$-variables,
		\item \label{it:B}
		$\sup_{(t,x)\in[0,T]\times \R^d} \tfrac{B_z (t,x)}{1 + \|x\|_{\R^d}} <\infty$
		for all $z \in C([0,T],\R^d)$, and
		\item \label{it:Bvy_C02_with_at_most_polynomially_growing_derivatives}
		for all $z \in C([0,T],\R^k)$, the mapping
		$B_z\in C^{0,2}([0,T]\times\R^d, \R)$ has at most polynomially 
		growing partial derivatives of first and second order with respect to the $x$-variables.
	\end{enumerate} 
	Moreover, note that it follows from \eqref{eq:YZ} that
	\[
	\inf_{\alpha \in (0,1)} 
	\sup_{s\neq t \in [0,T]} \tfrac{\Vert Z_t - Z_s\Vert_{\R^k}}{|t-s|^\alpha} 
	< \infty \quad \mbox{$\P$-a.s.}
	\]
	This, \eqref{eq:mathfrak-B} and \eqref{eq:def:fb} assure that for $\P$-a.a.
	$\omega \in \Omega$ the following hold:
	\begin{align*}
	& 
	\inf_{\alpha \in (0, 1)}
	\sup_{s \neq t \in [0,T]} 
	\sup_{x \in \R^d}
	\frac{\lVert b_{Z(\omega)}(t,x)- b_{Z(\omega)}(s,x)\rVert_{\R^d} }{|t-s|^\alpha}<\infty,\\
	& 
	\inf_{\alpha \in (0, 1)}
	\sup_{n \in \{1, 2\}} 
	\sup_{s \neq t \in [0,T]} 
	\sup_{x \in \R^d}
	\frac{\lVert \tfrac{\partial^n}{\partial x^n} b_{Z(\omega)}(t,x)-
	\tfrac{\partial^n}{\partial x^n}b_{Z(\omega)}(s,x)\rVert_{L^{(n)}(\R^d,\R^d)}
	}{|t-s|^\alpha}<\infty,\\
	&
	\inf_{\alpha \in (0, 1)}
	\sup_{s\neq t \in [0,T]} \sup_{x \in \R^d}
	\frac{|B_{Z(\omega)}(t,x)- B_{Z(\omega)}(s,x)| }{(1+\Vert x \Vert_{\R^d})|t-s|^\alpha}<\infty,\\
	&
	\inf_{\alpha \in (0, 1)}
	\sup_{n\in \{1,2\}} \sup_{s\neq t \in [0,T]} \sup_{x \in \R^d}
	\frac{\lVert \tfrac{\partial^n}{\partial x^n} B_{Z(\omega)}(t,x)-
	\tfrac{\partial^n}{\partial x^n}B_{Z(\omega)}(s,x)\rVert_{L^{(n)}(\R^d,\R)}  }{(1
	+\Vert x \Vert_{\R^d})|t-s|^\alpha}<\infty,\\
	&
	\sup_{(t,x) \in [0,T]\times \R^d} \tfrac{B_{Z(\omega)}(t,x)}{1+\Vert x \Vert_{\R^d}} < \infty.
	\end{align*}
\eqref{it:bv_C02_with_bounded_derivatives}--\eqref{it:Bvy_C02_with_at_most_polynomially_growing_derivatives}, the assumption that 
	$\varphi \in C^2(\R^d,[0,\infty))$ has at most polynomially growing derivatives
	up to the second order, \cref{le:sde-solution-exponential-integrability},
	\cref{le:Zakai-smooth-x}, and \cref{prop:Zakai-approx} (with 
	$b \leftarrow b_{V(\omega)}$,
	$R^{t,x} \leftarrow R^{V(\omega),t,x}$,
$\varphi \leftarrow \varphi$,
	$B \leftarrow B_{V(\omega),Y(\omega)}$,
	$u \leftarrow u_{V(\omega),Y(\omega)}$
	in the notation of \cref{le:Zakai-smooth-x} and \cref{prop:Zakai-approx}) 
	ensure the following:
	\begin{enumerate}[(A)]
		\item \label{mathfrak-u-A} 
		for all $\omega \in \Omega$, the mapping
		$(t, x) \mapsto \fu(t, x, \omega)$ is in 
		$C^{0,2}([0,T] \times \R^d,\R)$ and there exist constants 
		$a(\omega), c(\omega) \in \R$ such that 
		\[
		\sup_{t \in [0,T]} |\fu(t, x, \omega)| \le a(\omega) e^{c(\omega) \| x\|_{\R^d}}
		\quad \mbox{for all $x \in \R^d$,}
		\]
		\item \label{mathfrak-u-C} 
		for $\P$-a.a.\ $\omega \in \Omega$, the mapping
		$(t, x) \mapsto \fu(t, x, \omega)$ is in $C^{1,2}([0,T]\times \R^d,\R)$, 
		and
		\item \label{mathfrak-u-D} $\mbox{}$\\[-14mm]
		\begin{align*} \label{nonlinearRPDE-Det}
		& \fu(t, x, \omega) 
		=
		u_{Z(\omega)}(t,x)
		=
		\varphi(x)
		+ 
		\int_0^t
		\Bigl[ 
		\tfrac{1}{2} \operatorname{Trace}_{\R^d} \bigl( \sigma \sigma^{T} 
		\operatorname{Hess}_x (\fu)(s, x, \omega) \bigr)\\
		& +
		\left\langle b_{Z(\omega)}(s,x) , \nabla_x \fu(s,x,\omega)\right\rangle_{\Rd}
		+ B_{Z(\omega)}(s,x) \fu(s,x,\omega)
		\Bigr] \,ds 
		\end{align*}
		for all $t \in [0,T]$, $x\in \R^d$ and $\P$-almost all $\omega \in \Omega$.
	\end{enumerate}
	It follows from the assumptions on $h\in C^4(\R^d,\R^d)$ that
	for all $\omega \in \Omega$, the mapping
	$(t, x) \mapsto I(t, x, \omega)$ is in $C^{0,2}([0, T] \times \R^d, \R)$.
	Moreover, since $X_t(x) = \fu(t, x) e^{I(t, x)}$ for all $t \in [0,T]$ and $x \in \R^d$,
	we obtain from \eqref{mathfrak-u-A} that
 for every $\omega \in \Omega$, the mapping 
$(t,x)\mapsto X_t(x,\omega)$ is in $C^{0,2}([0,T]\times\R^d,\R)$ and there exist constants 
		$a(\omega), c(\omega) \ge 0$ such that 
		\[
		\sup_{t \in [0,T]} |X(t, x, \omega)| \le a(\omega) e^{c(\omega) \| x\|_{\R^d}}
		\quad \mbox{for all $x \in \R^d$.}
		\]
Hence, $X$ fulfils \eqref{it:Zakai-2}.	
	
	Next, note that it follows from
	\be \label{nablaI}
	\nabla_x I(t, x) = [Dh(x)]^T Z_t \quad \mbox{$t \in [0, T]$, $x \in \R^d$,}
	\ee
	\eqref{eq:def-L-operator} and \eqref{mathfrak-u-D} that 
	\be \label{eq:nonlinearRPDE-Det-umgeschrieben}
	\fu(t, x) 
	=
	\varphi(x)
	+ 
	\int_0^t \bigl[ \cL_x \fu (s, x) 
	+
	\left\langle \sigma \sigma^T \nabla_x I(s, x), \nabla_x \fu(s,x)\right\rangle_{\Rd}
	+ B_{Z}(s,x) \, \fu(s, x)
	\bigr]
	\,ds \quad \mbox{$\P$-a.s.}
	\ee
	for all $t \in [0,T]$ and $x \in \R^d$.
	Now, consider the random field
		\begin{equation}\label{def:A-operator}
	\mathfrak{v}(t,x) =	e^{-I(t,x)} \cL X_t(x), \quad 
	t \in [0,T], \, x \in \R^d.
	\end{equation}
	Since
	$X_t(x) = \fu(t, x) e^{I(t, x)} $ for all $t \in [0, T]$ and $x \in \R^d$, we obtain from
	\eqref{nablaI} and \cref{le:L-operator} below 
	(applied for every $t \in [0, T]$ and $\omega \in \Omega$ 
	with $a \leftarrow \sigma \sigma^{T}$, $\nu \leftarrow \mu$,
	$f \leftarrow I(t,.,\omega)$,
	$g \leftarrow \fu(t, ., \omega)$ in the notation of \cref{le:L-operator}) that
		\[ \label{le:eq:L-operator-2}
	\begin{split}
	\mathfrak{v}(t, x)
	&= 
	\cL_x \fu(t,x) 
	+ \left\langle \sigma \sigma^{T} \nabla_x I(t, x), \nabla_x \fu(t, x) \right\rangle_{\Rd} 
	\\
	&\quad
	+ \bigl[ \tfrac{1}{2}\left\langle \sigma \sigma^{T} \nabla_x I(t, x), \nabla_x I(t, x) 
	\right\rangle_{\Rd} + \cL_x I(t,x)\bigr] \fu(t, x) \quad \mbox{for all $t \in [0,T]$ and $x \in \R^d$.}
	\end{split}
	\]
	So it follows from \eqref{eq:nonlinearRPDE-Det-umgeschrieben} that
	\[
	\label{eq:RPDE-trans1}
	\begin{split}
	& \fu(t, x) 
	\\
	& =
	\varphi(x)
	+ 
	\int_0^t
	\bigl[ \mathfrak{v}(s, x) 
+ \crl{ B_{Z}(s,x) - \tfrac{1}{2}\left\langle \sigma \sigma^T \nabla_x I(s, x), \nabla_x I(s, x) \right\rangle_{\Rd} - \cL_x I(s,x)} \fu(s, x) 
	\bigr] \,ds
	\end{split}
	\]
	$\P$-a.s. for all $t \in [0,T]$ and $x \in \R^d$.
	Combining this and \eqref{eq:mathfrak-B} with \eqref{nablaI} shows that 
		\be \label{uh}
	\fu(t,x) 
	=
	\varphi(x) + \int_0^t \bigl[ \mathfrak{v}(s,x) 
	- \crl{ \tfrac12 \lVert h(x) \rVert_{\R^k}^2 + \operatorname{div} (\mu)(x) } \fu(s, x) \bigr] \,ds \quad \mbox{$\P$-a.s.}
	\ee
	for all $t \in [0,T]$ and $x \in \R^d$.
	Next, note that since $I(t, x) = \langle h(x), Z_t \rangle_{\R^k}$, $t \in [0,T]$, $x \in \R^d$,
	we obtain from It\^o's formula that
		\be \label{eI}
	e^{I(t,x)}
	=
	1
	+
	\int_0^t
	e^{I(s,x)}\Bigl(\langle h(x), dZ_s\rangle_{\R^k}
	+ \tfrac{1}{2}\lVert h(x) \rVert_{\R^k}^2 \Bigr) ds \quad
	\mbox{$\P$-a.s for all $t \in [0,T]$ and $x \in \R^d$.}
	\ee
Since $X_t(x) = \fu(t, x) e^{I(t, x)}$ for all $t \in [0,T]$ and $x \in \R^d$, we obtain 
from \eqref{eI}--\eqref{uh} and It\^o's formula that
	\[	
	\begin{split}
	X_t( x ) 
	=
	\varphi( x ) 
	&+ 
\int_0^t e^{I(s, x)} \biggl[ \mathfrak{v}(s,x) -
\Bigl( \tfrac12 \lVert h(x) \rVert_{\R^k}^2 +
\operatorname{div} (\mu)(x) \Bigr) \fu(s, x) \biggr]\,ds
	\\& 
	+ \int_0^t \fu(s, x) \, e^{I(s,x)}\Big(\langle h(x), dZ_s\rangle_{\R^k}
	+ \tfrac{1}{2} \lVert h(x) \rVert_{\R^k}^2 \,ds\Big) \quad
	\mbox{$\P$-a.s.}
	\end{split} 
	\]
	for all $t \in [0,T]$ and $x \in \R^d$. Moreover, it follows from
	\eqref{def:A-operator} that
	\[
		X_t(x) 
	= 
	\varphi(x) 
	+ \int_0^t \bigl[ \cL X_s(x) - X_s(x) \operatorname{div} (\mu)(x) \bigr] 
	+ \int_0^t X_s(x) \langle h(x), dZ_s \rangle_{\R^k} \quad
	\mbox{$\P$-a.s.}
	\]
	for all $t \in [0,T]$ and $x \in \R^d$.
	This, \eqref{it:Zakai-2}, the fact that 	
	\[
	\operatorname{div} (\mu X_s )(x) 
	= \langle \mu(x), \nabla X_s(x) \rangle_{\R^d} 
	+ X_s(x) \operatorname{div} (\mu)(x) \quad \mbox{for all $s \in [0, T]$ and $x \in \R^d$,}
	\]
	\eqref{eq:YZ} and \eqref{eq:def-L-operator} show that 
		\[
	\begin{split}
	X_t(x) 
	&= \varphi(x)
	+
	\int_0^t \Bigl[\tfrac12 \operatorname{Trace}_{\R^d}
	\bigl( \sigma \sigma^T \operatorname{Hess} (X_s)(x) \bigr) - \operatorname{div} (\mu X_s)(x) \Bigr] \,ds 
	+ \int_0^t X_s( x ) \langle h(x), dZ_s\rangle_{\R^d}
	\end{split}
	\]
	$\P$-a.s. for all $t \in [0, T]$ and $x \in \R^d$, which shows that $X$ satisfies
	\eqref{it:Zakai-3}.
	
	To show uniqueness, let us assume that $\tilde{X} \colon [0,T] \times \R^d \times \Omega \to \R$ 
	is another random field satisfying \eqref{it:Zakai-1}--\eqref{it:Zakai-3} of 
	Theorem \ref{thm:main}. Then one obtains from 
	the preceding arguments in the reverse order that for $\P$-a.a. $\omega \in \Omega$, 
	the function $(t,x) \mapsto \tilde{\fu}(t,x,\omega) = \tilde{X}_t(x, \omega) e^{-I(t,x, \omega)}$ 
	belongs to $C^{0,2}([0,T]\times \R^d,\R)$ and satisfies \eqref{mathfrak-u-D}. 
	Therefore, since by assumption, $h$ has
	bounded derivatives, there exists a subset $\tilde{\Omega} \subseteq \Omega$ with $\P[\tilde{\Omega}] = 1$ such that for all $\omega \in \tilde{\Omega}$, 
$\tilde{\fu}$ satisfies \eqref{mathfrak-u-C}--\eqref{mathfrak-u-D} and there exist 
constants $a(\omega), c(\omega) \ge 0$ such that 
\be \label{tildeubound}
\sup_{t \in [0,T]} |\tilde{\fu}(t,x,\omega)| \le a(\omega) e^{c(\omega) \|x \|_{\R^d}}
	\quad \mbox{for all } x \in \R^d.
\ee
Now, fix an $\omega \in \tilde{\Omega}$ and set $z = Z(\omega)$.
Then, the function $(t,x) \mapsto \tilde{u}(t,x) = \tilde{\fu}(t,x,\omega)$
belongs to $C^{1,2}([0,T]\times \R^d,\R)$ and satisfies
		\be \label{tildeu}
		\begin{aligned}
		& \tilde{u}(t,x)		=
		\varphi(x)
		+ 
		\int_0^t
		\Bigl[ 
		\tfrac{1}{2} \operatorname{Trace}_{\R^d} \bigl( \sigma \sigma^{T} 
		\operatorname{Hess}_x (\tilde{u})(s, x) \bigr)\\
		& +
		\left\langle b_z(s,x), \nabla_x \tilde{u}(s,x)\right\rangle_{\Rd}
		+ B_z(s,x) \tilde{u}(s,x)
		\Bigr] \,ds \quad \mbox{$\P$-a.s.}
		\end{aligned}
		\ee
		for all $t \in [0,T]$ and $x\in \R^d$, where 
		$b_z \colon [0,T]\times \R^d \to \R^d$ is given by \eqref{eq:def:fb} above.
By \eqref{tildeu}, we obtain from It\^{o}'s formula that
		\be \label{tildeuIto} \begin{aligned}
		& d \edg{\tilde{u} \brak{t-s, R^{z,t,x}_s} \exp \brak{\int_0^s B_{z}(t-r, R^{z,t,x}_r) dr} }
	\\ & \; \; =
		\left\langle b_z(t-s, R^{z,t,x}_s), \nabla_x \tilde{u}(t-s, R^{z,t,x}_s)\right\rangle_{\Rd} \exp \brak{\int_0^s B_{z}(t-r, R^{z,t,x}_r) dr} 
		\sigma dU_s
		\end{aligned} 
		\ee
for all $t \in [0,T]$, $s \in [0,t]$ and $x \in \R^d$. Using the stopping times
\[
\tau_n = \inf \{s \in [0,t] : \|R^{z,t,x}_s \|_{\R^d} \ge n\}, \quad n \in \N,
\]
one obtains from \eqref{tildeuIto} that
\be \label{tildeuE}
\begin{aligned}
& \tilde{u}(t,x) =  \E \edg{\varphi(R^{z,t,x}_{t} )
\exp \brak{\int_0^t B_z(t-s, R^{z,t,x}_s) ds} \mathbbm{1}_{\crl{\tau_n \ge t}}}\\
& + \E \edg{\tilde{u} \brak{t-\tau_n, R^{z,t,x}_{\tau_n} }
\exp \brak{\int_0^{\tau_n} B_z(t-s, R^{z,t,x}_s) ds} \mathbbm{1}_{\crl{\tau_n < t}}}
\quad \mbox{for all } n \in \N.
\end{aligned}
\ee
We know from \cref{le:sde-solution-exponential-integrability}.\eqref{it:sde-solution-exponential-integrability:item1} that
\be \label{Cp}
c_p := \E \edg{\sup_{s \in [0,t]} \exp \brak{p \| R^{v,t,x}_s \|_{\R^d}}} < \infty
\quad \mbox{for all } p \in (0,\infty).
\ee
Moreover, it follows from \eqref{tildeubound} and \eqref{it:B} that there exist 
constants $a,c \ge 0$ such that 
\be \label{boundu}
\sup_{s \in [0,t]} \abs{\tilde{u} \brak{t-s, f(s)}}
\exp \brak{\int_0^s B_z(t-r, f(r)) dr} \le
\sup_{s \in [0,t]} a e^{c \|f(s) \|_{\R^d}}
\ee
for every continuous function $f \colon [0,T] \to \R^d$, which together with 
\cref{le:sde-solution-exponential-integrability}.\eqref{it:sde-solution-exponential-integrability:item1} 
and Lebesgue's dominated convergence theorem implies 
that the first expectation in \eqref{tildeuE} converges to 
\[
 \E \edg{\varphi(R^{v,t,x}_{t} ) \exp \brak{\int_0^t B_z(t-s, R^{v,t,x}_s) ds}}
 \quad \mbox{for } n \to \infty.
\]
In addition, it can be seen from \eqref{boundu}
that the second expectation in \eqref{tildeuE} is bounded by 
$a e^{c n} \P[\tau_n < t]$. But, by \eqref{Cp}, one has
\[
\P[\tau_n < t] \le \P \edg{\sup_{s \in [0,t]} \|R^{v,t,x}_s \|_{\R^d} \ge n}
\le c_p e^{-pn} \quad \mbox{for all } p \in (0,1).
\]
So for $p > c$, the second expectation in \eqref{tildeuE} is bounded by 
$a c_p e^{(c-p)n}$, which converges to $0$ for $n \to \infty$. This shows that
\[
\tilde{\fu}(t,x,\omega) = u_{Z(\omega)}(t,x) \quad \mbox{and therefore,}
\quad \tilde{X}_t(x, \omega) = X_t(x, \omega)
\]
for all $t \in [0,T]$, $x \in \R^d$ and $\P$-a.a. $\omega \in \Omega$, which completes 
the proof of \cref{thm:main}.
\hfill \qed

\begin{lemma}\label{le:L-operator}
	Let $d\in \N$, and consider the mapping 
	$\cL \colon C^2(\R^d, \R) \to C(\R^d, \R)$ given by 
	\begin{equation}
	\label{def-L-operator-le}
	\cL g(x)
	= \tfrac{1}{2} \operatorname{Trace}_{\R^d} \bigl( a \operatorname{Hess} (g)( x ) \bigr)
	- \left \langle \nu( x ), \nabla g(x) \right\rangle_{ \R^d }, \quad 
	g \in C^2(\R^d, \R), \, x \in \R^d,
	\end{equation}
	where $a$ is a symmetric $d \times d$-matrix and
	$\nu$ a function in $C^1(\R^d, \R^d)$. 
	Then 
	\[
	\label{le:eq:L-operator}
	\begin{split}
	e^{-f(x)} \cL(e^f g)(x)
	= \cL g( x ) 
	+ \left\langle a \nabla f( x ), \nabla g( x )\right\rangle_{\Rd}
	+ \edg{ \tfrac{1}{2}\left \langle a \nabla f( x ), \nabla f( x )\right \rangle_{\Rd} 
	+ \cL f( x ) } g( x )
	\end{split}
	\]
	for all $f,g \in C^2(\R^d, \R)$ and $x \in \R^d$.
\end{lemma}

\begin{proof}[Proof]
If $f ,g \in C^2(\R^d, \R)$, one obtains from the product and chain rule that
	\be \label{eq:le-L-1-der}
	\begin{split}
	\nabla(e^f g)( x ) 
	&= \nabla e^f(x) g(x) + e^{f(x)} \nabla g(x) 
	= e^{f(x)} \nabla f(x) g(x) + e^{f(x)} \nabla g(x) 
	\\
	& =
	e^{f( x )} 
	\bigl[\nabla f( x ) g(x) + \nabla g( x ) \bigr] \quad \mbox{for all } x \in \R^d.
	\end{split}
	\ee
	Another application of the product and chain rule gives
	\[
	\label{eq:le-L-2-der}
	\begin{split}	
	\tfrac{\partial^2}{\partial x_i \partial x_j}( e^f g)( x ) 
	&= e^{f(x)} \tfrac{\partial}{\partial x_i} f( x ) \left[\tfrac{\partial}{\partial x_j} f(x) \, g(x) 
	+
	\tfrac{\partial}{\partial x_j} g(x) \right] 
	\\
	& + 
	e^{f(x)} \left[ \tfrac{\partial^2}{\partial x_i \partial x_j} f(x) \, g(x) 
	+ \tfrac{\partial}{\partial x_j} f(x) \tfrac{\partial}{\partial x_i} g( x ) 
	+ \tfrac{\partial^2}{\partial x_i \partial x_j} g( x ) \right]
	\end{split}
	\]
	for all $i, j \in \{1, 2, \ldots, d\}$ and $x\in\Rd$, which implies
	\begin{align*}
	&e^{-f(x)} \operatorname{Hess} (e^f g)(x) 
	\\ \nonumber
	&= 
	\operatorname{Hess} (g)( x ) 
	+ \nabla f(x) (\nabla g(x))^T
	+ \nabla g(x) (\nabla f(x))^T
	+ g(x) \bigl[ \nabla f(x) (\nabla f(x))^T + \operatorname{Hess}(f)( x ) \bigr]
	\end{align*}
	for all $x \in \R^d$. Since $a$ is symmetric, one has
	$ \operatorname{Trace}_{\R^d}( a p q^T ) = \langle a p, q \rangle_{\R^d} 
	= \langle a q, p \rangle_{\R^d} = \operatorname{Trace}_{\R^d}(a q p^T)$ 
	vor all $p,q \in \R^d$. Therefore, 
one obtains from \eqref{def-L-operator-le} and \eqref{eq:le-L-1-der} that
	\[
	\begin{split}
	&e^{-f(x)} \cL(e^f g)(x)
	= \tfrac{1}{2}\operatorname{Trace}_{\R^d}\bigl( a e^{-f(x)}\operatorname{Hess}(e^f g)(x) \bigr) 
	- \langle \nu(x), e^{-f(x)} \nabla(e^f g)( x ) \rangle_{ \R^d } 
	\\
	&
	= \tfrac12 \operatorname{Trace}_{\R^d} \bigl( 
	a \operatorname{Hess} (g)( x ) 
	+ a \nabla f(x) (\nabla g(x))^T
	+ a \nabla g(x) (\nabla f(x))^T
	\\
	& \quad 
	+ g( x ) a \nabla f(x) (\nabla f(x))^T
	+ g(x) a \operatorname{Hess} (f)( x ) \bigr) 
	- \left \langle \nu( x ), g(x) \nabla f( x ) + \nabla g( x ) \right\rangle_{ \R^d }
	\\
	&= \tfrac{1}{2} \operatorname{Trace}_{\R^d}\bigl( a \operatorname{Hess} (g)(x) \bigr) 
	- \langle \nu(x), \nabla g(x) \rangle_{ \R^d } 
	\\&\quad
	+
	\langle a \nabla f(x), \nabla g(x) \rangle_{ \R^d }
	+ 
	\tfrac12 \langle a \nabla f(x), \nabla f(x) \rangle_{ \R^d } g( x )
	\\
	& \quad + 
	\tfrac12 \operatorname{Trace}_{\R^d}\bigl( 
	a \operatorname{Hess} (f)(x)
	\bigr) g(x) - \langle \nu(x), \nabla f(x) \rangle_{\R^d} g(x)
	\\
	&=\cL g( x )	
	+ \left \langle a \nabla f( x ), \nabla g( x ) \right\rangle_{\Rd}
	+ \tfrac{1}{2} \left\langle a \nabla f( x ), \nabla f( x ) \right \rangle_{\Rd} g( x )
	+ \cL f( x ) \, g( x ),
	\end{split}
	\]
	for all $x \in \R^d$, which proves the lemma.
\end{proof}
\end{appendix}

\bibliographystyle{acm}

\end{document}